\pdfoutput=1
\documentclass[11pt]{article}

\usepackage{mathrsfs}

\usepackage[margin=1in]{geometry}
\usepackage[utf8]{inputenc}
\usepackage{enumerate}
\usepackage{amsfonts}
\usepackage{amsmath}
\usepackage{amssymb}
\usepackage{hyperref}
\usepackage{bbm}
\usepackage{amsthm}
\usepackage{xcolor}

\usepackage{colonequals}
\usepackage{theoremref}
\usepackage{enumitem}
\numberwithin{equation}{section}
\usepackage{appendix}

\newcommand{\sig}{\sigma}

\newcommand{\ch}{\cosh}
\newcommand{\hth}{\tanh}
\newcommand{\e}{\mathbb E}

\newcommand{\Prob}{\mathbb P}
\newcommand{\p}{\mathbb P}
\newcommand{\1}{\mathbbm 1}
\newcommand{\eps}{\varepsilon}

\newcommand{\la}{\langle}
\newcommand{\ra}{\rangle}

\newcommand{\calA}{\mathcal{A}}
\newcommand{\bsig}{\boldsymbol{\sigma}}
\newcommand{\VB}{\mbox{\tiny PVB}}
\newcommand{\B}{\mbox{\tiny VB}}
\newcommand{\calB}{\mathcal{B}}

\theoremstyle{definition}

\newtheorem{theorem}{Theorem}[section]
\newtheorem{corollary}[theorem]{Corollary}
\newtheorem{proposition}[theorem]{Proposition}
\newtheorem{lemma}[theorem]{Lemma}
\newtheorem{remark}[theorem]{Remark}

\newtheorem{definition}[theorem]{Definition}

\begin{document}

\title{Some Rigorous Results on the L\'evy Spin Glass Model}	

 \author{Wei-Kuo Chen\thanks{Email: wkchen@umn.edu. Partly supported by NSF grants DMS-1752184 and DMS-2246715 and Simons Foundation grant 1027727 } \and  Heejune Kim \thanks{ Email: kim01154@umn.edu. Partly supported by NSF grants DMS-1752184 and DMS-2246715}  \and Arnab Sen \thanks{Email: arnab@umn.edu. Partly supported by Simons Foundation grant MP-TSM-00002716} }

\maketitle 
\vspace{-3em}
\begin{center}
University of Minnesota  
\end{center}

\medskip

\begin{abstract}
We study the L\'evy spin glass model, a fully connected model on $N$ vertices with heavy-tailed interactions governed by a power law distribution of order $0<\alpha<2.$ Our investigation is divided into three cases $0<\alpha<1$, $\alpha=1$, and $1<\alpha<2.$ When $1<\alpha<2,$ we identify a high temperature regime, in which the limit and fluctuation of the free energy are explicitly obtained and the site and bond overlaps are shown to exhibit concentration, interestingly, while the former is concentrated around zero, the latter obeys a positivity behavior. At any temperature, we further establish the existence of the limiting free energy and derive a variational formula analogous to Panchenko's framework in the setting of the Poissonian Viana-Bray model. 
For $\alpha=1$, the free energy scales super-linearly and converges to a constant proportional to $\beta$ in probability at any temperature.
In the case of $0<\alpha<1$, the scaling for the free energy is again super-linear, however, it converges weakly to the sum of a Poisson Point Process at any temperature. 
Additionally, we show that the Gibbs measure puts most of its mass on the configurations that align with signs of the polynomially many heaviest edge weights. 


\end{abstract}

	\tableofcontents
	
\section{Introduction}

The study of spin glasses has drawn attention from both physicists and mathematicians for decades, in particular, a lot of efforts have been made to understand the  Sherrington-Kirkpatrick mean-field model \cite{Sherrington1975} as well as the Edwards-Anderson short-range model \cite{Edwards1975}. Whereas the investigation on the former has achieved tremendous success (see, e.g., \cite{MPV,panchenkobook,Talagrand2013vol1, Talagrand2013vol2}), the latter was scarcely understood due to its rigid geometry until recent years, see \cite{Chatterjee} and the references therein. In the hope of bridging the gap between the mean-field model and realistic models of short-range interactions, diluted variants of the Sherrington-Kirkpatrick model have been introduced, such as the spin glass model on a Bethe lattice, the Viana-Bray model, and the  diluted $p$-spin models (see \cite{Kanter1987,bethelattice,Viana1985}) that possess a bounded average number of spin interactions at any site as opposed to the diverging connectivity in the fully-connected models. For recent progress on the diluted and related models, see \cite{Achlioptas2005, Alberici,BCS,Coja-Oghlan2018, Coja-Oghlan2019,Ding2016,Sly2022,Franz, Guerra,bethelattice, talagrand2004, Starr}. 

In the present paper, we aim to consider the L\'evy spin glass proposed by Cizeau-Bouchaud \cite{Cizeau1993}, which is a Sherrington-Kirkpatrick-type model, however, the spin interactions are formulated through a heavy-tailed distribution with a power law density $2^{-1}\alpha {|x|}^{-(1+\alpha)}$ for $|x| \geq 1$ with stable exponent $0
<\alpha<2.$  A major motivation for considering the heavy-tailed interactions stems from its relevance in the study of the Ruderman–Kittel–Kasuya–Yosida interaction in the experimental metallic spin glass \cite{Klein}, see the elaboration in \cite{Cizeau1993}. On the other hand, it is also an intermediate model that bridges the mean-field and the diluted models, since even though the L\'evy spin glass is a fully connected model, the strong bonds have a diluted structure. 
In physics literature, the investigation of the L\'evy model started from \cite{Cizeau1993} by means of the cavity method, and
subsequently, several papers (\cite{Mezard, Hartmann, Metz})  revisited this model with different tools, such as the replica method. Notably, \cite{Mezard} introduced one way to understand the L\'evy model through the Viana-Bray model by adopting a cutoff to divide the interactions into two groups of strong bonds and weak bonds and then sending this cutoff to zero to restore the original problem. 

Besides spin glass, symmetric matrices with heavy-tailed distribution have also been  studied extensively both in physics and mathematics, see \cite{Auffinger2009, BenArous2008, Cizeau1994, tarquini2016level} and references therein. After a series of impressive works \cite{aggarwal2022mobility, aggarwal2018goe, bordenave2013localization, bordenave2017delocalization}, it is now known that for $1< \alpha < 2$, the eigenvectors are completely delocalized. For $0< \alpha<1$, the model exhibits a phase transition similar to Anderson localization, i.e., the eigenvector
transitions from being completely delocalized to being localized as its associated eigenvalue crosses some threshold (mobility edge). 

In contrast, the mathematical investigation of the L\'evy model remains very scarce. The key  difficulty is attributed to the fact that the power-law distribution is not additive.
As a result, many well-known techniques and tools such as the smart path method and integration by parts are no longer available in this model. To the best of our knowledge, the only rigorous result in the L\'evy model is the existence of the limiting free energy established recently in \cite{Jagannath} for $1<\alpha<2$ and a special choice of the heavy-tailed distribution.

Our study for the L\'evy model considers the entire domain $0<\alpha<2$ and general distributions.
In the case of $1<\alpha<2,$ we identify a high temperature regime through a critical temperature $\beta_\alpha$ by the second moment method and show that in this regime, the free energy and its fluctuation can be explicitly computed. More importantly, we show that the site and bond overlaps are concentrated for different structural reasons - while the individual spins do not have favored directions, the heavy weights on the edges do influence the behavior of the spin interactions as a whole. As a result, the site overlap is shown to be concentrated around zero; on the contrary, the bond overlap is concentrated around a positive quantity. Another main theme of this paper carries out the cutoff framework proposed in \cite{Mezard}, based on which we obtain quantitative controls between the L\'evy model and the Viana-Bray model as well as its Poissonian variant. Within this framework, we establish the superadditivity for the free energy and conclude the existence of the limiting free energy as in \cite{Jagannath}. Furthermore, we show that the limiting free energy admits a variational representation analogous to the formulation in Panchenko's treatment for the Poissonian Viana-Bray model, see~\cite{panchenko2013}.

Interestingly,  the behavior of the L\'evy model is significantly different when $\alpha=1$ and $0<\alpha < 1.$ In both cases, the free energy scales super-linearly, instead of $N$ as in the case of $1<\alpha<2$. For $\alpha=1$, the scaling is slower than $N^{1+\eps}$ for any $\eps>0$, and at any inverse temperature $\beta$, the normalized free energy concentrates around $\beta$ in probability. For $0<\alpha<1$, the normalization in the free energy is roughly like $N^{1/\alpha}$. We show that at any temperature, the normalized free energy always converges weakly to the sum of a Poisson Point Process parametrized by the heavy tailed distribution. Furthermore, we prove that the Gibbs measure carries most of its weights on the spin configurations, for which the signs of the spin interactions agree with those of $N^{\min(1/2,1-\alpha)-}$ many leading heavy weights.

We close this section by adding that there are many fundamentally important questions left unaddressed in the present work (see Section \ref{op}) - our article can be viewed just as a beginning step towards the rigorous understanding of the L\'evy model.

\section{Main Results and Open Problems} \label{Sec:results}

For $0<\alpha<2,$ let $  X$ be a symmetric random variable with tail probability
\begin{equation*}
    \p(|  X|>x)= \frac{L(x)}{x^\alpha},\ \forall x>0,
\end{equation*} where $L:(0,\infty)\to (0,\infty)$ is a slowly varying function at $\infty$, i.e., for any $t>0$, \[\lim_{x\to\infty} \frac{L(tx)}{L(x)}=1.\] 
For each $N\geq 1,$ define 
\begin{equation}\label{def:J}
    J= a_N^{-1} X,
\end{equation}
where $a_N :=\inf\bigl\{x>0: N\p(|  X|>x)\le 1 \bigr\}.$
 An important case of $X$ is the (symmetric) Pareto distribution that has the density
\begin{equation*}
    \frac{\alpha}{2} |x|^{-(\alpha+1)}\1_{|x|\ge 1}
\end{equation*}  
so that $a_N=N^{1/\alpha}$ and 
 the density of $J=N^{-1/\alpha}X$ is equal to
\begin{equation}
\frac{\alpha}{2N}\frac{\1_{\{|x|\geq N^{-1/\alpha}\}}}{|x|^{\alpha+1}}dx. \label{density of J}
\end{equation}
This is the disorder for the original L\'evy model considered by Cizeau-Bouchaud \cite{Cizeau1993}.

For a given (inverse) temperature $\beta>0$, the Hamiltonian of the L\'evy model is defined as
\begin{equation*}
	-H_N(\sigma)=\beta\sum_{1\leq i<j\leq N}J_{ij}\sigma_i\sigma_j
\end{equation*}
for $\sigma\in \{-1,1\}^N,$ where $(J_{ij})_{1\leq i<j\leq N}$ are i.i.d. copies of $J.$ Set the partition function as $$
Z_N=Z_N(\beta)=\sum_{\sigma\in \{-1,1\}^N}e^{-H_N(\sigma)}.
$$
Define the free energy by $\ln Z_N$ and the Gibbs measure by  $G_N(\sigma)=e^{-H_N(\sigma)}Z_N^{-1}$ for $\sigma\in\{-1,1\}^N.$
We denote by $\sigma,\sigma^1,\sigma^2,\ldots$ the i.i.d.\ samples (replicas) from $G_N$ and by $\la \cdot\ra$ the Gibbs expectation with respect to these random variables. There are two important parameters naturally associated with the L\'evy model. The first is the {site (multiple) overlap} defined as \[R_k=R(\sigma^1,\dots,\sigma^{k}):= \frac{1}{N}\sum_{i=1}^N \sigma^1_i \sigma^2_i\dots \sigma^{k}_i\]
for $k\geq 1,$ which measures the average similarities of the replicas over each spin site. Another parameter is the bond overlap that is concerned with the average similarities of two replicas on the  edges, whose absolute disorder strengths are at least $K$, namely, for $K\geq 0,$
\[Q_K=Q_K(\sigma^1,\sigma^2):= 
\begin{cases}
	\frac{1}{M}\sum_{i<j}\1_{\{|J_{ij}|\ge K\}} \sigma^1_i\sigma^1_j\sigma^2_i\sigma^2_j, & \text{if $M\ge 1$},\\
	0, & \text{if $M=0$}.
\end{cases}\]
Here, $M:=\sum_{i<j}\1_{\{|J_{ij}|\ge K\}}$ is 
 Binomial$(N(N-1)/2,\min(1,\p(|J|\ge K) ))$
, so $M$ is of order $N$ with high probability if $K>0$ and of order $N^2$ if $K=0$. We now state our main results in two major cases: $0<\alpha\le 1$ and $1<\alpha<2$.

\subsection{Case $1<\alpha<2$}

\subsubsection{High Temperature Behavior}

We proceed to state our main results on the high-temperature behavior of the L\'evy model. For $0<\alpha <2,$ define 
\begin{equation}\label{beta_critical}
\beta_{\alpha} \colonequals \Bigl(\alpha \int_{0}^{\infty} \frac{\hth^2(x)}{x^{\alpha +1}}dx\Bigr)^{-1/\alpha},
\end{equation}
where this integral is finite since $|\tanh(x)|\leq \min(|x|,1).$ We say that the model is in the high-temperature regime if $0<\beta<\beta_\alpha.$ For $1<\alpha<2,$ the following theorem establishes the limiting free energy inside this regime.

\begin{theorem}\label{freeenergy}
	Assume $1<\alpha <2$ and $0<\beta<\beta_{\alpha}$. For every $0<p<\alpha,$ we have
 \[ \lim_{N\to\infty}\e \Bigl|\frac{1}{N}\ln Z_N-\Bigl(\ln 2+\frac{\alpha}{2} \int_{0}^{\infty} \frac{\ln \ch (\beta x)}{x^{\alpha +1} }dx\Bigr)\Bigr|^p =0.\]
\end{theorem}

Following the convergence of the free energy, we establish the fluctuation of the free energy, which holds for all $0<\alpha<2$ and $0<\beta<\beta_\alpha.$ For notational brevity, define
\begin{equation}\label{def:bN}
    b_N =a_N^{-1}a_{{N\choose 2}}. 
\end{equation}
\begin{theorem}
	\label{fluctuationfreeenergy}
	Assume $0<\alpha<2$ and $0<\beta<\beta_\alpha.$ We have that
	\begin{equation*}
		\frac{1}{b_N}\Bigl(\ln Z_N-N\ln 2-{N\choose 2}\e \Bigl(\ln \ch (\beta J) \1_{\{\ln\ch(\beta J)\le \beta b_N\}}\Bigr)\Bigr)\stackrel{d}{\to}
		\beta Y_\alpha,
	\end{equation*}
  where the characteristic function of $Y_\alpha$ is given by \[t\mapsto\exp\Bigl( \int_0 ^\infty (e^{itx}-1-it x\1_{\{x\le 1\}})\alpha x^{-(\alpha+1)}dx\Bigr).\]
\end{theorem}

    
    For any $0< \alpha < 2$, the normalizing factor $b_N$ scales like $N^{1/\alpha}$ times some slowly varying function. When $1 < \alpha < 2$, the centering term is linear in $N$ and hence, is of order 
    $\gg b_N$. For $\alpha=1$, the centering term is again of order $\gg b_N$. However, when $0<\alpha<1$, the order of the centering term is also $b_N$, so it can be absorbed in the limiting distribution and this allows us to get a distributional convergence of $b_N^{-1}\ln Z_N$.
    However, as we see in Theorem \ref{thm:free_energy_below_one}, the same convergence continues to hold, with a more illuminating description, for any temperature.


\begin{remark}
$Y_\alpha$ is a stable distribution with stable index $\alpha$ and skewness parameter $\kappa = 1$ (see Equations (3.8.10) and (3.8.11) in \cite{durrett_2019}). As a result,  its right tail decays polynomially as $\mathbb{P}( Y_\alpha > x) \sim c x^{-\alpha}$ as $x \to \infty$ for some positive constant $c=c(\alpha)$.  When $0< \alpha < 1$, the infimum of the support is finite, while if $1 \le \alpha < 2$, the support of $Y_\alpha$ is the entire $\mathbb{R}$, in which case, the left tail decays faster than a power law. See, e.g., \cite[Theorem 1.2 and Proposition 3.1]{nolan2020univariate}.
\end{remark}

Next, we study the behavior of the site and bond overlaps.

\begin{theorem} \label{hightempoverlap}
	Assume $1<\alpha<2$. For any $k\geq 1,$ the site overlap satisfies
	\begin{equation}\label{siteoverlapzero}
		\lim_{N\to\infty}\e\bigl\la R_{2k}^2\bigr\ra=0
	\end{equation}
 for any $0<\beta<\beta_\alpha$
 and 
 \begin{equation}\label{eq: overlap and temperature}
            \lim_{\beta \to \infty} \liminf_{N \to \infty} \e \la R_{2k}^2 \ra =1.
    \end{equation}
Furthermore, for any $0<\beta<\beta_\alpha$ and $K>0,$ the bond overlap is concentrated at a positive number,
  \begin{equation}\label{bondoverlapnonzero}
		\lim_{N\to\infty}\e \Bigl\la \Bigl(Q_K - \alpha K^{\alpha} \int_{K}^\infty \frac{\hth^2(\beta x)}{x^{1+\alpha}}dx\Bigr)^2\Bigr\ra =0.
	\end{equation}
\end{theorem}

We comment that the behavior of the overlap in \eqref{siteoverlapzero} essentially says that the spins do not have favored directions on each site, whereas the positivity of the bond overlap in \eqref{bondoverlapnonzero} indicates that the spins on each edge tend to align in certain directions so that $\sigma_i\sigma_j=\mbox{sign}(J_{ij})$ as long as the absolute disorder exceeds the level $K.$ 
If we do not restrict ourselves on the heavy edges, i.e., set $K=0$ instead, then $Q_K=N(N-1)^{-1}R_2^2-(N-1)^{-1}$ and from \eqref{siteoverlapzero}, this quantity vanishes under $\e\la \cdot\ra$ as $N\to\infty.$ 

From \eqref{eq: overlap and temperature}, it seems reasonable to expect that the site overlap should exhibit a phase transition in temperature in the sense that there exists a critical temperature that distinguishes when the site overlap is a Dirac measure at zero or has a nontrivial (symmetric) distribution. See the second open problem in Section~\ref{op}.

\subsubsection{Superadditivity of the Free Energy}\label{sec:superadditivity}

Our next result establishes the superadditivity and convergence of the free energy in the L\'evy model with Pareto distributions and $1<\alpha<2.$ Note that the convergence in high temperature is already apparent as seen in Theorem \ref{freeenergy}.

\begin{theorem}\label{superadditivity}
Let $1<\alpha<2$ and consider the Pareto distribution. For any $\beta>0,$ there exists a constant $C>0$ such that for any $M,N\geq 1$ with $N/2\leq M\leq 2N,$
	\begin{equation*}
	\e \ln Z_{M+N} \geq \e \ln Z_M+\e \ln Z_N -C\phi(M+N)
	\end{equation*}
	for $
	\phi(x)=x^{1-(2-\alpha)/(1+3\alpha)}.
	$ Furthermore, $\lim_{N\to\infty}N^{-1}\e \ln Z_N$ exists and is the limit of $N^{-1}\ln Z_N$ in any $L^p$ norm for $1\le p<\alpha$.
\end{theorem}


Our approach to this superadditivity relies on making a connection to a special case of the so-called Viana-Bray as well as its Poissonian generalization, which are diluted 2-spin models defined on the sparse Erd\H{o}s-R\'enyi and Poisson (multi)graph, respectively. For convenience, they are called the VB and PVB models for the rest of the paper. First of all, note that for any given $\eps>0,$ by splitting the disorder in the Hamiltonian $H_N$ as $J\1_{\{|J|< \eps\}}+J\1_{\{|J|\geq \eps\}}$ for each pair and dropping the first components, as we shall prove, the resulting free energy will not deviate away from the original one too much. Now, on the one hand, since in distribution $J\1_{\{|J|\geq \eps\}}=Bg_\eps$ for some independent $B\sim$Ber$(\eps^{-\alpha}/N)$ and $g_\eps\sim 2^{-1}\alpha\eps^\alpha |x|^{-(1+\alpha)}\1_{\{|x|\geq \eps\}}$ and $g_\eps$ has a finite first moment, the free energy corresponding to the truncated Hamiltonian indeed equals that of the VB model, whose Hamiltonian is defined as
\begin{equation*}
	-H_{N,\eps}^{\B}(\sigma)=\beta \sum_{1\leq i<j\leq N}B_{ij}g_{\eps,ij}\sigma_i\sigma_j,
\end{equation*}
where $(B_{ij},g_{\eps,ij})_{1\leq i<j\leq N}$ are i.i.d.\ copies of $(B,g_\eps).$ On the other hand, since each $B_{ij}$ is asymptotically equal to the Poisson random variable with mean $\eps^{-\alpha}/N$, we can further argue that the free energy of the VB model is essentially the same as that of the PVB model, whose Hamiltonian is defined as
\begin{equation}\label{pvb}
	-H_{N,\eps}^{\VB}(\sigma)=\beta \sum_{k=1}^{\pi(\eps^{-\alpha}N)}g_{\eps,k}\sigma_{I(1,k)}\sigma_{I(2,k)},
\end{equation}
where $\pi(\eps^{-\alpha}N)$ is Poisson with mean $\eps^{-\alpha}N$, $(g_{\eps,i})_{i\geq 1}$ are i.i.d. copies of $g_\eps,$ $(I(1,k),I(2,k))_{k\geq 1}$ are i.i.d. sampled uniformly from $\{(i,j):1\leq i<j\leq N\},$ and these are all independent of each other. These three approximations are obtained quantitatively in terms of $\eps$ and $N$, based on which we prove Theorem \ref{superadditivity} by using a key fact, the free energy of the PVB model is superadditive, that will be established in this paper by an adoption of the Franz-Leone type interpolation in \cite{Franz}.
The existence of the limiting free energy in the L\'evy model was also established in  a recent work \cite{Jagannath} relying again on the superadditivity argument, where the authors adopted the combinatorial interpolation strategy in \cite{BGT}, different from our present approach.
While no characterization of the limiting free energy was given in \cite{Jagannath}, one of the main contributions of our paper is to provide a variational formula for this limit in Theorem \ref{main:thm2} below.


\subsubsection{Invariance Principle and Variational Formula}\label{sec:inv+var}

In this subsection, we will establish an invariance principle and a variational representation for the limiting free energy in the case of the Pareto distribution.

Our formulation is closely connected to the framework in the PVB model established by Panchenko \cite{panchenko2013}. To begin with, denote by $\mathcal{S}$ the set of all measurable functions $\sigma:[-1,1]^4\to\{-1,1\}$. Recall the PVB model from \eqref{pvb}. Consider the perturbed PVB Hamiltonian defined in \eqref{perturbedVB} associated with the choice of the parameter \eqref{input} and denote by $(\sigma^l)_{l\geq 1}$ the replicas from the corresponding Gibbs measure. The Gibbs expectation with respect to these replicas is denoted by $\la \cdot\ra^{\VB}_\eps$. Observe that  for any $N\geq 1,$ the replicas $(\sigma^l)_{l\geq 1}$ enjoy two symmetries, in the sense that under $\e\la \cdot\ra^{\VB}_\eps$, their joint distribution is invariant under any finite permutations of the spin indices $i$ and the replica indices $l.$ If, along a subsequence as $N\to\infty$, their joint distribution converges to a certain infinite array $(s_i^l)_{i,l\geq 1}$, then the same symmetries will still be preserved, namely, $(s_i^l)_{i,l\geq 1}$ is invariant with respect to any finite row or column permutations. By the Aldous-Hoover representation \cite{Aldous1985,Hoover1982}, we can then express $$s_i^l=\sigma(w,u_l,v_i,x_{i,l})$$ for all $i,l\geq 1$ for some $\sigma\in \mathcal{S}$. Denote by $\mathcal{M}_\eps$ the collection of all $\sigma$ that arise from the Aldous-Hoover representation for any possible weak limits of the replicas. In \cite{panchenko2013}, it was shown that for any $\sigma\in \mathcal{M}_\eps$, the corresponding spin distributions $(s_i^l)_{i,l\geq 1}$ satisfy an invariant principle and the free energy can be expressed as a variational formula in terms of $\mathcal{M}_\eps$, see the detailed review in Section \ref{VBM}. 

Our main results below validate the analogous framework in the setting of the L\'evy model, where the key ingredient in our formulation is played by the set  $\mathcal{N}$ consisting of all $\sigma\in \mathcal{S}$ that can be obtained as the limit of some sequence, $(\sigma_m)_{m\geq 1}$, in the finite-dimensional distribution sense (see Definition \ref{def:fdd} and Remark \ref{welldefiniteness}), where $\sigma_m\in \mathcal M_{\eps_m}$ for all $m\geq 1$ and $\eps_m\downarrow 0$. Before stating our results, we introduce the following notation.
Consider the measure $\mu$ defined on $\mathbb{R}\setminus\{0\}$ with density
\begin{equation}\label{intensitymeasure}
\mu(dx)=\frac{\alpha dx}{2|x|^{1+\alpha}}.
\end{equation}
Let $\lambda>0.$ Denote by $(\xi_k)_{k\geq 1}$ and $(\xi_k')_{k\geq 1}$  Poisson Point Processes (PPP) with mean measures $\mu,$ and $\lambda\mu/2,$ respectively.  Assume that their absolute values are ordered in the decreasing manner. Denote by $(\xi_{i,k})_{k\geq 1}$ and  $(\xi_{i,k}')_{k\geq 1}$ independent copies of $(\xi_k)_{k\geq 1}$ and $(\xi_k')_{k\geq 1}$, respectively, for all $i\geq 1$. Let $\delta$ and $(\delta_i)_{i\geq 1}$ be i.i.d.\ Rademacher random variables. Let $w,u,(u_l)_{l\geq 1}$ be i.i.d. uniform on $[0,1].$ Denote by $v_I,\hat v_I,x_I,\hat x_I$ i.i.d. uniform random variables on $[0,1]$ for any $r\geq 1$ and $I=(i_1,\ldots,i_r)\in \mathbb{N}^r.$ Here we assume all these sources of randomness are independent of each other.  Denote by $\e_{u,\delta,x}$ the expectation with respect to $u,$ $\delta$, $\delta_i$ for all $i\geq 1$ and $x_I,\hat x_I$ for all $n\geq 1$ and $I\in \mathbb{N}^n.$ 
For $\sigma\in \mathcal{S}$ and $I\in \mathbb{N}^r,$ denote
\begin{equation}
	\label{add:sec1.1:eq1}
        s_I=\sigma(w,u,v_{I},x_I)\,\,\mbox{and}\,\,
		\hat s_I=\sigma(w,u,\hat v_l,\hat x_l).
\end{equation}
For each $\sigma\in \mathcal{S},$ we associate an infinite array $(s_i^l)_{i,l\geq 1}$ to $\sigma$ by letting $s_i^l=\sigma(w,u_l,v_i,x_{i,l})$ for all $ i,l\geq 1.$ 
Let $m,n,q\geq 1$ and $m\ge n.$ Consider $q$ replicas and for each $1\leq l\leq q,$ let $C_l\subseteq \{1,2,\ldots,m\}$ be the collection of spin coordinates corresponding to the $l$-th replica. We further divide $C_l$ into the cavity coordinates and non-cavity coordinates, respectively, as 
\begin{equation*}
	C_l ^1=C_l \cap \{1,\dots,n\}\,\,\mbox{and}\,\,
	C_l ^2 = C_l \cap \{n+1,\dots,m\}.
\end{equation*}
For each $\sigma\in \mathcal{S},$ we define two functionals
\begin{align*}
	\begin{split}
		{U}_{l,\lambda}(\sigma)&=\e_{u,\delta,x}\Bigl(\prod_{i \in C_l^1}\delta_i\Bigr)\Bigl(\prod_{i\leq n}\prod_{k\geq 1}\bigl(1+\tanh(\beta \xi_{i,k})s_{i,k}\delta_i\bigr)\Bigr)\\
		&\qquad\qquad\Bigl(\prod_{i\in C_l^2}s_i\Bigr)\Bigl( \prod_{k\geq 1}\bigl(1+\tanh(\beta \xi_{k}')\hat s_{1,k}\hat s_{2,k}\bigr)\Bigr)
	\end{split}
\end{align*}
and
\begin{equation*}	{V}_\lambda(\sigma)=\e_{u,\delta,x}\Bigl(\prod_{i\leq n}\prod_{k\geq 1}\bigl(1+\tanh(\beta \xi_{i,k})s_{i,k}\delta_i\bigr)\Bigr)\Bigl( \prod_{k\geq 1}\bigl(1+\tanh(\beta \xi_{k}')\hat s_{1,k}\hat s_{2,k}\bigr)\Bigr).
\end{equation*}
Note that the definitions of $U_{l,\lambda}(\sigma)$ and $V_\lambda(\sigma)$ involve infinite products. Their well-definedness is ensured by Theorem \ref{invariancestep1} below. The following theorem states that the spin distribution $(s_i^l)_{i,l\geq 1}$ associated with $\sigma\in \mathcal{N}$ satisfies a system of consistent equations.

\begin{theorem}\label{main:thm1}
Let $1<\alpha<2$ and consider the Pareto distribution. Any $\sigma\in \mathcal{N}$ satisfies the following equations that  for any $\lambda\geq 0$, $n,m,q\geq 1$, and the sets $C_l\subseteq\{1,\ldots,m\}$,
	\begin{equation}
	    \label{main:thm1:eq1}
		\e \prod_{l\leq q}\prod_{i\in C_l}s_i^l=\e\prod_{l\leq q}\e_{u,x}\prod_{i\in C_l}s_i=\e\frac{\prod_{l\leq q}U_{l,\lambda}(\sigma)}{V_\lambda(\sigma)^q}.
	\end{equation}
\end{theorem}

Let $\mathcal{N}_{\mathrm{inv}}$ be the collection of all $\sigma\in \mathcal{S}$ such that the invariant equations in Theorem  \ref{main:thm1} hold. Set the functional
\begin{equation*}
	Q(\sigma)=\ln 2+\e \ln \e_{u,x}\cosh\Bigl(\beta\sum_{k\geq 1}\xi_ks_k\Bigr)-\e\ln \e_{u,x}\exp\Bigl(\beta\sum_{k\geq 1}\xi_k's_{1,k}s_{2,k}\Bigr),
\end{equation*}
where in the second PPP $(\xi_k')_{k\geq 1},$ we fix $\lambda=1.$
Note that this functional is always finite as also guaranteed by Theorem \ref{invariancestep1} below. Our next result shows that  the limiting free energy can be expressed as a variational formula.

\begin{theorem}\label{main:thm2} Assume $1<\alpha<2.$ Let $\beta>0$ and $1\le p<\alpha.$ For the  Pareto distribution, the following convergence holds with respect to the $L^p$ norm,
	\begin{equation*}
		\lim_{N\to\infty}\frac{1}{N}\ln Z_N=\inf_{\sigma\in \mathcal{N}}Q(\sigma)=\inf_{\sigma\in \mathcal{N}_{\mathrm{inv}}}Q(\sigma).
	\end{equation*}
\end{theorem}

Later on, we will show that the limiting free energy does  not change if we replace the Pareto distribution with the general L\'evy distribution, see \eqref{thm:universality}. The proposition below carries out the structure of $\sigma\in \mathcal{N}$ at high temperature.

\begin{proposition}\label{hightemp}
	Assume $1<\alpha<2$ and the Pareto  distribution. Let $0<\beta<\beta_\alpha.$ For any $\sigma\in \mathcal{N}$, we have that $\e_x\sigma(w,u,v,x)=0$. 
\end{proposition}

As a consequence of this proposition, it can be checked directly that when $\beta<\beta_\alpha,$
$\inf_{\sigma\in \mathcal{N}}Q(\sigma)$ matches the limiting free energy obtained in Theorem \ref{freeenergy}.
We comment that while the definition of $\mathcal{N}$ involves $\mathcal{M}_\eps$, it seems more natural that both Theorems \ref{main:thm1} and \ref{main:thm2} should still hold with the replacement of $\mathcal{N}$ by the set $\mathcal{N}^{\circ}$ that gathers all $\sigma\in \mathcal{S}$ arising from the Aldous-Hoover representation for all possible weak limits of the replicas sampled from the L\'evy model with an added perturbative Hamiltonian term to the original system. Although this is exactly the case in the PVB model in \cite{panchenko2013}, the Poissonian perturbation approach therein does not seem to work in the L\'evy model since our heavy-tailed disorder is not additive. For this reason, we approach our theorems by analyzing the $\eps$-limit of Panchenko's invariance principle and the variational formula in the PVB model with truncation level $\eps$; handling such limit involves a heavy analysis, especially, in obtaining the uniform convergences and continuities of the related functionals.

\begin{remark}
    \rm In Theorem \ref{thm:universality} below, we will show that the expectation of the free energy for the general L\'evy model is asymptotically the same as that associated to the Pareto distribution. Thus, the existence of the limiting free energy in Theorem \ref{superadditivity} and the same expressions in Theorem \ref{main:thm2} hold for the general L\'evy model.
\end{remark}




\subsection{Case $0<\alpha \le 1$}

In this section, we investigate the behavior of the limiting free energy and the Gibbs measure in the L\'evy model assuming that $0<\alpha \le 1.$


In the case $\alpha=1$, upon suitable normalization, the limiting free energy is given by a simple constant proportional to $\beta$ and the formula holds for any temperature. Recall $b_N$ from \eqref{def:bN}. Let
\[h_N =  \int_0^{b_N}   \hth( y)N \p (| X|> ya_N)dy.\]
   
\begin{theorem}\label{thm: generalized: alpha=1: free energy}
     Assume $\alpha =1$.
     We have that $1 \ll h_N\ll N^{\eps}$ for any $\eps>0$. Moreover, for any $\beta>0$, \begin{equation*}
        \frac{1}{N h_N}\ln Z_N  \stackrel{p}{\to} \frac{\beta}{2}.
    \end{equation*}
    
\end{theorem}

Recall from Theorem \ref{freeenergy} that when $1<\alpha<2,$ the normalization for $\ln Z_N$ is $N$ and the free energy converges to a deterministic constant in any $L^p$ norm for $1\le p<\alpha$. In sharp contrast, we show that when $0<\alpha<1$,  the correct normalization for $\ln Z_N$ becomes $b_N$ which is roughly like $N^{1/\alpha}$ and the normalized free energy converges to the sum of a PPP in distribution. More precisely, let $E_1, E_2, \ldots $ be i.i.d.\ exponential random variables with rate one and let their cumulative sums be given by $\gamma_j  =  E_1+ E_2+ \cdots+ E_j$ for  $j \ge 1$. Note that $(\gamma_j)_{j\geq 1}$ forms a PPP with unit intensity. By the mapping theorem, the points \begin{equation}\label{PPP2} 
	\gamma_1^{-1/\alpha} >  \gamma_2^{-1/\alpha} > \cdots 
\end{equation} form a PPP with intensity $\nu(x) =  \alpha x^{ -(1+\alpha)}$ on $(0, \infty)$, ordered descendingly. The following is our main result.


\begin{theorem} \label{thm:free_energy_below_one}
	For $0 < \alpha < 1$ and $\beta > 0$, we have that as $N\to\infty,$

	\[ \frac{1}{b_N}\ln Z_N \stackrel{d}{\to} \beta  \sum_{j=1}^\infty \gamma_j^{-1/\alpha},\]
 where  the infinite sum on the right-hand side is finite a.s..
\end{theorem}

Next, we prove that a spin configuration drawn from the Gibbs measure at any temperature aligns itself with the signs of the interactions on the heaviest edges with high probability. In fact, the number of such heaviest edges can be taken to be growing polynomially with $N$. Denote $n =N(N-1)/2$ and let $1 \le  U(k, 1) <U(k, 2) \le N$ for $1 \le k \le n$ be the indices such that
\begin{equation}\label{heavytailed:add:eq1}
	| J_{U(1, 1), U(1, 2)}| >  | J_{U(2, 1), U(2, 2)}| > \cdots >  | J_{U(n, 1), U(n, 2)}|. 
\end{equation}
For $1 \le k \le n$, define the indicator variable
\[ \chi_k =  \chi_k (\sigma) =  \1_{\{  \mathrm{sgn}(J_{U(k, 1), U(k, 2)} \sigma_{U(k, 1)} \sigma_{U(k, 2)}) =1 \} }.  \]

\begin{theorem}\label{thm:Gibbs_structure}
	Let $0 < \alpha < 1$ and $\beta >0$. Set $\kappa  = \min( 1/2 , 1- \alpha) - \delta$ for some arbitrarily small constant $\delta>0$ and take $R = N^\kappa$.  
	Then, as $N \to \infty$, we have 
	\[ \mathbb{E}  \bigl\langle \1_{\{ \chi_k  = 1,  \ 1 \le k \le R \} } \bigr\rangle \to   1.\]
\end{theorem}

The exponent $\kappa$ in Theorem \ref{thm:Gibbs_structure} is not likely to be optimal.
 We add that from our analysis for Theorem~\ref{thm:free_energy_below_one}, one can obtain the limit of the ground state energy, i.e.,  $b_N^{-1}\min_{\sigma}\sum_{i<j}J_{ij}\sigma_i\sigma_j$, as well as construct near ground states (near-minimizers). We omit the details in this work.

\subsection{Open Problems}\label{op}

 There are still many open problems in the L\'evy model which have not been explored in this paper. Below, we highlight three major important directions.

\begin{enumerate}

    \item Following Theorem \ref{main:thm1}, one of the main remaining questions in the L\'evy model is to understand the so-called M\'ezard-Parisi ansatz for the limiting free energy for $1<\alpha<2$.  In the setting of the PVB model, if one assumes that the overlap is supported only on a finite number of atoms, then this ansatz was understood in the work \cite{Panchenko2016} and led to the M\'ezard-Parisi formula for the limiting free energy. However, since the overlap distribution is expected to be supported on an interval in the low temperature regime, the general scenario remains open and as a result, we cannot take this formula to understand the limiting free energy of the L\'evy model directly. 

 \item Understanding the behavior of the site and bond overlaps are important subjects in the L\'evy model. In particular, in the case $1<\alpha<2$, it is predicted in \cite{Cizeau1993} that there exists another critical temperature $\beta_c>\beta_\alpha$ such that for $\beta_\alpha<\beta<\beta_c,$ the site overlap is concentrated at a positive number, while for $\beta>\beta_c$, this quantity possesses a nontrivial distribution. For $0<\alpha\leq 1,$ we do not have any results on the overlaps. It would be particularly interesting to understand the phase transition of the overlaps with respect to the temperature. We anticipate that the bond overlap should be strictly positive at any temperature.


\end{enumerate}

\subsection{Structure of the Paper}

We organize the rest of the paper as follows. From Sections 3 to 9,
we shall present the proofs for all results stated in Section 2 by restricting ourselves to the Pareto disorder \eqref{density of J}, which already retains the key features of the general model while rendering the calculations throughout our proofs more tractable. 

We assume $1<\alpha<2$ throughout Sections 3 to 8. 
In Section~3, we establish the limiting free energy by utilizing the second moment method, which naturally gives rise to the critical temperature $\beta_\alpha$ in good agreement with the findings of the physics literature, e.g., \cite{Cizeau1993, Mezard, Hartmann}. 
In Section 4, we derive the fluctuation of the free energy $\ln Z_N$ by dividing this quantity into the scalar and effective components, $\ln \bar Z_N$ and $\ln \widehat{Z}_N$ (see \eqref{split}), and establishing their fluctuations individually. For the former, it is obtained through a truncation argument together with a standard central limiting theorem for the sum of i.i.d. heavy tailed distributions; for the latter, we employ the powerful cluster expansion technique introduced by Aizenman-Lebowitz-Ruelle \cite{aizenman}.
Section 5 contains the proof of the concentration of the overlaps in the high temperature regime.
In Section~6, superadditivity of the free energy is demonstrated in order to establish the existence of the limiting free energy in any temperature. The proof of the invariance principle and the variational formula stated in Section \ref{sec:inv+var} is presented in Section 8, which is based on Panchenko's analogous results in the setting of the PVB model \cite{panchenko2013} reviewed in Section 7.  Section 9 is devoted to the proof of our results for the cases, $\alpha=1$ and $0<\alpha<1.$ We begin by establishing the limiting free energy for $\alpha=1$ stated in Theorem \ref{thm: generalized: alpha=1: free energy} and then focus on the case $0<\alpha<1$, for which we prove the weak convergence of the limiting free energy and analyze the structure of the Gibbs measure by making use of a representation of the disorder matrix in terms of a PPP parametrized by the heavy-tailed distribution (see Lemma \ref{lem:heavy_tail_conv_representation}). 

Finally, the rest of the paper consists of four appendices. Appendices \ref{ProofStep1} and \ref{ProofStep2} gather the proofs for the technical components, Theorems \ref{invariancestep1} and \ref{invariancestep2}, stated  in Section 8. Recalling that between Sections 3 and 9, our proofs are limited to the Pareto disorder \eqref{density of J}, the proofs of our main results for the general L\'evy case are kept in Appendix \ref{sec:generalization} that is based on the exposition on the limiting expectation for the L\'evy distribution in Appendix \ref{sec:proof of expectationlemma}.

\section{Establishing the Limiting Free Energy at High Temperature}
\subsection{Concentration of the Free Energy}

Note for any $a=\pm 1$ and $x\in \mathbb{R},$ we have the identity,
\begin{equation}\label{identity}
	e^{ax}=\cosh (x)\bigl(1+\tanh(x)a\bigr).
\end{equation}
Using this, we can split the partition function into the scalar and effective parts as
\begin{equation}\label{split}
Z_N=\bar Z_N\cdot \widehat Z_N,    
\end{equation}
where
\[\bar Z_N \colonequals\prod_{i<j} \ch (\beta J_{ij} )\,\,\mbox{and}\,\, \widehat Z_N := \sum_{\sigma \in \{-1,1\}^N} \prod_{i<j}\big(1+ \sig_i\sig_j \hth (\beta J_{ij})\big). \]
We need the concentration of the free energy in the L\'evy model, which will be used to establish the limiting free energy at high temperature. Incidentally, the same concentration was recently established in  \cite[Proposition 5.1]{Jagannath}. For the sake of completeness, we present a proof below.

\begin{theorem}[Concentration]\label{con}
	Assume $1<\alpha <2$ and $\beta>0.$
	For any $0<\eps<\alpha-1$ and $(2+\eps)(1+\alpha)^{-1}\alpha<p<\alpha$,
	we have for large enough $N$,
	\begin{equation}
		\label{Z}	\e \Bigl|\frac{\ln Z_N}{N}-\e\frac{\ln  Z_N}{N}\Bigr|^{p} \le \frac{K}{N^{p+p/\alpha-2-\eps}} 
	\end{equation}
 and the same inequality also holds for $\widehat Z_N$ and $\bar Z_N$,
	where  $K>0$ is a constant depending only on $\alpha,p,\beta$.  
\end{theorem}

\begin{remark}
	\rm Since $\e |J|^p<\infty$ for every $p<\alpha$ and $\e |J|^\alpha =\infty$, \eqref{Z} cannot be extended to any $p\ge\alpha$.
\end{remark}

\begin{proof}[\bf Proof of Theorem \ref{con}]
    
We revisit the argument in \cite{Jagannath} and point out that the same steps lead to the concentration of  $\ln \widehat Z_N$.
Write $X_{ij}=N^{1/\alpha}J_{ij}$ so that $X_{ij}$ are i.i.d.\ with density $2^{-1}\alpha\1_{\{|x|\ge1\}}x^{-1-\alpha}$.  Let $n = N(N-1)/2$.
Let $e_1, e_2, \ldots, e_n$ be an arbitrary enumeration of the elements from $\mathcal I =  \{(i,j):1\le i<j\le N\}$.
For $e = (i, j)$, let $X_e$ and $\sigma_e$ denote $X_{ij}$ and $\sigma_i \sigma_j$ respectively. 
Define $\mathcal F_{k}$ to be the sigma-field generated by the random variables $\{ X_{e_\ell}: 1 \le \ell \le k \}$.
Consider the martingale difference sequence 
\[ D_k =  \frac{1}{N}\e \bigl[\ln Z_N \big|\mathcal F_k\bigr] - \frac{1}{N}\e \bigl[\ln Z_N \big|\mathcal F_{k-1}\bigr],  \ \ 1 \le k \le n,  \]
where the expectation  is well-defined since $\alpha >1$.
Writing $\la\cdot\ra_k$ for the Gibbs expectation for  the same Hamiltonian as $H_N$ with $X_{e_k} = 0$ and $Z_N^{(k)}$ for the corresponding partition function, we can express $Z_N  =  Z_N^{(k)}\la \exp(\beta N^{-1/\alpha} X_{e_k} \sigma_{e_k})\ra_k$. Since $\e [\ln Z^{(k)}|\mathcal F_k]= \e [\ln Z^{(k)}|\mathcal F_{k-1}]$,
we have 
\begin{equation}
    ND_k= \e \bigl[\ln \la \exp(\beta N^{-1/\alpha} X_{e_k} \sigma_{e_k})\ra_k\big|\mathcal F_{k}\bigr]- \e \bigl[\ln \la \exp(\beta N^{-1/\alpha} X_{e_k} \sigma_{e_k})\ra_k\big|\mathcal F_{k-1}\bigr], 
    \label{eq:D}
\end{equation}
which leads to the bound 
\begin{equation}\label{eq:D_bd}
|D_k |\le \beta N^{-1-1/\alpha}\bigl(|X_{e_k}|+\e|X_{e_k}|\bigr).
\end{equation}
Let $0< p<\alpha$.
By Burkholder-Davis-Gundy inequality, there exists a constant $C_p$ such that
\begin{equation}\label{eq:burk}
    \e \Bigl| \frac{1}{N} \ln Z_N  - \frac{1}{N}\e  \ln Z_N \Bigr|^p \le C_p \e \Bigl|\sum_{k} D_k^2\Bigr|^{p/2}\le C_p \sum_k \e |D_k|^p 
\end{equation} 
where the last inequality follows from the fact that $p/2 < 1$. Combining \eqref{eq:D_bd} and \eqref{eq:burk}, and 
using the fact that  $X_k$ is $L^p$-integrable, we derive that 
\begin{equation}
    \e \Bigl| \frac{1}{N}\ln Z_N  - \frac{1}{N} \e  \ln Z_N \Bigr|^p\le \frac{C \e|X|^p}{    N^{p-2+p/\alpha}} 
     \label{eq:D final bound}
\end{equation} 
for some constant $C  = C(p, \alpha, \beta) >0$. 
The proof is complete by noticing that $p-2+p/\alpha>0$ is equivalent to $p>2\alpha /(1+\alpha)$.

The proof for $\ln \bar Z_N$ is similar. In   this case, we take  our martingale difference to be 
$$ D_k = \frac{1}{N} \e \bigl[\ln \bar Z_N \big|\mathcal F_k\big] -\frac{1}{N}\e \bigl[\ln \bar Z_N \big|\mathcal F_{k-1}\bigr]=\frac{1}{N}\ln\cosh(\beta N^{-1/\alpha}X_{e_k}) $$
and note that $|D_k|\leq \beta N^{-1-1/\alpha}|X_{e_k}|.$
The rest of the proof remains the same. Finally, the concentration of $\ln \widehat Z_N$ follows from that of $\ln Z_N$ and $\ln \bar Z_N$ by using the identity $\ln \widehat Z_N= \ln  Z_N - \ln  \bar Z_N$ and the triangle inequality.  

\end{proof}



\subsection{Proof of Theorem \ref{freeenergy}}\label{subsec:add1}
 Note that as $N \to \infty$, 
\begin{equation*}
		\e\frac{\ln \bar Z_N}{N}=\frac{N-1}{2} \e \ln \ch (\beta J )\to \frac{\alpha}{2}\int_0^{\infty}\frac{\ln \ch (\beta x) }{x^{1+\alpha} } dx,
	\end{equation*}
 where the inequality $|\ln \ch (x)|\leq \min(|x|, x^2/2)$ for $x\in\mathbb R$ assures that the integral is finite in the range $1<\alpha<2$. From this and Theorem \ref{con},
 it suffices to show that $\lim_{N\to\infty}N^{-1}\e\ln \widehat Z_N=\ln 2.$ 
 
 We argue by using the second moment method. First of all, by independence and symmetry of the weights $(J_{ij})$, \[\e \widehat Z_N = \sum_\sigma \prod_{i<j} (1+\sig_i\sig_j \e \hth (\beta J_{ij})) =2^N,\]
Then Jensen's inequality implies that \[\frac{1}{N}\e \ln \widehat Z_N  \leq \frac{1}{N}\ln \e \widehat Z_N =\ln 2.\]
Thus, it remains  to show $
\liminf_{N\to\infty}N^{-1}\e \ln \widehat Z_N=\ln 2.
$
For any $N\geq 2$,
\begin{align*}
	\e \widehat Z_N^2 &=\sum_{\sigma,\tau} \prod_{i<j, \ k<l}\e\big(  1+\sig_i\sig_j\hth(\beta J_{ij})\bigr)\bigl(1+\tau_k\tau_l\hth(\beta J_{kl})\bigr)
	\\&=\sum_{\sigma,\tau} \prod_{i<j} \bigl(1+\sig_i\sig_j\tau_i\tau_j \e \hth^2(\beta J)\bigr).
\end{align*}
 Using the inequality $1+x \leq e^x$ for $x\in \mathbb R$, we obtain that
 \begin{align*}
	\e \widehat Z_N^2 &\leq \sum_{\sigma,\tau} \exp\Bigl(\sum_{i<j} \sig_i\sig_j\tau_i\tau_j \e \hth^2(\beta J)\Bigr)
	\\&\leq \sum_{\sigma,\tau} \exp \Big(\frac{1}{2}\Big(\sum_{i\leq N} \sig_i\tau_i\Big)^2 \e \hth^2(\beta J)\Big)
	\leq  4^N\big( 1-N\e \hth^2(\beta J) \big)^{-1/2}.
\end{align*}
 The last inequality above follows from the fact that $(\sig_i\tau_i)_{i\leq N}$ can be regarded as independent Rademacher random variables on which we can apply the standard exponential moment estimate $\e \exp(2^{-1}(\sum_{i\leq N} a_i\eta_i )^2) \leq (1-\sum_{i\leq N} a_i^2)^{-1/2}$ for i.i.d.\ Rademacher $(\eta_i)_{i \le n} $ and $\sum_{i\leq N}a_i^2 <1$. Additionally, we used the fact that, as $N\to\infty$,
 \begin{equation} \label{freeenergy:eq2}
N\e \hth^2(\beta J)= \alpha \int_{N^{-1/\alpha}}^\infty \frac{\hth^2(\beta x)}{x^{1+\alpha}}dx \uparrow (\beta/\beta_{\alpha})^\alpha<1.
 \end{equation} 
Now, let $A$ be the event $\{\widehat Z_N > 2^{-1} \e \widehat Z_N\}$.
The Paley-Zygmund inequality implies \[\Prob (A)\geq \frac{1}{4} \frac{(\e \widehat Z_N)^2}{\e\widehat Z_N^2}\geq \frac{1}{4}\sqrt{1-(\beta/\beta_{\alpha})^\alpha}.\]
On the event $A$, it holds that \[\frac{\ln \widehat Z_N}{N} > -\frac{\ln 2}{N}+\frac{\ln \e \widehat  Z_N}{N} = -\frac{\ln 2}{N}+\ln 2. \]
From the concentration of $\ln \widehat Z_N$ in Theorem \ref{con},  we have $N^{-1} \ln \widehat Z_N  \to  N^{-1} \e \ln \widehat Z_N $ in probability. Fix $\epsilon>0$. 
For sufficiently large $N$, on an event of constant probability, we have
\[\frac{\e\ln \widehat Z_N}{N} = \frac{\e\ln \widehat Z_N}{N} - \frac{\ln \widehat Z_N}{N} +\frac{\ln \widehat Z_N}{N} \geq -\epsilon - \frac{\ln 2}{N}+\ln 2.\]
Since the previous inequality is not random, we have $\liminf_{N\to\infty}N^{-1}\e \ln \widehat Z_N=\ln 2,$ as desired.

\section{Establishing the Fluctuation of the Free Energy}

We study the fluctuation of $\ln Z_N$ in Theorem \ref{fluctuationfreeenergy} based on the following fluctuation results of $\ln \bar Z_N$ and  $\ln \widehat Z_N$. First, we achieve the fluctuation of $\ln \bar Z_N$ at any temperature.
\begin{theorem} \label{thm: generalized: fluctuation of Zbar}
	For $0<\alpha <2$ and $\beta>0$, we have \[\frac{1}{a_N^{-1}a_{{N\choose 2}}}\Bigl(\ln \bar Z_N -{N\choose 2}\e \Bigl(\ln \ch (\beta J) \1_{\{a_N\ln\ch(\beta J)\le \beta a_{N\choose 2}\}}\Bigr)\Bigr)\stackrel{d}{\to}\beta Y_\alpha,\] 
	as $N \to \infty$, where the characteristic function of $Y_\alpha$ is given by
    \begin{equation} 
    \label{eq:Y_a:charac}
	\log \e e^{it Y_\alpha}  =  \int_0 ^\infty (e^{itx}-1-itx \1_{\{ x \le 1\}})\alpha x^{-(\alpha+1)}dx. 
	\end{equation}
\end{theorem}


Next, we write down the fluctuation of $\ln \widehat Z_N$, which we obtain only at high temperature. 

\begin{theorem} \label{fluctuation hat Z}
	For $0<\alpha<2$ and $\beta<\beta_{\alpha}$, as $N \to \infty$, we have 
	 \[
	2^{-N} \widehat Z_N \stackrel{d}{\to} X_{\alpha, \beta },\] where the characteristic function of $\ln X_{\alpha, \beta}$ is given by \[t\mapsto \prod_{k=3}^\infty \exp\Bigl(\frac{\alpha ^k}{k 2^{k+1}} \int_{\mathbb R^k} \Bigl(e^{it \ln \bigl(1+\prod_{i\leq k}\hth(\beta x_i)\bigr)} -1- it\prod_{i\leq k}\hth(\beta x_i) \Bigr) \frac{1}{\prod_{i\leq k} |x_i|^{1+\alpha}}dx_1\dots dx_k\Bigr).\]
\end{theorem}

\begin{remark}
The limit $X_{\alpha, \beta}$ is distributed as $\exp(T_3+T_4+ \cdots),$ where the random variables $T_3, T_4, \ldots, $ are independent with the following distributions. For $k \ge 3$, 
let $(\xi_j)_{j \ge 1}$ be the PPP on $(\mathbb{R}\setminus \{0\})^k$  with intensity measure $\mu^{\otimes k}$ for $\mu$ defined in \eqref{intensitymeasure}. Let $ (\xi_j^1, \ldots, \xi_j^k)$ represent the coordinates of the vector $\xi_j$.  Then 
\[ T_k =  \frac{1}{2k} \sum_{ j \ge 1 }\ln \Bigl(1+\prod_{i\leq k}\hth(\beta \xi_j^i )\Bigr), \]
where the above infinite sum is interpreted as the weak limit, as $\eps \to 0$, of the finite sum over $j$ such that $|\xi_j^i| > \eps$ for each $ 1 \le i \le k$.
\end{remark}



\begin{proof}[\bf Proof of Theorem \ref{fluctuationfreeenergy} ] 
From Theorem \ref{fluctuation hat Z}, we  see that 
\[
    \frac{1}{N^{1/\alpha}}\bigl( \ln \widehat Z_N- N\ln2\bigr)\stackrel{d}{\to}0. 
\]
Since $\ln Z_N=\ln \bar Z_N +\ln \widehat Z_N$, our assertion follows by using the Slutsky's theorem and Theorem \ref{thm: generalized: fluctuation of Zbar}.
\end{proof}

\subsection{Proof of Theorem~\ref{thm: generalized: fluctuation of Zbar}} 
In the Pareto case, by \eqref{density of J}, we can recast the statement of the theorem into \begin{equation*}
    \frac{1}{N^{1/\alpha}}\Bigl(\ln \bar Z_N - m \alpha \int_{N^{-1/\alpha}}^{\beta^{-1}g(\beta m^{1/\alpha})} \frac{\ln \ch ( \beta x)}{x^{1+\alpha}}dx\Bigr) \stackrel{d}{\to} \frac{\beta Y_\alpha}{2^{1/\alpha}},
\end{equation*} where $m=(N-1)/2$ and $g$ is the inverse function of $x\mapsto \ln\ch(x)$.
Using \eqref{eq: centering small} and the fact $\lim_{x\to\infty} (g(x) -x) =\ln 2$, this is equivalent to
\begin{equation}\label{eq: centering}
    \frac{1}{N^{1/\alpha}}\Bigl(\ln \bar Z_N - \frac{\alpha N}{2}\int_{N^{-1/\alpha}}^{(\frac{N-1}{2})^{1/\alpha}} \frac{\ln \ch (\beta x)}{x^{1+\alpha}}dx\Bigr)\stackrel{d}{\to} \frac{\beta Y_\alpha}{2^{1/\alpha}},
\end{equation} which we shall prove in this section.

Let $n=N(N-1)/2.$	Recall that $\ln \bar Z_N=\sum_{i<j}\ln \ch (\beta J_{ij})$, which is a sum of a triangular array of i.i.d.\ heavy-tailed random variables. 
	We decompose $\ln \bar Z_N$ into light-tailed and heavy-tailed parts by writing   $\ln \bar Z_N= S_1 + S_2,$ where 
	\begin{align*}
		S_1 \colonequals\sum_{i<j}\ln \ch \bigl(\beta J_{ij}\1_{\{|J_{ij}|\leq 1\}}\bigr)\,\,\mbox{and}\,\, S_2 \colonequals\sum_{i<j}\ln \ch \bigl(\beta J_{ij}\1_{\{|J_{ij}|>1\}}\bigr).
	\end{align*}
Let us  first address $S_1$.
	Define $$X_{ij}=\ln\ch\bigl(\beta J_{ij}\1_{\{|J_{ij}|\leq 1\}}\bigr)-\e\ln\ch\bigl(\beta J_{ij}\1_{\{|J_{ij}|\leq 1\}}\bigr)$$ for $1 \le i< j \le N$ and put $s_N^2 =\sum_{i<j}\e X_{ij}^2$.
	It is easy to check that
	\begin{align*} s_N^2 &=\frac{\alpha(N-1)}{2} \Bigl(\int_{N^{-1/\alpha}}^1 \frac{\bigl(\ln\ch(\beta x)\bigr)^2}{x^{1+\alpha}}dx -\frac{\alpha}{N} \Bigl(\int_{N^{-1/\alpha}}^1 \frac{\ln\ch(\beta x)}{x^{1+\alpha}}dx\Bigr)^2\Bigr)  \\
 &~\sim \frac{\alpha N}{2}\int_0^1 \frac{\bigl(\ln\ch(\beta x)\bigr)^2}{x^{1+\alpha}}dx.
	\end{align*}
	Since $X_{ij}$'s are bounded and $s_N^2 \to\infty$ as $N\to\infty$, the Lindeberg condition \[\forall \epsilon>0, \  \frac{1}{s_N^2} \sum_{i<j}\e X_{ij}^2 \1_{\{|X_{ij}|\geq \epsilon s_N \}} \to 0 \]  holds trivially.
Hence, by the Lindeberg-Feller central limit theorem, we have, as $N \to \infty$,  
\begin{equation}\label{eq:clt1}
  \frac{1}{\sqrt{N}}\Bigl( S_1 -\frac{\alpha(N-1)}{2} \int_{N^{-1/\alpha}}^1\frac{\ln\ch(\beta x)}{x^{1+\alpha}}dx\Bigr)=\frac{\sum_{i<j}X_{ij}}{\sqrt{N}} \stackrel{d}{\to}   N\Bigl(0,\ \frac{\alpha}{2}\int_0^1 \frac{\bigl(\ln\ch(\beta x)\bigr)^2}{x^{1+\alpha}}dx\Bigr).
\end{equation}

	Next, we handle $S_2.$ Note that $J\1_{\{|J|>1\}} \stackrel{d}{=}  B g$, where $B$ is Bernoulli with parameter $1/N$, $g$ has density $(\alpha/2) |x|^{-(1+\alpha)}  \1_{\{|x|>1\}} $, and they are independent of each other. Therefore, we can represent 
    	\[ S_2  =  \sum_{i=1}^M \ln \ch(\beta g_i), \] where $(g_i)_{i \ge 1}$ are i.i.d.\  copy of  $g$ and 
	 $M$ is an independent $\mathrm{Binomial}(n, 1/N)$ random variable. In order to find out the asymptotic distribution of $S_2,$ we introduce $S_{2,m}=\sum_{i=1}^m\ln \cosh(\beta g_i)$ for any $m\geq 1.$ Note that $\ln \ch (\beta g)$ is a positive random variable with tail probability
	 \[ \Prob \bigl(\ln \ch (\beta g)>x\bigr) = \bigl( \beta^{-1} \mathrm{arccosh}(e^x)\bigr)^{-\alpha} \sim \beta^\alpha x^{-\alpha}\,\,\text{ as } x \to \infty. \]
	Set 
	  \[c_m =\inf\bigl\{x > 0 : \Prob \bigl( \ln \ch (\beta g) >x\bigr) < m^{-1} \bigr\}\,\,\mbox{and}\,\,d_m= m \e \ln \ch (\beta g)\1_{\{ \ln \ch (\beta g) \leq c_m\}}.\]
	It is easy to verify that 
	\begin{equation*}
		c_m \sim \beta m^{1/\alpha} \,\,\mbox{and}\,\,  d_m  =  \alpha m \int_{1}^{m^{1/\alpha}} \frac{\ln \ch (\beta x)}{x^{1+\alpha}}dx + o(m^{1/\alpha}).
	\end{equation*}
	By the standard result on the sum of i.i.d.\ heavy-tailed random variables, see, e.g., \cite[Theorem 3.8.2]{durrett_2019}, it follows that 
	\begin{equation}\label{eq:heavy_tailed_CLT}
	\frac{S_{2,m} -d_m}{c_m}\stackrel{d}{\to}Y_\alpha\,\,\mbox{as $m\to\infty$,}
	\end{equation}
 where the limiting distribution $Y_\alpha$ is as in  \eqref{eq:Y_a:charac}. We claim that 
	\begin{equation} \label{eq:res_error}
	 \frac{S_2 - S_{2,m}}{c_{m}}\stackrel{p}{\to}0\,\,\mbox{as}\,\,m=\frac{N-1}{2}\to\infty.
	  \end{equation}
	Let $Q = |M-m |$. It suffices to show that $ N^{-1/\alpha} \sum_{i=1}^{Q} \ln\ch(\beta g_i) \stackrel{p}{\to} 0$. 	To that end, fix $\epsilon>0$ and $\delta>0$. Keeping in mind that $\ln \cosh$ is positive, we can bound
	\begin{equation} \label{eq:random_sum_bound}
	 \Prob \Big ( N^{-1/\alpha} \sum_{i=1}^{Q} \ln\ch(\beta g_i) > \eps \Big ) \le \Prob \Big ( N^{-1/\alpha} \sum_{i=1}^{\lceil\sqrt{N}\ln N\rceil } \ln\ch(\beta g_i) > \eps \Big )  + \Prob\bigl( Q > \lceil\sqrt{N}\ln N\rceil \bigr).  
	 \end{equation}
	 On the one hand, by Chebyshev's inequality $\Prob( Q > \lceil\sqrt{N}\ln N \rceil) \le N^{-1} (\ln N)^{- 2} \mathrm{Var}(M) \le (\ln N)^{- 2}$, which implies that the second term of \eqref{eq:random_sum_bound} vanishes as $N \to \infty.$  
	On the other hand, from \eqref{eq:heavy_tailed_CLT}, we obtain that
	\begin{equation*}
	\lim_{K\to\infty}\liminf_{m\to\infty} \Prob \big( S_{2,m} \le K c_m + d_m  \big) =1. 
	\end{equation*}
Now, a straightforward computation yields that $d_m $ is bounded above by $C m^{1 /\alpha}, C m \ln m, $ and $ Cm$ when $ 0<  \alpha <1, \alpha = 1,$ and  $1< \alpha <2$, respectively, which implies that for any $0< \alpha< 2$, 
	\[  \lim_{K\to\infty} \liminf_{m\to\infty} \Prob \big( S_{2,m} \le K m^{\max(1, 1/\alpha)} \ln m  \big)=1. \]
 Here, if we take $m = \lceil\sqrt{N}\ln N \rceil$, then $m^{\max(1, 1/\alpha)} \ln m\leq CN^{2^{-1}\max(1,1/\alpha)}(\log N)^{2+1/\alpha}$ for some constant $C>0$ independent of $N$. Consequently, 
 \[  \lim_{K\to\infty} \liminf_{m=\lceil \sqrt{N}\log N\rceil\to\infty} \Prob \big( N^{-1/\alpha}S_{2,m} \le K N^{-\max(\alpha^{-1}-2^{-1},(2\alpha)^{-1})}(\log N)^{2+1/\alpha} \big)=1. \]
	Since $\max(\alpha^{-1}-2^{-1},(2\alpha)^{-1})>0,$ this limit implies that the first term of \eqref{eq:random_sum_bound} also goes to zero as $N \to \infty$, which confirms \eqref{eq:res_error}.

Finally, Theorem~\ref{thm: generalized: fluctuation of Zbar} follows from combining \eqref{eq:clt1}, \eqref{eq:heavy_tailed_CLT}, and \eqref{eq:res_error} and using the estimate
	\begin{equation}\label{eq: centering small}
	    \int_{1}^{N^{1/\alpha}} \frac{\ln \ch (\beta x)}{x^{1+\alpha}}dx  =  o(N^{1/\alpha} ),\,\,\forall 0<\alpha<2. 
	\end{equation}

\subsection{Proof of Theorem \ref{fluctuation hat Z}}

Our proof uses the powerful cluster expansion technique introduced by Aizenman, Lebowitz, and Ruelle \cite{aizenman} to study the fluctuation of the free energy in the SK model. This approach was later adapted to the VB model by K\"oster  \cite{koster}. Before we go into the main argument, we state some elementary results about the weak convergence and random graphs. We refer to  \cite{koster} for their proofs.

\begin{lemma} \label{cut}
	Let $(X_N)_{N \ge 1}$ be a family of random variables that can be approximated, with respect to the $L^2$ norm, by some weakly convergent sequences, i.e., for any $N\geq 1$ and $\eta>0$, there exist random variables $X_{N,\eta}$ and $X_\eta$ such that
\begin{itemize}
		\item[(i)] $\lim_{\eta\downarrow 0}\sup_N \e |X_N -X_{N, \eta} |^2 =0.$
		
		\item[(ii)] For each $\eta>0$, $X_{N, \eta} $ converges in distribution to $X_\eta$ as $N\to\infty$.
\end{itemize}
	Then there exists a random variable $X$ such that $X_\eta \stackrel{d}{\to} X$ as $\eta \to 0$ and $X_N \stackrel{d}{\to} X$ as $N \to \infty$. 
\end{lemma}

Let $G_{N, p/N}$ denote the Erd\"os-R\'enyi random graph on the vertex set $[N]$ with parameter $p/N$, i.e., each edge is selected with probability $p/N$, independently of the other edges.
\begin{lemma}[Lemma 2.1, \cite{koster}]\label{cyclepoisson}
 For  $k\geq 3$, let $Y_{N,k}$ denote the number of $k$-cycles contained in $G_{N, p/N}$.
	Then, for each fixed $m\geq 3$, as $N \to \infty$, 
	 \[(Y_{N,3},\dots,Y_{N,m})\xrightarrow[N\to\infty]{d}(Y_3,\dots,Y_m),\] where $Y_k$ is an independent Poisson random variable with parameter $p^k/2k$ for $3\leq k\leq m$.
\end{lemma}

\begin{lemma}[Lemma 2.2, \cite{koster}]\label{vertexdisjoint}
 For  $m\geq 3$, let $\mathcal{A}_{N,m}$ denote the event that  the cycles in $G_{N, p/N}$ of length at most $m$  are vertex-disjoint.
	Then, for each fixed $m\geq 3$, as $N \to \infty$, 
	\[\Prob(\mathcal{A}_{N,m})\to 1.\]
\end{lemma}

We now begin the main argument. Following \cite{aizenman}, we  expand the product and then sum over $\sigma$ to obtain that 
\[ \widehat Z_N = 2^N \sum_{\partial \Gamma = \emptyset} w(\Gamma),\]
 for $w(\Gamma):= \prod_{e\in E(\Gamma)}\hth(\beta J_{e})$. Here, the summation is over all simple and closed graphs $\Gamma$ with vertex set $V(\Gamma)\subseteq [N]$ and edge set $E(\Gamma)\subseteq\{(i,j):1\leq i<j\leq N\},$ i.e., the edges are are not repeated and 
 \[ \partial \Gamma := \{i\in [N] : \text{$i$ belongs to an odd number of edges of $\Gamma$}\}=\emptyset.\] 
 Note that such graphs can be written as the union of (edge-disjoint)  cycles, but such decomposition is  not necessarily unique. We denote a cycle by~$\gamma$. By a slight abuse of notation, we write  $w(\gamma)$  for $ \prod_{e\in E(\gamma) }\hth(\beta J_{e})$.
The following  lemma is crucial, which states that at high temperature, large graphs do not contribute much to the partition function. 
Similar results were also obtained earlier for the SK model    \cite[Lemma 3.3]{aizenman} and for the VB model  \cite[Lemma 2.3]{koster}. 
\begin{lemma}\label{Bigraphdie}
	Let $0<\alpha<2$. For $\beta<\beta_{\alpha}$,
 there exists a constant $C$ such that 
	\[\sup_{N \ge 1} \e \Bigl(\sum_{\partial \Gamma = \emptyset, |\Gamma|\geq m} w(\Gamma)  \Bigr)^2 \leq  Ce^{\sqrt{2m}}(\beta/\beta_{\alpha})^{\alpha m}, \,\,m \ge 1.\]
\end{lemma}
\begin{proof}
	Since  the weights attached to distinct edges are independent, $\e w(\Gamma) w(\Gamma')  = 0$ if $\Gamma \ne \Gamma'$. Therefore, 
	we have \[\e \Bigl(\sum_{\partial \Gamma = \emptyset, |\Gamma|\geq m} w(\Gamma)  \Bigr)^2 = \sum_{\partial \Gamma = \emptyset, |\Gamma|\geq m} \e w(\Gamma)^2. \]
	For sufficiently large $m$, let $\delta=\delta(m) >0$ be such that $e^\delta (\beta/\beta_{\alpha})^{\alpha} = 1-1/(2\sqrt{m})$.
	
 We claim that for any $N\ge m$, 
	\begin{equation}\label{eq:cycle_claim}
	\sum_{\gamma: |\gamma|=m} \e w(\gamma)^2 \le \frac{1}{2m}(\beta/\beta_{\alpha})^{\alpha m}, \quad \forall m\ge 1.
	\end{equation}
	Indeed, there are $$ \frac{N!}{(N-m)! (2m)}$$ many cycles of size $m$, and we obtain our claim from  \eqref{freeenergy:eq2}, since
	  \begin{align*}
		& \sum_{\gamma : |\gamma|=m} \e w(\gamma)^2 = \frac{N(N-1)\cdots(N-m+1)}{2m} \bigl(\e (\hth \beta J)^2\bigr)^m
		\\&= \frac{1}{2m} \prod_{i=1}^{m-1}  \Bigl(1-\frac{i}{N}\Bigr) \cdot \bigl( N\e \hth^2(\beta J) \bigr)^m \le \frac{1}{2m} (\beta/\beta_{\alpha})^{\alpha m}.
	\end{align*}

From the claim \eqref{eq:cycle_claim}, we derive that 
	\begin{align*}
		\sum_{\partial \Gamma = \emptyset, |\Gamma|\geq m} \e w(\Gamma)^2 &\leq \sum_{\partial \Gamma = \emptyset, |\Gamma|\geq m} e^{\delta |\Gamma| -\delta m}\e w(\Gamma)^2 
		\\
  &\leq  e^{-\delta m} \prod_{\gamma} \bigl(1+e^{\delta|\gamma|} \e w(\gamma)^2 \bigr) \\
  &=  e^{-\delta m} \prod_{k\leq N}\prod_{|\gamma|=k} \bigl(1+e^{\delta|\gamma|} \e w(\gamma)^2 \bigr) \leq e^{-\delta m} \exp\Bigl(\sum_{k\geq 3}\frac{( \beta/\beta_{\alpha})^{\alpha k}  e^{\delta k}}{2k} \Bigr),
	\end{align*}
 where the last inequality used $1+x\leq e^x$ for all $x>0.$
Now, we conclude that  
	\begin{align*}
e^{-\delta m} \exp\Bigl(\sum_{k\geq 3}\frac{( \beta/\beta_{\alpha})^{\alpha k}  e^{\delta k}}{2k} \Bigr)
&\leq \frac{( \beta/\beta_\alpha)^{\alpha m}}{(1-1/(2\sqrt{m}))^m}\exp\Bigl(\frac{1}{6(1-( \beta/\beta_\alpha)^\alpha e^\delta)}\Bigr)\\
		&=( \beta/\beta_{\alpha})^{\alpha m} \exp\Bigl( -m \ln (1-1/(2\sqrt{m})) +\frac{\sqrt{m}}{3}\Bigr)\\
		 &\le  ( \beta/\beta_{\alpha})^{\alpha m}  e^{\sqrt{2m}}, 
	\end{align*}
 where the first inequality used $\sum_{k\geq 3}x^k/(2k)\leq 6^{-1}(1-x)^{-1}$ for $0<x<1$ and the last inequality used  $\ln (1 - x)\ge -3x/2$ for $ 0 \le  x \le 1/2$.
 This completes our proof.
\end{proof}
Denote 
\[ B_{N,m}\colonequals \{\Gamma: \partial \Gamma = \emptyset, \text{ $\Gamma$ is a union of cycles of size at most $m$}\}.\]
For $N \ge 1$ and $m \ge 1$, let us define 
	\[
	\widehat X_N = 2^{-N} \widehat Z_N = \sum_{\partial \Gamma = \emptyset} w(\Gamma)  \,\,\mbox{and}\,\, \widehat X_{N, m} = \sum_{ \Gamma \in B_{N, m}} w(\Gamma).  \]
Note that each $w(\Gamma)$ in $\widehat X_N-\widehat X_{N,m}$ is a summation, in which $w(\Gamma)$ with  $\Gamma$ satisfying $\partial\Gamma=\emptyset$ and being of size at least $m$. Therefore, from Theorem \ref{Bigraphdie},
\begin{equation}\label{eq:m_approx}
 \lim_{m\to\infty}\sup_N  \e ( \widehat X_N  -  \widehat X_{N, m})^2 =0.
 \end{equation}
For $\eps >0$, 	
set
\[ \widehat X_{N, m, \eps } = \sum_{\Gamma \in B_{N, m} } w_\eps(\Gamma) \]
for $w_\eps(\Gamma):= \prod_{e\in E(\Gamma)}\hth(\beta J_{e})\1_{\{|J_e| > \epsilon\}}.$
\begin{lemma}\label{lem:fixed_m_approx}
For any $\beta>0$ and any each fixed $m \ge 1$, 
\[ \lim_{\eps\downarrow 0}\sup_{N\ge 3} \e ( \widehat X_{N, m, \eps } -  \widehat X_{N, m })^2=0.\]
\end{lemma}
\begin{proof}
Throughout this proof, to lighten up the notation, we will always assume that $\partial \Gamma = \emptyset$ without explicitly mentioning it.   Using orthogonality as before, the expectation in the lemma equals $\sum_{\Gamma \in B_{N,m}} \e  (w(\Gamma) - w_\eps(\Gamma) )^2$.  
For a fixed $\Gamma$ with $| \Gamma| = k$, 
\[  \e  \bigl(w(\Gamma) - w_\eps(\Gamma) \bigr)^2 = \e  w(\Gamma)^2 - \e w_\eps(\Gamma)^2 = \bigl( \e  \hth^2 (\beta J) \bigr)^k - \bigl( \e  \hth^2 (\beta J) \1_{ \{ |J| > \eps \} }  \bigr)^k . \]
Similar to \eqref{freeenergy:eq2}, 
\[  N \e \hth^2 (\beta J) \1_{ \{ |J| \le \eps \} }  \uparrow \frac{\eps^{2 - \alpha} \alpha \beta^2}{2-\alpha}. \]
Using this and the inequality $(x+y)^k  - x^k \le k (x+y)^{k-1} y$ for any $x, y \ge 0$, we bound the difference above as
 \begin{equation*}
N^k\e  \bigl(w(\Gamma) - w_\eps(\Gamma) \bigr)^2 \le k      \bigl( N\e  \hth^2 (\beta J) \bigr)^{k-1}   N \e  \hth^2 (\beta J) \1_{ \{ |J| \le \eps \} } \le  C \theta^k\eps^{2 - \alpha}   \end{equation*}
for some suitable constants $C$ and $\theta$ that only depend on $\alpha$ and $\beta$. 
	Since any graph in $B_{N, m}$ can be represented as a disjoint union of cycles of size at most $m$, and there are at most $N^k$ many cycles  of size $k$,  we obtain
\begin{align*}
\sum_{ \Gamma \in B_{N,m}} \e ( w(\Gamma) - w_\eps(\Gamma))^2 &\le C' \eps^{2 - \alpha} \prod_{| \gamma | \le m} \big ( 1+  \theta^{ |\gamma|}  N^{-|\gamma|} \big ) \\
&\le C' \eps^{2 - \alpha} \prod_{| \gamma | \le m} \exp \big ( \theta^{ |\gamma|}  N^{-|\gamma|} \big) 
\le C' \eps^{2 - \alpha}  \exp \Big ( \sum_{k=3}^m \theta^k   \Big)
\end{align*}  for another constant $C'>0$ that only depends on $\alpha$ and $\beta.$
Letting $\eps \to 0$ finishes the proof of the lemma. 
	\end{proof}


	We now proceed to show that for fixed $m$ and $\eps$, $\widehat X_{N, m, \eps}$ converges in distribution as $N \to \infty$. 
 First, observe that the set $E_{N,\eps}$ of edges $e$ with $|J_e| > \eps$ forms an Erd\H{o}s-R\'enyi graph $G_{N, \eps^{-\alpha}/N}$. Let $\mathcal{A}_{N,m}$ be the event that all cycles of size at most $m$ in $E_{N,\eps}$ are vertex disjoint.
	Observe that on the event $\mathcal{A}_{N,m}$, any  $\Gamma \in B_{N, m}$, satisfying $|J_e| > \eps$ for each edge $e\in E(\Gamma)$, can be represented uniquely as the union of these vertex disjoint cycles of size at most $m$ in $E_{N,\eps}$.
	Therefore, \[\1_{\mathcal{A}_{N,m}}\sum_{\Gamma \in \mathcal B_{N,m}} w_\eps(\Gamma) =\1_{\mathcal{A}_{N,m}} \prod_{|\gamma|\leq m}\Bigl(1+\prod_{e\in E(\gamma)}\hth(\beta J_{e}) \1_{\{|J_e|> \eps\}}\Bigr).\]
On the other hand, note that $\hth(\beta J) \1_{\{|J|\geq \epsilon\}} \stackrel{d}{=} \hth(\beta g) B$, where $g$ has a density 
\begin{equation}
    2^{-1}\alpha \eps^{\alpha} |x|^{-(1+\alpha)}  \1_{\{|x|\geq \eps\}},\label{eq:g_eps density}
\end{equation} $B$ is $\mathrm{Ber}(\eps^{-\alpha}/N)$, and they are independent of each other.
We can further write \[\1_{\mathcal{A}_{N,m}}\sum_{\Gamma \in \mathcal A_{N,m}} w_\eps(\Gamma) \stackrel{d}{=}\1_{\mathcal{A}_{N,m}} \prod_{|\gamma|\leq m}\Bigl(1+\prod_{e\in E(\gamma)}\hth(\beta g_{\gamma,e}) \Bigr).\]
Here, $(g_{\gamma,e})_{\gamma,e}$ are i.i.d. copies of $g$ and independent of other randomness. Lemma \ref{vertexdisjoint} implies that $\1_{\mathcal{A}_{N,m}}\to 1$ in probability as $N\to\infty$. 
	Hence,  for fixed $m$ and $\eps$, as $N\to\infty$, $\widehat X_{N, m, \eps}$ shares the same weak limit as that of
$$ Y_{N, m, \eps} := \exp\Bigl(\sum_{k=3}^m   \sum_{ \gamma \in E_{N,\eps}, |\gamma| = k}\ln \Bigl(1+\prod_{e\in E(\gamma)}\hth(\beta g_{\gamma,e}) \Bigr)\Bigr).$$ 
Recall that $E_{N,\eps}$ is an Erd\H{o}s-R\'enyi graph $G_{N, \eps^{-\alpha}/N}$.
Lemma \ref{cyclepoisson}  implies that $\ln Y_{N,m,\eps}$, as $N \to \infty$,  converges in distribution to the sum of $(m-2)$-many independent compound Poisson random variables, whose characteristic function is given by 
\begin{equation}\label{eq:charac1}
\prod_{k=3}^m \exp\Bigl(\frac{\eps^{-\alpha k}}{2k} \bigl(\e e^{it W_k}-1\bigr)\Bigr),
\end{equation}
	where the compound variable $W_k$ satisfies $W_k \stackrel{d}{=} \ln \bigl(1+\prod_{i\leq k}\hth(\beta g_i)\bigr)$ and $(g_i)_i$ are i.i.d.\ with distribution \eqref{eq:g_eps density}. A direct computation gives that 
	\begin{align*}
		\e e^{it W_k}-1 = \frac{\alpha^k\eps^{\alpha k}}{2^k}\int_{ I_{k, \eps}}\Bigl( e^{it \ln \bigl(1+\prod_{i\leq k}\hth(\beta x_i)  \bigr)} -1\Bigr)  \frac{1}{\prod_{i\leq k} |x_i|^{1+\alpha}}dx_1\dots dx_k,
	\end{align*}
	where $ I_{k, \eps} = \{ x \in \mathbb{R}^k: |x_i|\geq\epsilon, 1 \le i\leq k \}$. Plugging this into \eqref{eq:charac1}, we deduce that 
 $\ln Y_{N, m, \eps}$ converges in distribution to a random variable, say $\ln \widehat X_{m, \eps},$ with the characteristic function \[\prod_{k=3}^m \exp\Big(\frac{\alpha ^k}{k 2^{k+1}} \int_{I_{k, \eps}}\Bigl( e^{it \ln \bigl(1+\prod_{i\leq k}\hth(\beta x_i)  \bigr)} -1\Bigr)\frac{1}{\prod_{i\leq k} |x_i|^{  1+\alpha}}dx_1\dots dx_k\Big).\]
Up to here, we have shown that $\widehat X_{N, m, \eps } \stackrel{d}{\to} \widehat X_{m, \eps}$ as $N \to \infty$.
 
	Since $x\mapsto \hth(\beta x) $ is an odd function, the characteristic function of $\ln \widehat X_{m, \eps}$   is equal to \[\prod_{k=3}^m \exp\Big(\frac{\alpha ^k}{k 2^{k+1}} \int_{I_{k, \eps} } \Bigl(e^{it \ln \bigl(1+\prod_{i\leq k}\hth(\beta x_i)\bigr)} -1- it\prod_{i\leq k}\hth(\beta x_i) \Bigr) \frac{1}{\prod_{i\leq k} |x_i|^{1+\alpha}}dx_1\dots dx_k\Big),\]
	which converges as $\eps \to 0$ to
	\begin{equation}\label{eq:charac2}
	\prod_{k=3}^m \exp\Big(\frac{\alpha ^k}{k 2^{k+1}} \int_{\mathbb{R}^k } \Bigl(e^{it \ln \bigl(1+\prod_{i\leq k}\hth(\beta x_i)\bigr)} -1- it\prod_{i\leq k}\hth(\beta x_i) \Bigr) \frac{1}{\prod_{i\leq k} |x_i|^{1+\alpha}}dx_1\dots dx_k\Big).
	\end{equation}
The finiteness of the integral above near the origin is guaranteed by  $e^{it\ln(1+u)}-1-itu =O(u^2)$ as $u\to0$ and $|\tanh( x)| \le |x|$.   As a result, we have just shown that $\widehat X_{m, \eps} \stackrel{d}{\to} 
\widehat X_m$ for some random variable such that the characteristic function of $\ln  \widehat X_{m}$ is as given in \eqref{eq:charac2}. 
 By Lemma \ref{cut} and Lemma~\ref{lem:fixed_m_approx}, we deduce that  for each fixed $m$, $\widehat X_{N, m} \stackrel{d}{\to} \widehat X_m$  as $N \to \infty$. Another application of Lemma \ref{cut}, along with \eqref{eq:m_approx}, implies that the weak limit of $\widehat X_N$ exists and is the same as the weak limit of $\widehat X_m$. To identify the weak limit of $\widehat X_m$, we take the limit of the characteristic function of its logarithm as  given in   \eqref{eq:charac2}  and  realize that limit matches  with the characteristic function of $\ln X_{\alpha, \beta}$ as described in Theorem~\ref{fluctuation hat Z}. 
 This completes the proof of the theorem. 
 
\section{Establishing the Overlap Concentration}

This section is devoted to establishing the concentration of the overlaps in the L\'evy model at high temperature for $ 1 < \alpha < 2$ stated in Theorem~\ref{hightempoverlap}. We first study the site overlap and then the bond overlap.

\subsection{Proof of Equation (\ref{siteoverlapzero}) in Theorem~\ref{hightempoverlap}}\label{subsec:add2}

We need the following key proposition.

\begin{proposition}\label{prop1'}
	Assume $1<\alpha<2$.
	For any $\beta>0$, 
 there exists a constant $C>0$ independent of $N$ such that for any $N\geq 1,$
	\begin{equation*}\Bigl|\e\frac{1}{N}\sum_{i<j} J_{ij}\la \sigma_i\sigma_j\ra-\frac{c_1}{2}\bigl(1- \e\bigl\la R_{ L}^2\bigr\ra\bigr)\Bigr| \le \frac{C}{(\ln N)^{\alpha-1}},
\end{equation*}
where for $\ell\geq 1$, $c_\ell:=N\e  J\tanh^\ell(\beta  J)>0$ 
and, independent of all other randomness,
	$L$ is a random variable with the probability mass function,
\begin{equation*}
    \p(L=2k)=\frac{c_{2k-1}-c_{2k+1}}{c_1}, \quad \forall k\ge 1.
\end{equation*}
\end{proposition}

\begin{proof}
In this proof, we denote by $C_0,C_1,C_2,\ldots$ positive constants independent of $N.$ 
For $i<j,$ let  $\widehat{Z}_N^{(ij)}$ and $\la \cdot\ra_{ij}$ be the partition function and Gibbs expectation associated to the system, for which the disorder variable $J_{ij}$ is set to zero. Then we have
	\[ \widehat Z_N =  \big  \la 1+ \sigma_i \sigma_j \tanh(\beta J_{ij})  \big  \ra_{ij} \widehat Z_N^{(ij)},  \quad \la \sigma_i \sigma_j \ra \widehat Z_N =  \big  \la \sigma_i \sigma_j (1+ \sigma_i \sigma_j \tanh(\beta J_{ij}) )  \big  \ra_{ij} \widehat Z_N^{(ij)},\]
which implies that
	\begin{align}\label{overlap:eq1}
		\la \sigma_i\sigma_j\ra&=\frac{\la \sigma_i\sigma_j\ra_{ij}+\tanh(\beta J_{ij})}{1+\la \sigma_i\sigma_j\ra_{ij}\tanh(\beta J_{ij})}=f\bigl(\la \sigma_i\sigma_j\ra_{ij},\beta J_{ij}\bigr),
	\end{align}
 where for any $a\in (-1,1)$ and $x\in \mathbb{R},$
 $$
 f(a,x):=\frac{a+\tanh(x)}{1+a\tanh(x)}.
 $$
To handle this objective, note that for any $a\in (-1,1)$ and $x\in \mathbb{R},$ 
	\begin{align}
 \label{eq3.1}|f(a,x)|&< 1,\\
	\label{eq3}	|f(a,x)-a|&\leq 2|\tanh(x)|,
\end{align}
and
\begin{align}\label{eq3.2}			
\partial_af(a,x)
	&=\frac{1-\tanh(x)^2}{(1+a\tanh(x))^2}\leq \frac{2}{1-|\tanh(x)|}.
	\end{align}
	From \eqref{eq3},
	\begin{equation}\label{cavityineq}
		\e\bigl|\la\sigma_i\sigma_j\ra-\la\sigma_i\sigma_j\ra_{ij}\bigr|\leq 2\e|\tanh(\beta J_{ij})|\leq 2\e\beta|J_{ij}|
    =2\beta N^{-1/\alpha}\e|X|.
	\end{equation}
	From \eqref{overlap:eq1}, if we let $\la\cdot \ra_{ij}'$ be  the Gibbs expectation $\la \cdot\ra$, in which $J_{ij}$ is replaced by an independent copy $J_{ij}'$, then 
	\begin{align}\label{overlap:eq3}
		\e\bigl|		\la \sigma_i\sigma_j\ra_{ij}'- \la \sigma_i\sigma_j\ra_{ij}\bigr|\leq 2 \beta N^{-1/\alpha}\e|X|.
	\end{align}
	From \eqref{eq3.1} and \eqref{eq3.2}, we have that for any $x \in \mathbb{R}$, 
\begin{align*}
	\Psi_{ij} (x) := 	\bigl|f\bigl(\la \sigma_i\sigma_j\ra_{ij},\beta x\bigr)-f\bigl(\la \sigma_i\sigma_j\ra_{ij}',\beta x\bigr)\bigr|
		&\leq \min\Bigl(\frac{2\bigl|\la \sigma_i\sigma_j\ra_{ij}-\la \sigma_i\sigma_j\ra_{ij}'\bigr|}{1-|\tanh(\beta x)|},2\Bigr).
	\end{align*}
	Consequently, for any $K>N^{-1/\alpha},$
	\begin{align*}
	\e\bigl[|J_{ij}|  \Psi_{ij}(J_{ij}) ;|J_{ij}|\leq K\bigr] &\leq 2\e\Bigl[\frac{|J_{ij}|}{1-|\tanh(\beta J_{ij})|};|J_{ij}|\leq K\Bigr]\e\bigl|\la \sigma_i\sigma_j\ra_{ij}-\la \sigma_i\sigma_j\ra_{ij}'\bigr|, \\
	\e\bigl[|J_{ij}| \Psi_{ij}(J_{ij});|J_{ij}|\geq K\bigr]
		&\leq 2\e\bigl[|J_{ij}|;|J_{ij}|\geq K\bigr].
		\end{align*}
		 A simple computation yields 
\[  \e\Bigl[\frac{| J|}{1-|\tanh(\beta  J)|} ;| J|\leq K\Bigr] \le e^{2\beta K} \frac{\e| X|}{N^{1/\alpha}}, \quad \e[|J|;|J|\ge K]\le  \frac{\alpha K^{-\alpha+1}}{(\alpha-1)N}.
\]
	These bounds, coupled with the estimate \eqref{overlap:eq3}, imply that 
	\begin{align*}
\frac{1}{N}\sum_{i<j}\e\Bigl[|J_{ij}| \Psi_{ij}(J_{ij})\Bigr]
		&\leq \frac{ C_0 e^{2\beta K}}{N^{-1+2/\alpha}}+C_0 K^{-\alpha+1}.		
	\end{align*}
Letting $K=(4\beta)^{-1}(2/\alpha-1)\ln N$ yields that
	\begin{align}\label{overlap:eq4}
		\e\Bigl|\frac{1}{N}\sum_{i<j}J_{ij}\la \sigma_i\sigma_j\ra-\frac{1}{N}\sum_{i<j}J_{ij}f\bigl(\la \sigma_i\sigma_j\ra_{ij}',\beta J_{ij}\bigr)\Bigr|\leq \frac{C_1}{(\ln N)^{\alpha-1}}.
	\end{align}


	Next, we compute
	\[		\frac{1}{N}\sum_{i<j}\e\bigl[J_{ij}f\bigl(\la \sigma_i\sigma_j\ra_{ij}',\beta J_{ij}\bigr)\bigr]
        \]
	with the help of the fact that $\la \cdot \ra_{ij}'$ is independent of $J_{ij}$. 
	Using the geometric series, we  write
	\[ f(a,x) =a-\sum_{\ell=1}^\infty (-1)^\ell (a^{\ell-1}-a^{\ell+1})\tanh(x)^{\ell}.\]
	Therefore,
	\[J_{ij}f\bigl(\la \sigma_i\sigma_j\ra_{ij}',\beta J_{ij}\bigr)
		= J_{ij}\la \sigma_i\sigma_j\ra_{ij}'-\sum_{\ell=1}^\infty(-1)^\ell \bigl( {\la \sigma_i\sigma_j\ra_{ij}'}^{\ell-1}- {\la \sigma_i\sigma_j\ra_{ij}'}^{\ell+1}\bigr)J_{ij}\tanh(\beta J_{ij})^{\ell}.
	\]
		
	Note that \[
		\sum_{\ell=1}^\infty \bigl|(a^{\ell-1}-a^{\ell+1})x\tanh(x)^{\ell}\bigr|\le |x|(1-a^2) \sum_{\ell=1}^\infty |a|^{\ell-1} = |x|(1+|a|)\le 2|x|,
	\]
	so \[\sum_{\ell=1}^\infty \e \bigl| \bigl( {\la \sigma_i\sigma_j\ra_{ij}'}^{\ell-1}- {\la \sigma_i\sigma_j\ra_{ij}'}^{\ell+1}\bigr)J_{ij}\tanh(\beta J_{ij})^{\ell}\bigr|\le 2\e |J|<\infty.\]
Consequently, using the Fubini theorem and noting that $\la \sigma_i\sigma_j\ra$ and $\la \sigma_i\sigma_j\ra_{ij}'$ are identically distributed yield that
	\begin{align*}
		\sum_{i<j}\e\bigl[J_{ij}f\bigl(\la \sigma_i\sigma_j\ra_{ij}',\beta J_{ij}\bigr)\bigr]=\frac{1}{N}\sum_{\ell\,\,\mbox{\rm \footnotesize odd}}c_{\ell}\Bigl(\sum_{i<j}\e {\la \sigma_i\sigma_j\ra}^{\ell-1}- \sum_{i<j}\e{\la \sigma_i\sigma_j\ra}^{\ell+1}\Bigr).
	\end{align*}
	To handle the right-hand side, note that we can write
	\begin{align}\label{extra:eq1}
		\sum_{i<j}\bigl\la \sigma_i\sigma_j\bigr\ra^\ell=\sum_{i<j}\bigl\la \sigma_i^1\cdots\sigma_i^\ell\sigma_j^1\cdots\sigma_j^\ell\bigr\ra =\frac{1}{2}\bigl\la N^2R_\ell^2-N\bigr\ra
	\end{align}
	such that
	\begin{align}\label{add:equation1}
		\frac{1}{N}\sum_{i<j}\e\bigl[J_{ij}f\bigl(\la \sigma_i\sigma_j\ra_{ij}',\beta J_{ij}\bigr)\bigr]&=\frac{1}{2}\sum_{\ell\,\,\mbox{\rm \footnotesize odd}}c_{\ell}\bigl(\e\bigl\la R_{\ell-1}^2\bigr\ra-\e\bigl\la R_{\ell+1}^2\bigr\ra\bigr).
	\end{align}
	Consequently, from \eqref{overlap:eq4}
	\begin{align*}
		\limsup_{N\to\infty}\Bigl|\e\frac{1}{N}\sum_{i<j}J_{ij}\la \sigma_i\sigma_j\ra-\frac{1}{2}\sum_{\ell\,\,\mbox{\rm \footnotesize odd}}c_{\ell}\bigl(\e\bigl\la R_{\ell-1}^2\bigr\ra-\e\bigl\la R_{\ell+1}^2\bigr\ra\bigr)\Bigr|=0.
	\end{align*}
	
	Finally, note that for any sequences, $a_0\geq a_2\geq a_4\geq \cdots\geq 0,$ we can write
	\begin{align}\label{eq:summationbyparts}
		\sum_{\ell:\mbox{\rm \footnotesize odd}}c_\ell(a_{\ell-1}-a_{\ell+1})
		&=c_1\Bigl(a_0-\sum_{k=1}^\infty  \frac{c_{2k-1}-c_{2k+1}}{c_1}a_{2k}\Bigr),
	\end{align}
 where $
		c_1^{-1} \sum_{k=1}^\infty (c_{2k-1}-c_{2k+1})=1$. Also, note that from \eqref{extra:eq1}, $\e\la R_{\ell-1}^2\ra$ is nonincreasing in odd $\ell\geq 1.$
	It follows that
	\begin{align*}
	\limsup_{N\to\infty}\Bigl|\e\frac{1}{N}\sum_{i<j}J_{ij}\la \sigma_i\sigma_j\ra-\frac{c_1}{2}\bigl(1-\e\bigl\la R_L^2\bigr\ra\bigr)\Bigr|=0,
	\end{align*}
	where $L$ is a random variable, independent of $\la \cdot\ra$, with the probability mass function $\p(L=2k)=(c_{2k-1}-c_{2k+1})/c_1$ for all $k\geq 1$ .
%
%
\end{proof}

For each $\ell\ge 1$, we define \begin{equation*} 
    \gamma_\ell = \lim_{N\to\infty}c_\ell=\lim_{N\to\infty}\alpha \int_{N^{-1/\alpha}}^{\infty}\frac{\hth^\ell (\beta x)}{x^{\alpha}}dx=\alpha\int_{0}^{\infty}\frac{\hth^\ell (\beta x)}{x^{\alpha}}dx.
\end{equation*}

\begin{proof}[\bf Proof of Equation \eqref{siteoverlapzero}]
Using the convexity of $ 0<\beta \mapsto\ln Z_N(\beta)$ and applying Griffiths' lemma (see, e.g., \cite[Theorem 25.7]{rockafellar-1970a}), we obtain from Theorem \ref{freeenergy} that  for any $\beta < \beta_\alpha$, 
\begin{align*} 
 \lim_{N\to\infty}\frac{d}{d\beta}\frac{1}{N}\e\ln Z_N(\beta)=\frac{d}{d \beta} \Big ( \frac{\alpha}{2}\int_0^\infty\frac{\ln \cosh(\beta x)}{x^{\alpha+1}} dx\Big)=\frac{\gamma_1}{2},
\end{align*}
where the derivative is handled by using the fact that
for any $\beta,\beta'>0$ and $x\in \mathbb{R},$ 
	\[  \bigl|\ln\cosh(\beta x)-\ln \cosh(\beta' x)\bigr| \leq \min\bigl(\max(\beta,\beta')|x|^2,|x|\bigr) |\beta-\beta'|\]
 and the dominated convergence theorem.
On the other hand, by Proposition~\ref{prop1'}, 
\[  \lim_{N\to\infty}\frac{d}{d\beta}\frac{1}{N}\e\ln Z_N(\beta) =  \lim_{N\to\infty} \frac{1}{N}\sum_{i<j} \e J_{ij}\la \sigma_i\sigma_j\ra =  \frac{\gamma_1}{2}\big(1-  \lim_{N\to\infty} \e\bigl\la R_L^2\bigr\ra\big).  \]
As a result, we conclude that $\lim_{N\to\infty} \e\la R_L^2\ra = 0$ and this validates \eqref{siteoverlapzero} since for any $\ell\geq 1,$
\[	\frac{\gamma_{\ell-1}-\gamma_{\ell+1}}{\gamma_1} \limsup_{N\to\infty}\e\la R_{2\ell}^2\ra\leq \lim_{N\to\infty}  \e\bigl\la R_{ L}^2\bigr\ra=0.\]
\end{proof}

\subsection{Proof of Equation (\ref{eq: overlap and temperature}) in Theorem \ref{hightempoverlap}}
Since $\gamma_1 =\lim_{N\to\infty}c_1$, Proposition \ref{prop1'} implies \begin{equation*}
        \limsup_{N\to\infty}\Bigl| \e \la R_L^2\ra - \Bigl( 1- \frac{2}{\gamma_1}\e \frac{1}{N}\sum_{i<j}J_{ij}\la\sigma_i\sigma_j\ra \Bigr) \Bigr|  =0.
    \end{equation*}
    Then, \begin{equation}\label{eq: limit of R^2}
        1\ge \liminf_{N\to\infty} \e \la R_L^2 \ra = 1-\frac{2}{\gamma_1} \limsup_{N \to \infty}\e \frac{1}{N}\sum_{i<j} J_{ij} \la \sigma_i \sigma_j \ra.
    \end{equation}
    On the other hand, the average Hamiltonian at inverse temperature $\beta$ is bounded above by the free energy at any inverse temperature $\beta'>0$, as can be seen from the inequalities \begin{equation}\label{eq: hamiltonian and free energy}
        \frac{1}{N}\e \sum_{i<j}J_{ij}\la\sigma_i\sigma_j\ra  \le  \frac{1}{N}\e \max _{\sigma}\sum_{i<j}J_{ij}\sigma_i\sigma_j \le \frac{1}{N\beta '} \e \ln \sum_\sigma \exp\Bigl(\beta '\sum_{i<j}J_{ij}\sigma_i\sigma_j\Bigr).
    \end{equation} Since the limiting free energy exists for any $\beta'>0$ (see, e.g., Theorem \ref{superadditivity} or Theorem \ref{main:thm2}), \eqref{eq: hamiltonian and free energy} implies  \begin{equation}\label{eq: expectation of Hamiltonian bounded in beta}
    \sup_{\beta>0}\limsup_{N \to \infty}\e \frac{1}{N}\sum_{i<j} J_{ij} \la \sigma_i \sigma_j \ra \le C_\alpha    
    \end{equation}
     for some constant $C_\alpha>0$.
    A simple computation shows $\gamma_1 = \alpha \beta^{\alpha-1}\int_{0}^\infty  \tanh(x)x^{-\alpha}dx$ and $|R_L|\le 1$. In view of \eqref{eq: expectation of Hamiltonian bounded in beta}, for any $k\ge 1$, sending $\beta \to \infty$ in \eqref{eq: limit of R^2} yields \begin{equation}\label{eq: beta to infty: R_L^2}
            \lim_{\beta \to \infty} \liminf_{N \to \infty} \e \la R_{L}^2 \ra = \lim_{\beta \to \infty} \limsup_{N \to \infty} \e \la R_{L}^2 \ra =1.
    \end{equation}

    Fix $k\ge 1.$
      Since $\lim_{N\to\infty}c_k = \gamma_k$ for all $k\ge 1$, it is elementary to check using integration by parts that 
	\begin{equation*} 
	p_{2k}\colonequals \lim_{N\to\infty}\Prob(L=2k) =\frac{\gamma_{2k-1}-\gamma_{2k+1}}{\gamma_1}=\frac{\alpha \int_0^\infty \tanh^{2k}( x)x^{-(\alpha+1)}dx}  {2k\int_{0}^\infty \tanh( x)x^{-\alpha} dx}>0,\,\,\forall k\geq 1.
	\end{equation*}
 Using the fact that $|R_{2\ell}|\le 1$ for all $\ell\ge1$, it holds that \begin{align*}
     \e \la R_L^2\ra &= \p (L=2k)\e \la R_{2k}^2\ra + \sum_{\ell\ne k}\p(L=2\ell)\e \la R_{2\ell}^2\ra
     \\&\le \p (L=2k)\e \la R_{2k}^2\ra + (1- \p(L=2k) ).
 \end{align*}
Taking $N\to\infty$, \begin{align}
    \liminf_{N\to\infty}\e \la R_L^2\ra &\le p_{2k}\liminf_{N\to\infty}\e\la R_{2k}^2\ra +1-p_{2k} = 1-p_{2k}\bigl(1-\liminf_{N\to\infty}\e\la R_{2k}^2\ra\bigr)\le 1. \label{eq: R_L to R_2k}
\end{align} 
Note that $p_{2k}$ does not depend on $\beta$. Then, in view of \eqref{eq: beta to infty: R_L^2}, sending $\beta \to\infty$ in \eqref{eq: R_L to R_2k} yields the desired result.

\subsection{Moment Computations}

This subsection is devoted to computing the first and second moments of a modified bond overlap that will be used later when we turn to the proof of the concentration of  the bond overlap $Q_K$ in the next subsection.

\begin{proposition}\label{bondoverlap:firstmoment}
	Assume $1<\alpha<2$ and $\beta < \beta_\alpha$.
	Then, for any $K>0$,
	\begin{equation}\label{firstmoment}
		\lim_{N\to\infty}\e \Bigl\la \frac{1}{N}\sum_{i<j}\1_{|J_{ij}|\ge K} \sigma^1_i\sigma^1_j\sigma^2_i\sigma^2_j\Bigr\ra = C_K
  \end{equation}
  and
  \begin{equation}
  	\label{secondmoment}\lim_{N\to\infty}\e \Bigl\la \Bigl(\frac{1}{N}\sum_{i<j}\1_{\{|J_{ij}|\ge K\}} \sigma^1_i\sigma^1_j\sigma^2_i\sigma^2_j\Bigr)^2 \Bigr\ra = C_K^2
	\end{equation}
 for
 $$
 C_K:=\frac{\alpha}{2} \int_{K}^\infty \frac{\hth^2(\beta x)}{x^{1+\alpha}}dx.
 $$
\end{proposition}

For the rest of this subsection, we establish this proposition. We begin with the first assertion.

\begin{proof}[\bf Proof of Equation \eqref{firstmoment} in Proposition \ref{bondoverlap:firstmoment}]
	Note that \eqref{siteoverlapzero} of Theorem~\ref{hightempoverlap} implies that for $\beta < \beta_\alpha$ and  for any $\ell\ge 1$, 
	$\lim_{N\to\infty}\e\la\sigma_1\sigma_2\ra^{2\ell}=0,$
	so that from \eqref{cavityineq}, $\lim_{N\to\infty}\e\la\sigma_1\sigma_2\ra_{12}^{2\ell}=0.$
	Observe that \begin{align*}
		\1_{\{|J_{12}|>K\}}\bigl|\la\sigma_1\sigma_2\ra^2-\hth^2(\beta J_{12})\bigr| 
		&\le 2\1_{\{|J_{12}|>K\}} \Bigl| \frac{\la\sigma_1\sigma_2\ra_{12}(1-\hth^2(\beta J_{12}))}{1+\la\sigma_1\sigma_2 \ra _{12}\hth(\beta J_{12})}\Bigr|
		\\&\le 2\1_{\{|J_{12}|>K\}} \bigl|\la\sigma_1\sigma_2\ra_{12}\bigr|\bigl(1+|\hth(\beta J_{12})|\bigr).
	\end{align*}
	Taking expectation and using \eqref{eq:probconv}, we have  \[\e\1_{\{|J_{12}|>K\}}\bigl|\la\sigma_1\sigma_2\ra^2-\hth^2(\beta J_{12})\bigr|\le 4\e \bigl|\la\sigma_1\sigma_2\ra_{12}\bigr|\Prob(|J|>K)=o(1/N).\]
	Now, from symmetry among sites, we have \begin{equation*}
		\e \Bigl\la \frac{1}{N}\sum_{i<j}\1_{\{|J_{ij}|\ge K\}} \sigma^1_i\sigma^1_j\sigma^2_i\sigma^2_j\Bigr\ra = \frac{N-1}{2}\e \1_{\{|J_{12}|\ge K\}} \la\sigma_1\sigma_2\ra^2=\frac{N}{2}\e\hth^2(\beta J) \1_{\{|J|>K\}}+ o(1).
	\end{equation*}
The proof is complete since 
$\lim_{N\to\infty}N\e \1_{|J|\ge K} \hth^2(\beta J) = 2C_K$.
\end{proof}

The second moment computation requires more effort, which is based on three lemmas.

\begin{lemma}\label{lemma--3}
	Assume that $U_N,V_N,A_N,B_N$ are $[-1,1]$-valued random variables satisfying
	\begin{align*}
		U_N&=\frac{V_N+A_N\tanh(\beta J)}{1+B_N\tanh(\beta J)}.
	\end{align*}
	Then $	\lim_{N\to\infty}\e	|U_N-V_N|^2=0.$
\end{lemma}
\begin{proof}
	For notational clarity, we simply write $(U,V,A,B)$ for $(U_N,V_N,A_N,B_N).$
	From a direct computation, we have
	\begin{equation*}
		|U-V|=\Bigl|\frac{(A-VB)\hth(\beta J)}{1+B\hth(\beta J)}\Bigr|\leq \frac{2|\tanh(\beta J)|}{1-|\tanh(\beta J)|}.
	\end{equation*}
	Now, fix an $M>0$ and write
	\begin{equation*}
		|U-V|^2\leq \frac{4|\tanh(\beta J)|^2}{(1-|\tanh(\beta M)|)^2}\1_{\{|J|\leq M\}}+4 \1_{\{|J|\geq M\}}.
	\end{equation*}
  It holds that
\begin{equation*}
	\e\tanh^2(\beta  J)1_{\{| J|\leq M\}}\leq \frac{\alpha \beta^2}{N(2-\alpha)}M^{2-\alpha}, \quad \p(| J|\geq M)= \frac{1}{NM^\alpha}.
\end{equation*}
 These imply that for $N\ge 1$,
		\[ \e|U-V|^2 \leq \frac{4\alpha \beta^2M^{2-\alpha}}{N(2-\alpha)(1-\tanh(\beta M))^2}+\frac{4}{NM^{\alpha}}.
	\]
Sending $N\to\infty$ completes our proof.
\end{proof}

\begin{lemma}
	Let $U_N,V_N,W_N,A_N,B_N,C_N,D_N$ be $
	[-1,1]$-valued random variables and $J'$ is an independent copy of $J$.
	Assume that $V_N,A_N,J'$ are independent of $J$ and that $W_N,C_N,D_N$ are independent of $J,J'$. Also suppose that $U_N$ and $B_N$ can be written as
	$$
	U_N=\frac{V_N+A_N\tanh(\beta J)}{1+B_N\tanh(\beta J)}, \qquad  		B_N=\frac{W_N+C_N\tanh(\beta J')}{1+D_N\tanh(\beta J')}.
	$$
	\begin{enumerate}
		\item[$(i)$] If $\lim_{N\to\infty} \e |W_N|=0$ and $\lim_{N\to\infty}\e |D_N|=0,$ then
		\begin{align}\label{lemma:lem--2:eq2}
			\lim_{N\to\infty}N^2\bigl|\e\bigl(B_N^2-C_N^2\tanh^2(\beta J')\bigr)\1_{\{|J|,|J'|\geq K\}}\bigr|=0.
		\end{align}
		\item[$(ii)$] 	If $\lim_{N\to\infty} \e |W_N|=0$ and $\lim_{N\to\infty}\e |C_N|=0,$ then 
		\begin{align}\label{lemma:lem--2:eq1}
			\lim_{N\to\infty}N^2\bigl|\e \bigl(U_N^2-(V_N^2+A_N^2\tanh^2(\beta J))\bigr)\1_{\{|J|,|J'|\geq K\}}\bigr|=0.
		\end{align}
	\end{enumerate}
\end{lemma}
\begin{proof}
	For notational clarity,	we again ignore the dependence of our random variables on $N.$
 We also let $c\ge 1$ be some large absolute constants which may change line by line.
	From a direct computation, we have
		\begin{align}
  \begin{split}\label{add:eq--2}
		\bigl|U-(V+A\tanh(\beta J))\bigr|\leq \frac{2|B|}{1-|\tanh(\beta J)|}, \\
			\bigl|B-(W+C\tanh(\beta J'))\bigr|\leq \frac{2|D|}{1-|\tanh(\beta J')|}.
   \end{split}
	\end{align}
	We first prove \eqref{lemma:lem--2:eq2}.
	For $M>K$, set  $\mathcal{H} = \{ \min(|J|,|J'|) \ge K,   \ \max(|J|,|J'|) \ge M\}$. 
	We have
	\begin{align*}
		\p( \mathcal{H}) &\leq \p(|J|\geq M,|J'|\geq K)+\p(|J'|\geq M,|J|\geq K)= \frac{2}{N^2M^\alpha K^\alpha}.
	\end{align*}
	Using $|x^2-y^2|\le (|x|+|y|)|x-y|$ for  $x,y\in\mathbb R$ and \eqref{add:eq--2}, it follows that
	\begin{align*}
		\e \bigl|B^2-(W+C\tanh(\beta J'))^2\bigr|\1_{\{|J|,|J'|\geq K\}}
		&\leq 3\e \bigl|B-(W+C\tanh(\beta J'))\bigr|\1_{\{|J|,|J'|\geq K\}}\\
		&\leq \frac{6\e |D|}{1-\tanh(\beta M)}\p(K\le |J|,|J'|\leq M)+9 	\p( \mathcal{H})\\
		&\le\frac{6\e|D|}{N^2(1-\tanh(\beta M))K^{2\alpha}}+\frac{18}{N^2M^\alpha K^\alpha}.
	\end{align*}
	Since $\lim_{N\to \infty}\e |D|=0,$
	\begin{equation*}
		\limsup_{N\to\infty}N^2\bigl|\e\bigl(B^2-(W+C\tanh(\beta J'))^2\bigr)\1_{\{|J|,|J'|\geq K\}}\bigr|\leq 18 (MK)^{-\alpha}.
	\end{equation*}
	Since this holds for any $M>K,$ 
	\begin{equation*}
		\lim_{N\to\infty}N^2\bigl|\e \bigl(B^2-(W+C\tanh(\beta J'))^2\bigr)|\1_{\{|J|,|J'|\geq K\}}\bigr|=0.
	\end{equation*}
	Now using the symmetry of $J'$ and the independence between $J'$ and $W,C,J,$ this limit becomes
	\begin{equation*}
		\lim_{N\to\infty}N^2\bigl|\e \bigl(B^2-(W^2+C^2\tanh^2(\beta J')\bigr)\1_{\{|J|,|J'|\geq K\}}\bigr|=0.
	\end{equation*}
	From this, \eqref{lemma:lem--2:eq2} follows since, as $N \to \infty$, 
	\begin{align*}
		N^2\e W^2 \1_{\{|J|,|J'|\geq K\}}&=N^2\e W^2\cdot\p(|J|,|J'|\geq K)=  K^{-2\alpha} \e W^2 \le K^{-2\alpha}  \e W \to 0.
	\end{align*}
	
	To show \eqref{lemma:lem--2:eq1},
	from the independence of $V,A,J'$ with $J$ and the symmetry of $J$, it suffices to prove
	\begin{equation*}
		\limsup_{N\to\infty}N^2\bigl|\e \bigl(U^2-(V+A\tanh(\beta J))^2\bigr)\1_{\{|J|,|J'|\geq K\}}\bigr|=0.
	\end{equation*}
	Using \eqref{add:eq--2}, for $M>K$, we similarly control 
	\begin{align*}
		\e \bigl|U^2-(V+A\tanh(\beta J))^2\bigr|\1_{\{|J|,|J'|\geq K\}} &\leq 3\e \bigl|U-(V+A\tanh(\beta J))\bigr|\1_{\{|J|,|J'|\geq K\}}\\
		&\leq \frac{6\p(K\le |J|\leq M)}{1-\tanh(\beta M)}\e |B|\1_{\{K\le |J'|\le M\}}+ 9	\p( \mathcal{H})\\
		&\le \frac{6\e(|W|+|C|)}{N^2(1-\tanh(\beta M))^2K^{2\alpha}}+\frac{18}{N^2M^\alpha K^\alpha},
	\end{align*}
where we used the definition of $B$ to bound it by $( |W| + |C| ) / (1 - \tanh (\beta |J'|))$. 
	From the assumption $\lim_{N\to\infty}\e (|W|+|C|)=0,$ we see that for any $M>K$, \begin{equation*}
		\limsup_{N\to\infty}N^2\bigl|\e \bigl(U^2-(V+A\tanh(\beta J))^2\bigr)\1_{\{|J|,|J'|\geq K\}}\bigr|\le \frac{18}{ M^\alpha K^\alpha}.
	\end{equation*}
	Since this holds for any $M>K$, we obtain \eqref{lemma:lem--2:eq1}.
	\end{proof}

\begin{lemma}\label{lemma--4}
Assume that $1<\alpha<2$ and $\beta<\beta_\alpha.$	Let $ i,j,k,l$ be distinct positive integers with $i<j$ and $k<l.$ For any $K>0,$ we have that 
	\begin{equation*}
		\lim_{N\to\infty}N^2\e \bigl(\la \sigma_i\sigma_j\sigma_k\sigma_l\ra^2-\tanh^2(\beta J_{ij})\tanh^2(\beta J_{kl})\bigr)\1_{\{|J_{ij}|\geq K,|J_{kl}|\geq K\}}=0.
	\end{equation*}
\end{lemma}

\begin{proof}
	Recall the notation $\la \cdot\ra_{ij}$ from the proof of Proposition \ref{prop1'}. Denote by $\la \cdot\ra_{ijkl}$ the Gibbs expectation associated to the original system for which $J_{ij}$ and $J_{kl}$ are set to zero. To ease our notation, for a sequence of random variables $(X_N)_{N\geq 1}$, we say that $X_N\approx 0$ if $\lim_{N\to\infty}\e |X_N|=0.$ 
	First of all, we claim that
	\begin{equation}\label{lemma--4:eq1}
		\la \sigma_i\sigma_j\ra_{ijkl}\approx 0,\la \sigma_k\sigma_l\ra_{ijkl}\approx 0,		\la \sigma_i\sigma_j\sigma_k\sigma_l\ra_{ijkl}\approx 0.
	\end{equation}
	To see this, recall that \eqref{siteoverlapzero} and \eqref{cavityineq} imply 
	$\la \sigma_i\sigma_j\ra_{ij}\approx 0.$
	Since
	\begin{equation}\label{lemma--3:proof:eq1}
		\la \sigma_i\sigma_j\ra_{ij}=\frac{\la \sigma_i\sigma_j\ra_{ijkl}+\la \sigma_i\sigma_j\sigma_k\sigma_l\ra_{ijkl}\tanh(\beta J_{kl})}{1+\la \sigma_k\sigma_l\ra_{ijkl}\tanh(\beta J_{kl})},
	\end{equation}
	we see that from Lemma \ref{lemma--3}, $\la \sigma_i\sigma_j\ra_{ijkl}\approx 0$ and by symmetry, $\la \sigma_k\sigma_l\ra_{ijkl}\approx 0.$ On the other hand, we can write
	\begin{align}
		\label{lemma--3:proof:eq2}		\la \sigma_i\sigma_j\sigma_k\sigma_l\ra&=\frac{\la \sigma_i\sigma_j\sigma_k\sigma_l\ra_{ij}+\la \sigma_k\sigma_l\ra_{ij}\tanh(\beta J_{ij})}{1+\la \sigma_i\sigma_j\ra_{ij}\tanh(\beta J_{ij})},\\
		\label{lemma--3:proof:eq3}		\la \sigma_i\sigma_j\sigma_k\sigma_l\ra_{ij}&=\frac{\la \sigma_i\sigma_j\sigma_k\sigma_l\ra_{ijkl}+\la \sigma_i\sigma_j\ra_{ijkl}\tanh(\beta J_{kl})}{1+\la \sigma_k\sigma_l\ra_{ijkl}\tanh(\beta J_{kl})}.
	\end{align}
	Again, \eqref{siteoverlapzero} implies $\la \sigma_i\sigma_j\sigma_k\sigma_l\ra\approx 0$.
	It follows by applying Lemma \ref{lemma--3} twice that $\la \sigma_i\sigma_j\sigma_k\sigma_l\ra_{ijkl}\approx 0.$ This completes the proof of our claim.

	Next, we express \eqref{lemma--3:proof:eq2} and  \eqref{lemma--3:proof:eq1} respectively as
\[ \la \sigma_i\sigma_j\sigma_k\sigma_l\ra =\frac{V_N+A_N\tanh(\beta J)}{1+B_N\tanh(\beta J)}, \quad 
		\la \sigma_i\sigma_j\ra_{ij}=\frac{W_N+C_N\tanh(\beta J')}{1+D_N\tanh(\beta J')} \]
	for $J=J_{ij}$ and $J'=J_{kl}.$ Note that $V_N,A_N,J'$ are independent of $J$ and $W_N,C_N,D_N$ are independent of $J,J'$ and that $W_N,C_N,D_N\approx 0$ by \eqref{lemma--4:eq1}. It follows from  \eqref{lemma:lem--2:eq1} that
	\begin{equation}\label{add:eq--1}
		\lim_{N\to\infty}N^2\bigl|\e \bigl(\la \sigma_i\sigma_j\sigma_k\sigma_l\ra^2-\bigl(\la \sigma_i\sigma_j\sigma_k\sigma_l\ra_{ij}^2+\la \sigma_k\sigma_l\ra_{ij}^2\tanh^2(\beta J_{ij})\bigr)\bigr)\1_{\{|J_{ij}|\geq K,|J_{kl}|\geq K\}}\bigr|=0.
	\end{equation}
	Similarly, we express
	\begin{equation*}
		\la \sigma_k\sigma_l\ra_{ij}=\frac{\la \sigma_k\sigma_l\ra_{ijkl}+\tanh(\beta J_{kl})}{1+\la \sigma_k\sigma_l\ra_{ijkl}\tanh(\beta J_{kl})}=\frac{W_N'+C_N'\tanh(\beta J')}{1+D_N'\tanh(\beta J')},
	\end{equation*}
	and since $W_N',C_N',D_N'$ are independent of $J,J'$ and $W_N',D_N'\approx 0$ by \eqref{lemma--4:eq1}, from \eqref{lemma:lem--2:eq2},
	\begin{equation*}
		\lim_{N\to\infty}N^2\bigl|\e\bigl(\la \sigma_k\sigma_l\ra_{ij}^2-\tanh^2(\beta J_{kl})\bigr)\1_{\{|J_{ij}|,|J_{kl}|\geq K\}}\bigr|=0,
	\end{equation*}
	which implies that
	\begin{equation}\label{lemma--4:proof:eq2}
		\lim_{N\to\infty}N^2\bigl|\e\bigl(\la \sigma_k\sigma_l\ra_{ij}^2\tanh^2(\beta J_{ij})-\tanh^2(\beta J_{kl})\tanh^2(\beta J_{ij})\bigr)\1_{\{|J_{ij}|,|J_{kl}|\geq K\}}\bigr|=0.
	\end{equation}
	In addition, we can express \eqref{lemma--3:proof:eq3} as
	\begin{equation*}
		\la \sigma_i\sigma_j\sigma_k\sigma_l\ra_{ij}=\frac{W_N''+C_N''\tanh(\beta J')}{1+D_N''\tanh(\beta J')}.
	\end{equation*}
	Since $W_N'',C_N'',D_N''$ are independent of $J,J'$ and $W_N'',D_N''\approx 0$ by \eqref{lemma--4:eq1}, from \eqref{lemma:lem--2:eq2},
	\begin{equation*}
		\lim_{N\to\infty}N^2\bigl|\e\bigl(\la \sigma_i\sigma_j\sigma_k\sigma_l\ra_{ij}^2-\la\sigma_i\sigma_j\sigma_j\sigma_k\ra_{ijkl}^2\tanh^2(\beta J_{kl})\bigr)\1_{\{|J_{ij}|,|J_{kl}|\geq K\}}\bigr|=0.
	\end{equation*}
	Here, since
	\begin{equation*}
		\lim_{N\to\infty}	N^2\e \la\sigma_i\sigma_j\sigma_j\sigma_k\ra_{ijkl}^2\tanh^2(\beta J_{kl})\1_{\{|J_{ij}|,|J_{kl}|\geq K\}}\le \lim_{N\to \infty} K^{-2\alpha}\e \la\sigma_i\sigma_j\sigma_j\sigma_k\ra_{ijkl}^2=0,
	\end{equation*}
	we also see that
	\begin{equation*}
		\lim_{N\to\infty}	N^2\e\la \sigma_i\sigma_j\sigma_k\sigma_l\ra_{ij}^2 \1_{\{|J_{ij}|,|J_{kl}|\geq K\}}=0.
	\end{equation*}
	This together with \eqref{add:eq--1} and \eqref{lemma--4:proof:eq2} completes our proof.
\end{proof}

\begin{proof}[\bf Proof of Equation \eqref{secondmoment} in Proposition \ref{bondoverlap:firstmoment}]
	Note \begin{equation*}
		\Bigl\la  \Bigl(\frac{1}{N}\sum_{i<j}  \1_{\{|J_{ij}|\ge K\}}\sigma_i^1\sigma_j^1\sigma_i^2\sigma_j^2\Bigr)^2\Bigr\ra=\frac{1}{N^2}\sum_{i<j,\ k<l} \1_{\{|J_{ij}|,|J_{kl}|\ge K\}}\la \sigma_i\sigma_j\sigma_k\sigma_l\ra ^2 .
	\end{equation*}
	There are at most $O(N^3)$ many cases where at least two of the indices from $i,j,k,l$ such that $i<j$ and $k<l$ coincide.
	Since  $ \p( |J_{ij}|,|J_{kl}|\ge K )  = O(N^{-2}),$  it suffices to show that \[ \lim_{N\to\infty}\frac{1}{N^2}\sum_{i<j,\ k<l \text{ distinct}}\e \1_{\{|J_{ij}|,|J_{kl}|\ge K\}}\la \sigma_i\sigma_j\sigma_k\sigma_l\ra ^2 = C_K^2. \]
	Since there are total $N(N-1)(N-2)(N-3)/4$ many choices of $i<j$ and $k<l$ such that the indices are all distinct, we deduce, using the symmetry among sites, that 
	\begin{align*}
		&\frac{1}{N^2}\sum_{i<j,\ k<l\,\mathrm{distinct}}\e \1_{\{|J_{ij}|,|J_{kl}|\ge K\}}\la \sigma_i\sigma_j\sigma_k\sigma_\ell\ra ^2\\
  & =  \frac{N(N-1)(N-2)(N-3)}{4N^4}N^2 \e \1_{\{|J_{12}|,|J_{34}|\ge K\}}\la \sigma_1\sigma_2\sigma_3\sigma_4\ra ^2.
	\end{align*}
	From Lemma \ref{lemma--4},  the right-hand side of the  equation above converges to \begin{equation*}
		\tfrac{1}{4} \lim_{N\to\infty}N^2\e \hth^2(\beta J_{12})\hth^2(\beta J_{34})\1_{\{|J_{12}|\ge K\}}\1_{\{|J_{34}|\ge K\}}=C_K^2.
	\end{equation*}
\end{proof}

\subsection{Proof of Equation (\ref{bondoverlapnonzero}) in Theorem \ref{hightempoverlap}}\label{subsection:add:add3}

First of all, recall that $M=\sum_{i<j}\mathbbm{1}_{\{| J|\geq K\}}\sim \mbox{Binomial}(N(N-1)/2,K^{-\alpha}/N)$.
It is elementary to check that as $N\to\infty,$ $M/N\to 2^{-1}K^{-\alpha}$ in $L^2$, which implies that
\begin{equation}
    \label{extra:eq2}
    \frac{N}{M}\1_{\{M\ge1\}} \stackrel{p}{\to} 2K^{\alpha}.
\end{equation}
	Now, from Proposition \ref{bondoverlap:firstmoment}, we have \[\lim_{N\to\infty}\e \Bigl\la \Bigl(\frac{1}{N}\sum_{i<j}\1_{\{|J_{ij}|\ge K\}} \sigma^1_i\sigma^1_j\sigma^2_i\sigma^2_j - C_K \Bigr)^2 \Bigr\ra =0.\]
From this, we readily have \begin{align}
	\label{extra:eq3}\Bigl\la\frac{1}{N}\sum_{i<j}\1_{\{|J_{ij}|\ge K\}} \sigma^1_i\sigma^1_j\sigma^2_i\sigma^2_j\Bigr\ra  &\stackrel{L^2}{\to} C_K,\\
	\label{extra:eq4}	\Bigl\la \Bigl(\frac{1}{N}\sum_{i<j}\1_{\{|J_{ij}|\ge K\}} \sigma^1_i\sigma^1_j\sigma^2_i\sigma^2_j\Bigr)^2\Bigr\ra  &\stackrel{L^1}{\to} C_K^2.
	\end{align}
Consequently, from \eqref{extra:eq3} and writing \begin{equation*}
		Q_K= \frac{N}{M}\1_{\{M\ge 1\}} \cdot\frac{1}{N}\sum_{i<j}\1_{\{|J_{ij}|\ge K\}} \sigma^1_i\sigma^1_j\sigma^2_i\sigma^2_j,
	\end{equation*} 
	it follows from \eqref{extra:eq2} that \[\la Q_K\ra\xrightarrow[]{p}\alpha K^{\alpha}\int_{K}^\infty \frac{\hth^2(\beta x)}{x^{1+\alpha}}dx. \]
	Since $|Q_K|\le 1$, from the dominated convergence theorem, we have \[    \lim_{N\to\infty}\e \la Q_K \ra = \alpha K^{\alpha}\int_{K}^\infty \frac{\hth^2(\beta x)}{x^{1+\alpha}}dx.\]
	With \eqref{extra:eq4}, an identical argument  shows \[\lim_{N\to\infty} \e \la Q_K^2 \ra= \Bigl(\alpha K^{\alpha}\int_{K}^\infty \frac{\hth^2(\beta x)}{x^{1+\alpha}}dx\Bigr)^2.\]
	These limits imply the desired concentration for the bond overlap.

\section{Establishing the Superadditivity of the Free Energy}

We establish the proof of Theorem \ref{superadditivity} in this section. First of all, we will introduce the VB and PVB models and show that their free energies are asymptotically the same. Next, we show that the free energy in the L\'evy model can be approximated by that of the PVB model with a properly chosen disorder and edge connectivity. Based on these, the proof of Theorem \ref{superadditivity} is presented in Subsection \ref{sec:Proof:superadditivity}.

\subsection{Viana-Bray Model and its Poissonian Generalization}\label{BVBmodel}
Let $\beta>0.$ Let $\gamma>0$ be a fixed parameter and let $y$ be a random variable with finite first moment. Let $(B_{ij})$ be i.i.d. Bernoulli$(\gamma/N)$ and $(y_{ij})$ and $(y_k)$ be i.i.d. sampled from $y$. For any $r\geq 0,$ let $\pi(r)$ be a Poisson random variable with mean $r$. Each time when we see $\pi(r),\pi(r'),\pi(r''),\ldots$, we shall assume that they are independent of each other even when $r=r'.$ Let $(I(1,k),I(2,k))$ be i.i.d. uniformly sampled from the set $\{(i,j):1\leq i<j\leq N\}$. We assume that these random variables are all independent of each other. The Hamiltonians of the VB and PVB are defined as
\begin{equation*}
	-H_N^{\B}(\sigma)=\beta\sum_{1\leq i<j\leq N}B_{ij}y_{ij}\sigma_i\sigma_j
\end{equation*}
and
\begin{equation*}
	-H_N^{\VB}(\sigma)=\beta \sum_{k=1}^{\pi(\gamma N)}y_{k}\sigma_{I(1,k)}\sigma_{I(2,k)}.
\end{equation*}
Define the free energy of the above two models by $$F_N^{\B}=\frac{1}{N}\ln \sum_\sigma e^{-H_N^{\B}(\sigma)}$$ and $$F_N^{\VB}=\frac{1}{N}\ln \sum_{\sigma} e^{-H_N^{\VB}(\sigma)}.$$
Note that in comparison, the PVB model has been studied a bit more owing to the fact, as observed in \cite{Franz}, that many technicalities  in the Bernoulli model disappear in the PVB model. The following theorem nevertheless shows that the  free energies in these two models are asymptotically the same.

\begin{theorem} \label{whytruncation2}
	There exists a constant $C>0$ independent of $y$ and $\gamma$ such that
	\begin{align*}
		\bigl|\e F_N^{\B}-\e F_N^{\VB}\bigr|\leq \frac{C}{N}\Bigl(\gamma^3+\gamma\Bigl(1+\gamma +\frac{\gamma^2}{N}+\gamma^3\Bigr)\e|y|\Bigr),\,\,\forall N\geq 1.
	\end{align*}
\end{theorem}

For the rest of this subsection, we establish this theorem. We need the following lemma.

\begin{lemma}\label{poiber}
	Let $f:\mathbb N \cup \{0\} \to \mathbb R$ be a function such that $|f(1)-f(0)|\le d$ and $|f(2)-f(0)|\le d$ for some $d$.
	Assume that $X\sim\mathrm{Poi}(\Delta)$ and $Y\sim\mathrm{Ber}(\Delta)$ for some $\Delta >0$.
	We have that  \[|\e f(X)-\e f(Y)|\le \frac{3}{2} \bigl(\Delta ^3|f(0)| +d \Delta^2 + \Delta^3 \e|f(X+3)|\bigr).\]
\end{lemma}
\begin{proof}
	The idea is that the distributions of $X$ and $Y$ are similar except when their values are small. 
	Write \[
	\e f(X)-\e f(Y)=f(0)(e^{-\Delta} -(1-\Delta))+f(1)(e^{-\Delta}\Delta -\Delta)+f(2)e^{-\Delta}\frac{\Delta^2}{2!}+\sum_{\ell\ge3}f(\ell)e^{-\Delta}\frac{\Delta^\ell}{\ell!}.
	\]
	The first three terms can be rearranged as
	\begin{align*}
		&f(0)\bigl(e^{-\Delta}-(1-\Delta)+\Delta(e^{-\Delta}-1)+2^{-1}e^{-\Delta}\Delta^2\bigr)\\& \quad+(f(1)-f(0))\Delta (e^{-\Delta}-1)+2^{-1}(f(2)-f(0))e^{-\Delta}\Delta^2\\
		&=f(0)\bigl(e^{-\Delta}(1+\Delta+2^{-1}\Delta^2)-1\bigr)+(f(1)-f(0))\Delta (e^{-\Delta}-1)+2^{-1}(f(2)-f(0))e^{-\Delta}\Delta^2\\
		&= -f(0) e^{-\Delta}\sum_{\ell\ge 3}\frac{\Delta^\ell}{\ell!} +(f(1)-f(0))\Delta (e^{-\Delta}-1)+2^{-1}(f(2)-f(0))e^{-\Delta}\Delta^2.
	\end{align*}
	Therefore, noting that 
	$$
	e^{-\Delta}\sum_{\ell\ge 3}\frac{\Delta^\ell}{\ell!} \leq \frac{\Delta^3}{6},
	$$
	the absolute value of the first three terms is bounded above by
	\begin{align*}
		\frac{|f(0)|\Delta^3}{6}+\frac{3d\Delta^2}{2}.
	\end{align*}
	As for the last term, it can be	bounded by \[\biggl|\sum_{\ell\ge3}f(\ell)e^{-\Delta}\frac{\Delta^\ell}{\ell!}\biggr|\le \sum_{\ell\ge0}|f(\ell+3)|e^{-\Delta}\frac{\Delta^\ell}{\ell!}\frac{\ell!}{(\ell+3)!}\Delta^3\le \frac{\Delta^3}{6}\e|f(X+3)|.\]
	Combining everything, we have \[| \e f(X)-\e f(Y)|\le \frac{3}{2}( \Delta^3|f(0)|+d\Delta^2+\Delta^3 \e|f(X+3)|).\]
\end{proof}

\begin{proof}[\bf Proof of Theorem \ref{whytruncation2}]
	Let $(P_{ij})$ be i.i.d. Poisson$(2\gamma/ N)$ and let $(B_{ij})$  be i.i.d. Bernoulli$(2\gamma/N)$. These are independent of each other and other randomness. Denote by $\hat F_N^{\VB}$ the same as $F_N^{\VB}$, but the parameter $\gamma$ in $F_N^{\VB}$ is replaced by $\gamma(N-1)/N.$ Note that we can write $\pi(\gamma N){=}\pi(\gamma)+\pi(\gamma(N-1))$ in distribution so that
	\begin{equation*}
		\e F_N^{\VB}-\e \hat F_N^{\VB}=\frac{1}{N}\e\ln \Bigl\la \exp \Bigl(\sum_{k=1}^{\pi(\gamma)}y_k'\sigma_{I'(1,k)}\sigma_{I'(2,k)}\Bigr)\Bigr\ra',
	\end{equation*}
	where $(y_k')$ are i.i.d. copies of $y$, $(I'(1,k),I'(2,k))$ are i.i.d. copies of $(I(1,k),I(2,k))$, and they are independent of each other and other randomness. Also, $\la \cdot\ra'$ is the Gibbs expectation associated with the free energy $\e \hat F_N^{\VB}.$ Thus, 
	\begin{equation}\label{add:eq-9}
		\bigl|\e F_N^{\VB}-\e \hat F_N^{\VB}\bigr|\leq \frac{1}{N}\e\sum_{k=1}^{\pi(\gamma)}|y_k'|=\frac{\gamma \e|y|}{N}.
	\end{equation}
	Next, by using the thinning property of the Poisson random variable, we can write
	\begin{equation*}
		\hat F_N^{\VB}=\frac{1}{N}\e\ln\sum_{\sigma}\exp\Bigl(\beta \sum_{i<j}\Bigl(\sum_{k=1}^{P_{ij}}y_{ij}^k\Bigr)\sigma_{i}\sigma_{j}\Bigr),
	\end{equation*}
	where $(y_{ij}^k:i,j,k\geq 1)$ are i.i.d. copies of $y$ and independent of all other randomness.
	Let $(i_t,j_t)$ for $1\leq t\leq M:=N(N-1)/2$ be an enumeration of the pairs $\{(i,j):1\leq i<j\leq N\}.$ For any $0\leq s\leq M$ and $n\in \mathbb{N}\cup\{0\}$, define
	\begin{equation*}
		F_s(n)=\frac{1}{N}\e\ln\sum_{\sigma}\exp\Bigl(\beta \sum_{1\leq t\leq s-1}\Bigl(\sum_{k=1}^{P_{i_tj_t}}y_{i_tj_t}^k\Bigr)\sigma_{i_t}\sigma_{j_t}+\Bigl(\sum_{k=1}^{n}y_{i_sj_s}^k\Bigr)\sigma_{i_s}\sigma_{j_s}+\sum_{s+1\leq t\leq M}\Bigl(\sum_{k=1}^{B_{i_tj_t}}y_{i_tj_t}^k\Bigr)\sigma_{i_t}\sigma_{j_t}\Bigr).
	\end{equation*}
	From the assumption that $y$ is integrable, we can adapt the same argument as that of \eqref{add:eq-9} to show that $|F_s(0)-F_s(1)|$ and $|F_s(0)-F_s(2)|$ are uniformly $O(\e|y|/N)$ over all $1\leq s\leq M$. Now that we also have
	\begin{align*}
		|F_s(0)|&\leq \ln 2+\frac{\beta}{N} \e\Bigl[\sum_{1\leq t\leq s-1}\sum_{k=1}^{P_{i_tj_t}}|y_{i_tj_t}^k|+\sum_{s+1\leq t\leq M}\sum_{k=1}^{B_{i_tj_t}}|y_{i_tj_t}^k|\Bigr]\\
		&=\ln 2+ \frac{\beta(M-1)}{N^2}2\gamma\e|y|\\
		&\leq \ln 2+\beta \gamma \e|y|
	\end{align*}
	and for $N\geq 2,$
	\begin{align*}
		\e|F_s(P_{i_sj_s}+3)|&\leq \ln 2+\frac{\beta}{N} \e\Bigl[\sum_{1\leq t\leq s-1}\sum_{k=1}^{P_{i_tj_t}}|y_{i_tj_t}^k|+\sum_{k=1}^{P_{i_sj_s}+3}|y_{i_sj_s}^k|+\sum_{s+1\leq t\leq M}\sum_{k=1}^{B_{i_tj_t}}|y_{i_tj_t}^k|\Bigr]\\
	&=\ln 2+ \frac{\beta}{N}\Bigl(\frac{2\gamma M}{N}+3\Bigr)\e|y|\\
		&\leq \ln 2+\beta \Bigl(\gamma +\frac{3}{N}\Bigr)\e|y|.
	\end{align*}
	As a result, from Lemma \ref{poiber}, we readily have that there exists a constant $C>0$ independent of $N$ such that
	\begin{equation*}
		\bigl|\e F_s(P_{i_sj_s})-\e F_s(B_{i_sj_s})\bigr|\leq \frac{C}{N^3}\Bigl(\gamma^3+\gamma^2\Bigl(1+\frac{\gamma}{N}+\gamma^2\Bigr)\e|y|\Bigr).
	\end{equation*}
	Thus, noting that 
	$$
	\e F_s(P_{i_sj_s})=\e F_{s+1}(B_{i_{s+1}j_{s+1}})\,\,\mbox{and}\,\,\sum_{k=1}^{B_{i_tj_t}}y_{i_tj_t}^k\stackrel{d}{=}B_{i_tj_t}y_{i_tj_t},
	$$ we have
	\begin{equation*}
		\bigl|	\e\hat F_N^{\VB}-\e F_N^{\B}\bigr|\leq \sum_{s=0}^M\bigl|\e F_s(P_{i_sj_s})-\e F_s(B_{i_sj_s})\bigr|\leq\frac{C}{N}\Bigl(\gamma^3+\gamma^2\Bigl(1+\frac{\gamma}{N}+\gamma^2\Bigr)\e|y|\Bigr).
	\end{equation*}
	This together with \eqref{add:eq-9} completes our proof.
\end{proof}

\subsection{Free Energy Connection Between the L\'evy and Diluted Models}\label{Sec:levyDilutedModel}

For any $\eps>N^{-1/\alpha}$, denote by $g_\eps$ the distribution of $J$ conditionally on $|J|\geq \eps$, i.e., 
\begin{equation}
	\label{add:eq-11}
	g_\eps\sim \frac{\alpha\varepsilon^{\alpha}}{2}\frac{\1_{\{|x|\geq \varepsilon\}}}{|x|^{1+\alpha}}.
\end{equation}
Let 
\begin{equation}
	\label{input}
	y=g_\eps\,\,\mbox{and}\,\,\gamma=\eps^{-\alpha}.
\end{equation}
Denote by the corresponding free energy of the PVB model by $F_{N,\eps}^{\VB}.$ Also, denote $F_N=N^{-1}\ln Z_N.$ The following theorem shows that the free energy of the L\'evy model is asymptotically the same as that in the PVB model whenever $1<\alpha<2$.

\begin{theorem}\label{whytruncation}
	Assume $1<\alpha<2$.
   We have that
	\begin{align*}
		\lim_{\eps\downarrow 0}\lim_{N\to\infty}\bigl|\e F_N-\e F_{N,\eps}^{\VB}\bigr|=0.
	\end{align*}
\end{theorem}

To establish this theorem, we need the following lemma.

\begin{lemma}\label{whytruncation1}
	For any $N^{-1/\alpha}<\varepsilon$, we have that
	\begin{align}\label{add:eq-10}
		\Bigl|\e F_N-\frac{1}{N}\e\ln \sum_{\sigma}\exp\Bigl(\beta\sum_{i<j}J_{ij}\1_{\{|J_{ij}|\geq \eps\}}\sigma_i\sigma_j\Bigr)\Bigr|\leq \frac{\alpha\beta^2\varepsilon^{2-\alpha}}{(2-\alpha)}.
	\end{align}
\end{lemma}
\begin{proof}
	Let $\varepsilon>0$ be fixed. Set $A_{\varepsilon}=\{x\in \mathbb{R}:|x|<\varepsilon\}$. For any $t\in [0,1],$ set
	\begin{equation*}
		-H_{N,t}(\sigma)= \sum_{i<j}J_{ij}\bigl(t\1_{\{J_{ij}\in A_{\varepsilon}\}}+\1_{\{J_{ij}\notin A_{\varepsilon}\}}\bigr)\sigma_i\sigma_j.
	\end{equation*}
	Consider the free energy associated to $H_{N,t},$
	\begin{equation*}
		F_{N,t}=\frac{1}{N}\ln \sum_{\sigma}e^{-\beta H_{N,t}(\sigma)}
	\end{equation*}
	and set $\la \cdot\ra_{t}$ the Gibbs expectation associated with the Hamiltonian above.
	Then
	\begin{equation*}
		\frac{d}{dt}    F_{N,t}=\beta\sum_{i<j} J_{ij}\1_{\{J_{ij}\in A_{\varepsilon}\}}\la\sigma_i\sigma_j\ra_t.
	\end{equation*}
	Note that when $J_{ij}\in A_{\varepsilon}$,
	$
	J_{ij}\bigl(t\1_{\{J_{ij}\in A_{\varepsilon}\}}+\1_{\{J_{ij}\notin A_{\varepsilon}\}}\bigr)=tJ_{ij}.
	$
	Similar to the proof of Proposition~\ref{prop1'}, we can write
	\begin{align*}
		\la \sigma_i\sigma_j\ra_t&=\frac{\la \sigma_i\sigma_j\ra_{ t ,ij}+\tanh(\beta tJ_{ij})}{1+\la \sigma_i\sigma_j\ra_{t,ij}\tanh(\beta tJ_{ij})},
	\end{align*}
	where $\la \cdot\ra_{t,ij}$ is the Gibbs measure associated to $H_{N,t}$ with the replacement that $J_{ij}=0.$ Now, using \eqref{eq3},
	\begin{align*}
		\e|J_{ij}|\1_{\{J_{ij}\in A_\eps\}}\bigl|\la\sigma_i\sigma_j\ra_t-\la\sigma_i\sigma_j\ra_{t,ij}\bigr|&\leq 2\e\bigl[ |J_{ij}\tanh(\beta tJ_{ij})|;J_{ij}\in A_\eps\bigr]\\
		&\leq 2\beta\e\bigl[ |J_{ij}|^2;J_{ij}\in A_\eps\bigr]\\
		&\leq \frac{2\alpha \beta}{N}\int_{0}^\eps \frac{1}{x^{\alpha-1}}dx=\frac{2\alpha \beta\eps^{2-\alpha}}{N(2-\alpha)}.
	\end{align*}
	On the other hand, from the independence between $J_{ij}$ and $\la \sigma_i\sigma_j\ra_{t ,ij}$ and note that $|J_{ij}|$ has a finite mean, we have
	\begin{align*}
		\e J_{ij}\1_{\{J_{ij}\in A_\eps\}}\la\sigma_i\sigma_j\ra_{t,ij}&=\e J_{ij}\1_{\{J_{ij}\in A_\eps\}}\cdot \e\la\sigma_i\sigma_j\ra_{t,ij}=0.
	\end{align*}
	Combining these two displays together yields that
	\begin{equation*}
		\bigl|\e \bigl[J_{ij}\la \sigma_i\sigma_j\ra_t;J_{ij}\in A_{\varepsilon}\bigr]\bigr|\leq \frac{2\alpha\beta\varepsilon^{2-\alpha}}{N(2-\alpha)}.
	\end{equation*}
	Consequently,
	\begin{equation*}
		\Bigl|\frac{d}{dt}   \e  F_{N,t}\Bigr| \leq \frac{\alpha\beta^2\varepsilon^{2-\alpha}}{(2-\alpha)}
	\end{equation*}    
	and our proof is completed by noting that $|\e F_{N,1}-\e F_{N,0}|$ equals the left-hand side of  \eqref{add:eq-10}.
\end{proof}

\begin{proof}[\bf Proof of Theorem \ref{whytruncation}]
	In view of Lemma \ref{whytruncation1} and Theorem \ref{whytruncation2}, our assertion follows directly by noting that $Bg_\eps \stackrel{d}{=} J\1_{\{|J|\ge\eps\}},$ where $B\sim \mathrm{Ber}(\gamma/N)$ is independent of $g_\eps.$ 
\end{proof}

\subsection{Proof of Theorem \ref{superadditivity}}\label{sec:Proof:superadditivity}

\begin{lemma}[Superadditivity]\label{subadditivity}
	In the PVB model, if we assume that $y$ is symmetric with $\e |y|<\infty$ and  then for any $M,N\geq 2,$
	\begin{equation*}
		(M+N)\e F_{M+N}^{\VB}\geq M\e F_M^{\VB}+N\e F_N^{\VB}-6\beta\gamma \e|y|.
	\end{equation*}
\end{lemma}

\begin{proof}
	Let $M,N\geq 2$ be fixed. Let $(I(1),I(2))$, $(I(1,k),I(2,k))_{k\geq 1}$ be i.i.d. uniformly sampled from $\{(i,j):1\leq i<j\leq M+N\}$. Similarly, let $(I'(1),I'(2))$, $(I'(1,k),I'(2,k))_{k\geq 1}$, and $(I''(1),I''(2))$, $(I''(1,k),I''(2,k))_{k\geq 1}$ be i.i.d. uniformly sampled from $\{(i,j):1\leq i<j\leq M\}$ and  $\{(i,j):1\leq i<j\leq N\}$, respectively. Let $(y_k),(y_k'),(y_k''),y',y''$ be independent copies of $y.$ These random variables are all independent of each other. For any $\sigma=(\tau,\rho)\in \{-1,1\}^M\times\{-1,1\}^N,$ set
	\begin{equation*}
		H_{a,a',a''}(\sigma)=\beta \sum_{k\leq a}y_k\sigma_{I(1,k)}\sigma_{I(2,k)}+\beta \sum_{k\leq a'}y_k'\tau_{I'(1,k)}\tau_{I'(2,k)}+\beta \sum_{k\leq a''}y_k''\rho_{I''(1,k)}\rho_{I''(2,k)}.
	\end{equation*}
	For $0\leq t\leq 1,$ letting $a(t)=\pi(t\gamma(M+N))$, $a'(t)=\pi((1-t)\gamma M),$ and $a''(t)=\pi((1-t)\gamma N),$ consider the interpolated free energy,
	\begin{equation*}
		\Phi_t=\e \ln \sum_{\sigma\in \{-1,1\}^{M+N}}\exp H_{a(t),a'(t),a''(t)}(\sigma)
	\end{equation*}
	and denote the corresponding Gibbs expectation by $\la \cdot\ra_t.$
	Note that the Poisson random variable $\pi(s)$ satisfies the following equation,
	$$
	\frac{d}{ds}\e f(\pi(s))=\e f(\pi(s)+1)-\e f(\pi(s)).
	$$
	Therefore,
	\begin{equation*}
		\frac{d}{dt}\Phi_t =
		\gamma(M+N)\e \ln \bigl\la  e^{\beta y\sigma_{I(1)}\sigma_{I(2)}}\bigr\ra_t-\gamma M\e \ln \bigl\la e^{\beta y'\tau_{I'(1)}\tau_{I'(2)}}\bigr\ra_t-\gamma N\e \ln \bigl\la e^{\beta y''\rho_{I''(1)}\rho_{I''(2)}}\bigr\ra_t.
	\end{equation*}
	Now, using the identity \eqref{identity}, we can write
	\begin{align}
		\nonumber\e \ln \bigl\la  e^{\beta y\sigma_{I(1)}\sigma_{I(2)}}\bigr\ra_t&=\e\ln \cosh \beta y+\e \ln \bigl(1+\tanh(\beta y)\la \sigma_{I(1)}\sigma_{I(2)}\ra_t\bigr)\\
	\nonumber	&=\e\ln \cosh \beta y-\sum_{r=1}^\infty \frac{(-1)^r}{r}\e \tanh^r(\beta y)\cdot \e\la \sigma_{I(1)}\sigma_{I(2)}\ra_t^r\\
	\label{add:add:eq1}	&=\e\ln \cosh \beta y-\sum_{r=1}^\infty \frac{1}{2r}\e \tanh^{2r}(\beta y)\cdot \e \la \sigma_{I(1)}\sigma_{I(2)}\ra_t^{2r}.
	\end{align}
	Since
	\begin{align*}
		\e \la \sigma_{I(1)}\sigma_{I(2)}\ra_t^{2r}&=\frac{2}{(M+N)(M+N-1)}\sum_{1\leq i<j\leq M+N}	\e \la \sigma_{i}\sigma_{j}\ra_t^{2r}\\
		&=\frac{2}{(M+N)(M+N-1)}\sum_{1\leq i<j\leq M+N}\e \bigl\la \sigma_{i}^1\sigma_i^2\cdots\sigma_i^{2r}\cdot\sigma_{j}^1\sigma_j^2\cdots\sigma_j^{2r}\bigr\ra_t\\
		&=\frac{M+N}{M+N-1}\Bigl(\e \bigl\la R(\sigma^1,\ldots,\sigma^{2r})^2\bigr\ra_t-\frac{1}{M+N}\Bigr)\\
		&=A_{2r}-\frac{1}{M+N-1}\bigl(1-A_{2r})
	\end{align*}
	for $A_{2r}:=\e \la R(\sigma^1,\ldots,\sigma^{2r})^2\ra_t,$
	it follows that
	\begin{align*}
		\e \ln \bigl\la  e^{\beta y\sigma_{I(1)}\sigma_{I(2)}}\bigr\ra_t&=\e\ln \cosh \beta y-\sum_{r=1}^\infty \frac{1}{2r}\e\hth^{2r}(\beta y) \Bigl(A_{2r}-\frac{1}{M+N-1}\bigl(1-A_{2r})\Bigr).
	\end{align*}
	Similarly, for $A_{r}':=\e \la R(\tau^1,\ldots,\tau^{2r})^2\ra_t$ and $A_{r}'':=\e \la R(\rho^1,\ldots,\rho^{2r})^2\ra_t,$
	\begin{align*}
		\e \ln \bigl\la  e^{\beta y'\tau_{I'(1)}\tau_{I'(2)}}\bigr\ra_t&=\e\ln \cosh \beta y-\sum_{r=1}^\infty \frac{1}{2r}\e\hth^{2r}(\beta y) \Bigl(A_{2r}'-\frac{1}{M-1}\bigl(1-A_{2r}')\Bigr),\\
		\e \ln \bigl\la  e^{\beta y''\rho_{I''(1)}\rho_{I''(2)}}\bigr\ra_t&=\e\ln \cosh \beta y-\sum_{r=1}^\infty \frac{1}{2r}\e\hth^{2r}(\beta y) \Bigl(A_{2r}''-\frac{1}{N-1}\bigl(1-A_{2r}'')\Bigr).
	\end{align*}
	Putting these together yields that
	\begin{equation*}
		\frac{d}{dt}\Phi_t =-\gamma \sum_{r=1}^\infty \frac{\e\tanh^{2r}(\beta y)}{2r}\e \la \Delta_{2r}+\Gamma_{2r} \ra_t,
	\end{equation*}
	where $\Delta_r:=(M+N)A_{r}-MA_{r}'-NA_{r}''$ and 
	$$
	\Gamma_r:=-\frac{M+N}{M+N-1}\bigl(1-A_{r})+\frac{M}{M-1}\bigl(1-A_{r}')+\frac{N}{N-1}\bigl(1-A_{r}'').
	$$
	Here, since $$
	R(\sigma^1,\cdots,\sigma^{2r})=\frac{M}{M+N}R(\tau^1,\cdots,\tau^{2r})+\frac{N}{M+N}R(\rho^1,\cdots,\rho^{2r}),
	$$
	from Jensen's inequality, $\Delta_r\leq 0.$ On the other hand, obviously $|\Gamma_r|\leq 6.$
	Thus,
	\begin{align*}
		\frac{d}{dt}\Phi_t&\geq -6\gamma \sum_{r=1}^\infty \frac{\e\tanh^{2r}(\beta y)}{2r}=-3\gamma\e\ln \frac{1}{1-\tanh^2(\beta y)}=-6\gamma \e \ln\cosh(\beta y)\geq -6\beta \gamma \e|y|,
	\end{align*}
	where we used the formula $\cosh^{-2}x=1-\tanh^2 x$ and the inequality $\ln \cosh x\leq |x|.$ Our assertion follows by noting that $F_1=(M+N)\e F_{M+N}^{\VB}$ and $F_0=M \e F_M^{\VB}+N\e F_N^{\VB}.$

\end{proof}

\begin{proof}[\bf Proof of Theorem \ref{superadditivity}]
	Recall $(y,\gamma)$ from \eqref{input}. Note that
	\begin{equation*}
		\gamma\e |y| =\alpha\int_\eps^\infty \frac{1}{x^{\alpha}}dx=\eps^{-(\alpha-1)}.
	\end{equation*}
	Let $M,N$ be large enough and satisfy $N/2\leq M\leq 2N.$ Take $\eps=N^{-1/(1+3\alpha)}$. By applying Lemma~\ref{subadditivity}, we have
	\begin{equation*}
		(M+N)\e F_{M+N,\eps}^{\VB}\geq M\e F_{M,\eps}^{\VB}+N\e F_{N,\eps}^{\VB}-C(M+N)^{\frac{\alpha-1}{1+3\alpha}}.
	\end{equation*} 
	On the other hand, using Theorem \ref{whytruncation2} and Lemma \ref{whytruncation1} gives
	\begin{equation*}
		\bigl|\e F_{M+N}-\e F_{M+N,\eps}^{\VB}\bigr| \leq  \frac{C'}{(M+N)^{\frac{2-\alpha}{1+3\alpha}}}
	\end{equation*}
	and
	\begin{align*}
		\bigl|\e F_{M}-\e F_{M,\eps}^{\VB}\bigr|&\leq  \frac{C'}{M^{\frac{2-\alpha}{1+3\alpha}}}\leq \frac{C''}{(M+N)^{\frac{2-\alpha}{1+3\alpha}}},\\	
		\bigl|\e F_{N}-\e F_{N,\eps}^{\VB}\bigr|&\leq  \frac{C'}{N^{\frac{2-\alpha}{1+3\alpha}}}\leq \frac{C''}{(M+N)^{\frac{2-\alpha}{1+3\alpha}}}.
	\end{align*}
	Here, $C,C',C''$ are universal constants independent of $M,N.$ Combining these together yields that
	\begin{equation*}
		(M+N)\e F_{M+N} \geq M\e F_M+N\e F_N-(C+C'+2C'')\phi(M+N),
	\end{equation*}
	where  $
	\phi(x):=x^{1-(2-\alpha)/(1+3\alpha)}.
	$
	Here, $\phi$ satisfies the de Bruijn–Erd\H{o}s condition, namely, $\phi$ is increasing and $\int_1^\infty x^{-2}\phi(x)dx<\infty.$ Therefore, from \cite[Theorem 23]{de-Bruijn1952}, we conclude that $\lim_{N\to\infty}\e F_N$ converges. Now, Theorem \ref{con} completes our proof.
\end{proof}

\section{Panchenko's Invariance and Free Energy Representation}\label{VBM}

Recall the VB model introduced in Subsection \ref{BVBmodel}. In this section, we shall review its invariant equations for the spin distributions and the variational representation for the limiting free energy obtained in \cite{panchenko2013}. Let $c_N$ satisfy $c_N\to\infty,c_N/N\to 0$, and $|c_{N+1}-c_N|\to 0.$ For $r>0,$ denote by $(\pi_l(r))_{l\geq 1}$ i.i.d. Poisson random variables with mean $r.$ Let $(y_{kl})_{k,l\geq 1}$ be i.i.d. copies of $y$ and $(J(k,l))_{k,l\geq 1}$ be i.i.d. uniform on $\{1,2,\ldots,N\}.$ All these and the randomness defined in Subsection \ref{BVBmodel} are independent of each other.
Set 
$$
-H_{N}^{\mbox{\tiny pert}}(\sigma)=\sum_{l\leq \pi(c_N)}\ln 2\cosh\Bigl(\sum_{k\leq \pi_l(\gamma)}y_{kl}\sigma_{J(k,l)}\Bigr).
$$
In order to derive the invariance equations,  \cite{panchenko2013} redefined $H_N^{\VB}$ as
\begin{equation}\label{perturbedVB}
	-H_{N}^{\VB}(\sigma)=\beta\sum_{k=1}^{\pi(\gamma N)}y_{k}\sigma_{I(1,k)}\sigma_{I(2,k)}-H_{N}^{\mbox{\tiny pert}}(\sigma).
\end{equation}
Denote by $G_N^{\VB}$ and $F_N^{\VB}$ the Gibbs measure and free energy associated with this modified Hamiltonian, respectively. Note that since $c_N/N\to 0$, $H_N^{\mbox{\tiny pert}}$ can be understood as a small perturbation of the original PVB Hamiltonian.  In the limit, this perturbation does not change the free energy at all, but as one shall see, it does influence the structure of the Gibbs measure significantly and utilize the invariance equations for the spin distributions.

Denote by $\sigma,\sigma^1,\sigma^2,\ldots$ i.i.d. samplings from the Gibbs measure $G_N^{\VB}$ and by $\la \cdot\ra^{\VB}$ the Gibbs expectation. For each $N\geq 1$, we can extend $(\sigma^l)_{l\geq 1}\subset \{-1,1\}^N$ to be $(\bsig^{l})_{\ell\geq 1}\subseteq \{-1,1\}^\infty$ by defining
$$
\bsig_i^{l}=\sigma_i^l,\,\,1\leq i\leq N\,\,\mbox{and}\,\,\bsig_i^{l}=0,\,\,\forall i>N.
$$
Let $\nu_N$ be the distribution induced by $(\bsig^l)_{l\geq 1}$ with respect to $\e\la\cdot\ra^{\VB}.$ Let $\nu$ be the weak limit of $(\nu_N)_{N\geq 1}$ along certain subsequence if necessary. Note that  the distribution of $(\sigma_{i}^{l}:1\leq i\leq N,l\geq 1)$ under $\e\la\cdot\ra^{\VB}$ possesses two symmetries in the sense that it is invariant with respect to any finite permutation of $l$'s and with respect to any permutation of $1\leq i\leq N.$  With these, we readily see that the array $(s_{i}^l:i,l\geq 1)$ generated by $\nu$ is also symmetric under any finite permutations with respect to columns and rows. From the Aldous-Hoover representation \cite{Aldous1985,Hoover1982}, this array can be expressed as
$s_i^l=\sigma(w,u_l,v_i,x_{i,l})$ for some $\sigma\in \mathcal{S}$ and i.i.d. $w,u_l,v_i,x_{i,l}$ uniform random variables on $[0,1].$ Denote by $\mathcal{M}$  the collection of all $\sigma$'s in the Aldous-Hoover representation for all possible weak limits of $(\nu_N)_{N\geq 1}$.

For any $I\in\mathbb{N}^r$ for $r\geq 1,$ denote by $y_I$ and $\hat y_I$ i.i.d. copies from $y$. From Subsection \ref{sec:inv+var}, recall that $v_I,\hat v_I,x_I,\hat x_I$ are i.i.d. uniform on $[0,1]$ for any $r\geq 1$ and $I \in \mathbb{N}^r$, that $\delta$ and $(\delta_i)_{i\geq 1}$ be i.i.d. Rademacher random variables, and that these are all independent of each other and other randomness. For any $\sigma\in \mathcal{S}$, recall $s_I$ and $\hat s_I$ from \eqref{add:sec1.1:eq1}. As in Subsection \ref{sec:inv+var}, consider arbitrary $n,m,q,r\geq 1$ with $n\leq m.$ Let $C_l\subseteq \{1,\ldots,m\}.$ Denote the collection of the cavity and non-cavity spin coordinates, respectively, by 
\begin{equation*}
	C_l ^1 =C_l \cap \{1,\dots,n\},\,\,C_l ^2 = C_l \cap \{n+1,\dots,m\}.
\end{equation*}
Define two functionals $\mathcal{U}_{l}$ and $\mathcal{V}$ on $\mathcal{S}$ by
\begin{equation*}
\mathcal{U}_l(\sigma)=\e_{u,\delta,x}\Bigl(\prod_{i \in C_l^1}\delta_i\Bigr)\exp\Bigl(\beta\sum_{i\leq n}\sum_{k\leq \pi_i(\gamma)}y_{i,k}s_{i,k}\delta_i\Bigr)\Bigl(\prod_{i\in C_l^2}s_i\Bigr)\exp\Bigl( \beta\sum_{k\leq r}\hat y_k\hat s_{1,k}\hat s_{2,k}\Bigr)
\end{equation*}
and
\begin{equation*}
	\mathcal{V}(\sigma)=\e_{u,\delta,x}\exp\Bigl(\beta\sum_{i\leq n}\sum_{k\leq \pi_i(\gamma)}y_{i,k}s_{i,k}\delta_i\Bigr)\exp \Bigl(\beta\sum_{k\leq r}\hat y_k\hat s_{1,k}\hat s_{2,k}\Bigr).
\end{equation*}
The following theorem is Panchenko's invariant principle for the spin distribution $(s_i^l)_{i,l\geq 1}.$

\begin{theorem}[Theorem 1 in \cite{panchenko2013}]\label{pan:cavity} For any $\sigma\in \mathcal{M},$ we have
	\begin{equation}\label{pan:inv}
		\e\prod_{l\leq q}\prod_{i\in C_l}s_i^l =\e\prod_{l\leq q}\e_{u,x}\prod_{i\in C_l}s_i=\e\frac{\prod_{l\leq q}\mathcal{U}_l(\sigma)}{\mathcal{V}(\sigma)^q}.
	\end{equation}
\end{theorem}
Let $\mathcal{M}_{\mathrm{inv}}$ be the collection of $\sigma\in \mathcal{S}$ such that  all of the invariance equations in \eqref{pan:inv} hold.  From Theorem \ref{pan:cavity}, we readily see that $\mathcal{M}\subseteq \mathcal{M}_{\mathrm{inv}}.$ For any $\sigma\in \mathcal{S},$ define
$$
\mathcal{Q}(\sigma)=\ln 2+\e\ln \e_{u,x}\cosh\Bigl(\beta\sum_{k\leq \pi(\gamma)}y_ks_{k}\Bigr)-\e \ln \e_{u,x}\exp\Bigl(\beta\sum_{k\leq \pi(\gamma/2)}y_ks_{1,k}s_{2,k}\Bigr).
$$
With this, the limiting free energy of the PVB model can be expressed as a variational formula in terms of $\mathcal{Q}.$

\begin{theorem}[Theorem 2 in \cite{panchenko2013}]\label{panchenko:freeenergy}
	We have that almost surely and in $L^1,$
	\begin{equation*}
	\lim_{N\to\infty}F_N^{\VB} =\inf_{\sigma\in \mathcal{M}}\mathcal{Q}(\sigma)=\inf_{\sigma\in \mathcal{M}_{\mathrm{inv}}}\mathcal{Q}(\sigma).
	\end{equation*}
\end{theorem}

Denote by $\mathcal{M}_\eps$ and $\mathcal{M}_{inv,\eps}$, the pair $\mathcal{M}$ and $\mathcal{M}_{\mathrm{inv}}$ according to the input $(y,\gamma)$ in \eqref{input}.
Recall $g_\eps$ for any $\eps>0$ from \eqref{add:eq-11}. Let $(g_{k,\eps}),$ $(g_{i,k,\eps}),$ and $(\hat g_{k,\eps})$  be i.i.d. copies of $g_\eps$ and be independent of each and everything else. For any $\sigma\in \mathcal{S},$ define
\begin{align*}
	\begin{split}
		{U}_{l,\eps,\lambda}(\sigma)&=\e_{u,\delta,x}\Bigl(\prod_{i \in C_l^1}\delta_i\Bigr)\Bigl(\prod_{i\leq n}\prod_{k\leq \pi_i(\gamma)}\bigl(1+\tanh(\beta g_{i,k,\eps})s_{i,k}\delta_i\bigr)\Bigr)\\
		&\qquad\qquad\Bigl(\prod_{i\in C_l^2}s_i\Bigr)\Bigl( \prod_{k\leq \pi(\lambda \gamma/2)}\bigl(1+\tanh(\beta \hat g_{k,\eps})\hat s_{1,k}\hat s_{2,k}\bigr)\Bigr)
	\end{split}
\end{align*}
and
\begin{equation*}
	{V}_{\eps,\lambda}(\sigma) =\e_{u,\delta,x}\Bigl(\prod_{i\leq n}\prod_{k\leq \pi_i(\gamma)}\bigl(1+\tanh(\beta g_{i,k,\eps})s_{i,k}\delta_i\bigr)\Bigr)\Bigl( \prod_{k\leq \pi(\lambda \gamma/2)}\bigl(1+\tanh(\beta \hat g_{k,\eps})\hat s_{1,k}\hat s_{2,k}\bigr)\Bigr).
\end{equation*}
In addition, we set the functional
\begin{equation}\label{add:eq-12}
	P_\eps(\sigma)=\e\ln \e_{u,\delta,x}\prod_{k=1}^{\pi(\gamma)}\bigl(1+\tanh(\beta g_{k,\eps}) s_k\delta\bigr)-\e \ln \e_{u,x}\prod_{k=1}^{\pi(\gamma/2)} \bigl(1+\tanh(\beta g_{k,\eps})s_{1,k}s_{2,k}\bigr).
\end{equation}
By using the identity \eqref{identity}, taking $r$ by $\pi(\lambda\gamma/2)$ for some $\lambda>0$ in Theorem \ref{pan:cavity}, and using the Wald identity for the compound Poisson distribution, we can rewrite Theorem \ref{panchenko:freeenergy} as a corollary.

\begin{corollary}\label{add:cor1}
	For any $\sigma\in \mathcal{M}_\eps,$ we have
	\begin{align*}
		\e\prod_{l\leq q}\prod_{i\in C_l}s_i^l&=\e\prod_{l\leq q}\e_{u,x}\prod_{i\in C_l}s_i=\e\frac{\prod_{l\leq q}U_{l,\eps,\lambda}(\sigma)}{V_{\eps,\lambda}(\sigma)^q}.
	\end{align*}
	In addition, almost surely and in $L^1,$
	\begin{equation*}
    \lim_{N\to\infty}F_{N,\eps}^{\VB}=\inf_{\sigma\in \mathcal{M_\eps}}Q_\eps(\sigma)=\inf_{\sigma\in \mathcal{M}_{inv,\eps}}Q_\eps(\sigma),
	\end{equation*}
	where 
	\begin{equation}
		\label{Qeps}
			Q_\eps(\sigma):=\ln 2+\frac{\alpha}{2}\int_\eps^{\infty}\frac{\ln\cosh(\beta x)}{x^{\alpha+1}}dx+P_\eps(\sigma).
	\end{equation}
\end{corollary}



\section{Establishing Invariance Principle and Variational Formula}
We will establish the proof of Theorem \ref{main:thm1}, Theorem \ref{main:thm2}, and Proposition \ref{hightemp} in this section. As a preparation, we need the uniform convergences and continuities of some related functionals, whose proofs require a heavy analysis and will be deferred to the Appendix.

\subsection{Uniform Convergence and Continuity of the Functionals}\label{ucf}

Recall the functionals $Q_\eps,U_{l,\eps,\lambda},$ and $V_{\eps,\lambda}$ introduced in the previous subsection. Also, recall the three functionals $Q,$ $U_{l,\lambda},$ and $V_\lambda$ introduced in Subsection \ref{sec:inv+var}. We have the following theorem


\begin{theorem}\label{invariancestep1}
	We have that
	\begin{equation}
		\label{invariancestep1:eq1}
		\lim_{\eps\downarrow 0}\sup_\sigma \e\Bigl|\frac{\prod_{l\le q} U_{l,\eps,\lambda}(\sigma)}{V_{\eps,\lambda}(\sigma) ^q}- \frac{\prod_{l\le q} U_{l,\lambda}(\sigma)}{V_\lambda(\sigma) ^q}\Bigr|=0
	\end{equation} 
	and
	\begin{equation}	\label{invariancestep1:eq2}
		\lim_{\eps\downarrow 0}\sup_\sigma \bigl|Q_{\eps}(\sigma)-Q(\sigma)\bigr|=0.
	\end{equation}
\end{theorem}

\begin{remark}
	\rm  It is easy to see that $U_{l,\eps,\lambda}(\sigma)$ and $V_{\eps,\lambda}(\sigma)$ are finite almost surely and $Q_\eps(\sigma)$ is finite since their definitions consist of Poissonian numbers of terms in products.  In contrast, as the definitions of $U_{l,\lambda}(\sigma)$, $V_{\eps,\lambda}(\sigma)$, and $Q_\eps(\sigma)$ involve infinite products, their finiteness require justifications, which are parts of the statement of Theorem \ref{invariancestep1}. 
\end{remark}

Next, we show that the functionals in Theorem \ref{invariancestep1} are continuous in the finite-dimensional distribution sense:

\begin{definition}\label{def:fdd}
	Let $(\sigma_{m})_{m\geq 1}\subset \mathcal{S}$ and $\sigma_0\in \mathcal{S}$. We say that $\sigma_m\to \sigma_0$ in the finite-dimensional distribution (f.d.d.) sense as $m\to \infty$ if the infinite array $s_{m}:=(s_{m,i}^l)_{i,l\geq 1}=(\sigma_m(w,u_l,v_i,x_{i,l}))_{i,l\geq 1}$ converges to $s_0:=(s_{0,i}^l)_{i,l\geq 1}=(\sigma_0(w,u_l,v_i,x_{i,l}))_{i,l\geq 1}$ for any finitely many entries.
\end{definition}

\begin{remark}\label{welldefiniteness}
	The space $\mathcal{S}$ is sequentially compact with respect to the f.d.d. convergence. To see this, consider an arbitrary sequence $(\sigma_m)_{m\geq 1}$ from $\mathcal{S}$ and denote by $s_{m}=(s_{m,i}^l)_{i,l\geq 1}$ the corresponding infinite array for each $m\geq 1.$ Since  the space of all probability measures on the space $[-1,1]^{\mathbb{N}\times\mathbb{N}}$ equipped with the metric $$d(x,y)=\sum_{i,l\geq 1}\frac{|x_{i,l}-y_{i,l}|}{2^{i+l}}$$ is sequentially compact, it follows that along some subsequence, $(s_{m_k})_{k\geq 1}$ converges to some infinite array $s_0=(s_{0,i}^l)_{i,l\geq 1}$ in the sense that $(s_{m_k,i}^l)_{i,l\leq r}$ converges to $(s_{0,i}^l)_{i,l\leq r}$ weakly as $k\to \infty$ for all $r\geq 1.$ Obviously the distribution of $s_0$ must be invariant under any permutations in finitely many rows and columns. With this, the Aldous-Hoover representation ensures that there exists some $\sigma_0\in \mathcal{S}$ such that in distribution $s_0=(\sigma_0(w,u_l,v_i,x_{i,l}))_{l,i\geq 1}.$ Therefore, $\sigma_{m_k}$ converges to $\sigma_0$ in the f.d.d. sense. 
\end{remark}

\begin{theorem}\label{invariancestep2}
	Assume that $\sigma_m\to\sigma_0 $ f.d.d. as $m\to\infty$. For any $\eps>0$ and $\lambda>0,$ we have that as $m\to\infty,$
	\[\lim_{m\to\infty}\e\frac{\prod_{l\le q}U_{l,\eps,\lambda}(\sigma_m) }{V_{\eps,\lambda}^q (\sigma_m)} =\e\frac{\prod_{l\le q}U_{l,\eps,\lambda}(\sigma_0) }{V_{\eps,\lambda}^q (\sigma_0)}\]
	and for any $\eta>0,$
	\begin{equation*}
		\lim_{m\to\infty}Q_{\eps}(\sigma_m)=Q_\eps(\sigma_0).
	\end{equation*}
\end{theorem}

As we have mentioned before, the proofs of these theorems require some subtle analysis of the functionals. Since the detailed arguments are not needed for the rest of the paper, we defer their proofs to the Appendices \ref{ProofStep1} and \ref{ProofStep2}, respectively.

\subsection{Proof of Theorems \ref{main:thm1} and \ref{main:thm2}}

We need a simple lemma.

\begin{lemma}\label{functional3epsilon}
	Let $\{T_\eps,T\}_{\eps>0}$ be a family of functionals  and let $\sigma \in \mathcal S$ and $(\sigma_\eps)_{\eps>0}\subset \mathcal{S}$. Assume that
	$
		\lim_{\eps\downarrow 0}\sup_{\sigma\in\mathcal S} |T_\eps (\sigma) - T(\sigma)|=0
	$
	and that for  any $\eta>0$, $\lim_{\eps\downarrow 0}T_\eta (\sigma_\eps) =T_\eta (\sigma).$
	Then $\lim_{\eps\downarrow 0}T_\eps(\sigma_\eps)=T(\sigma)$
	and
	$\lim_{\eps\downarrow 0}T(\sigma_\eps)=T(\sigma).$
\end{lemma}
\begin{proof}
	The uniform convergence readily implies   
	$\lim_{\eps,\eta\downarrow0}  \sup_{\sigma \in \mathcal S}|T_\eps (\sigma)-T_\eta (\sigma)|=0.$
	Fix $\eps,\eta>0$.
	Write
	\begin{align*}
		|T_\eps(\sigma_\eps)-T(\sigma)|&\le |T_\eps(\sigma_\eps)-T_\eta(\sigma_\eps)|+|T_\eta (\sigma_\eps)-T_\eta(\sigma)|+|T_\eta(\sigma)-T(\sigma)|\\
		&\le \sup_{\sigma \in \mathcal S}|T_\eps (\sigma)-T_\eta (\sigma)|+|T_\eta (\sigma_\eps)-T_\eta(\sigma)|+\sup_{\sigma \in \mathcal S}|T_\eps (\sigma)-T(\sigma)|.
	\end{align*}
	Taking $\eps\downarrow0$ and then $\eta\downarrow0$ yields the first assertion. Moreover, the second assertion follows from $$|T(\sigma)-T(\sigma_\eps)|\le |T(\sigma)-T_\eps (\sigma_\eps)|+|T(\sigma_\eps)-T_\eps (\sigma_\eps)|\le |T(\sigma)-T_\eps (\sigma_\eps)|+\sup_\sigma |T(\sigma)-T_\eps (\sigma)|.$$
\end{proof}

\medskip

{\noindent \bf Proof of Theorem \ref{main:thm1}.}
Our assertion follows directly by using Theorems \ref{invariancestep1}, Theorem \ref{invariancestep2}, and Corollary \ref{add:cor1} and applying Lemma \ref{functional3epsilon}.

\medskip

{\noindent \bf Proof of Theorem \ref{main:thm2}:} Recall $F_N=N^{-1}\ln Z_N.$ As we have already known from Theorem \ref{con} that $F_N$ is concentrated with respect to the $L^p$ norm for $1 \le p<\alpha,$ the validity of 
\begin{equation}\label{proof:add:eq4}
	\lim_{N\to\infty}\e F_N=\inf_{\sigma\in \mathcal{N}}Q(\sigma)=\inf_{\sigma\in \mathcal{N}_{\mathrm{inv}}}Q(\sigma)
\end{equation}
will be enough to complete our proof.
We establish the first equality in \eqref{proof:add:eq4} first.
From Theorem~\ref{whytruncation},
\begin{equation}\label{proof:add:eq3}
	\lim_{N\to\infty}\e F_N=\lim_{\eps\downarrow 0}\lim_{N\to\infty}\e F_{N,\eps}^{\VB}=\lim_{\eps\downarrow 0}\inf_{\sigma\in \mathcal{M}_\eps}Q_\eps(\sigma).
\end{equation}
Let $\sigma_0\in \mathcal{N}.$ Then there exists a sequence $(\sigma_{m})_{m\geq 1}$ for $\sigma_{m}\in \mathcal{M}_{\eps_m}$ and $\eps_m\downarrow 0$ such that $\sigma_{m}\to \sigma_0$ f.d.d. Therefore, from  Theorem \ref{invariancestep1}, Theorem \ref{invariancestep2}, and Lemma \ref{functional3epsilon}, 
\begin{equation*}
\limsup_{N\to\infty}\e F_N\leq \lim_{m\to\infty}Q_{\eps_m}(\sigma_{m})=Q(\sigma_0).
\end{equation*}
Since this holds for any $\sigma_0\in \mathcal{N},$ it follows that $\limsup_{N\to\infty}\e F_N\leq \inf_{\mathcal{N}}Q(\sigma).$ Conversely, from \eqref{proof:add:eq3}, for any $\eta>0,$ as long as $\eps$ is small enough,
$
\liminf_{N\to\infty}\e F_N\geq \inf_{\sigma\in \mathcal{M}_\eps}Q_\eps(\sigma)-\eta.
$
We can pick a sequence $(\sigma_{m})_{m\geq 1}$ for $\sigma_{m}\in \mathcal{M}_{\eps_m}$ and $\eps_m\downarrow 0$ that again converges to some $\sigma_0$ f.d.d. and that for all $m\geq 1,$
$
\inf_{\sigma\in \mathcal{M}_{\eps_m}}Q_{\eps_m}(\sigma)\geq Q_{\eps_m}(\sigma_{m})-\eta.
$
Hence, from Theorem \ref{invariancestep1}, Theorem \ref{invariancestep2},  and Lemma \ref{functional3epsilon} again,
\begin{equation*}
	\liminf_{N\to\infty}\e F_N\geq \lim_{m\to\infty}Q_{\eps_m}(\sigma_m)-2\eta=Q(\sigma_0)-2\eta.
\end{equation*}
Since this holds for any $\eta>0,$ the lower bound, $\liminf_{N\to\infty}\e F_N\geq Q(\sigma),$ follows, from which, since $\sigma_0\in \mathcal{N}$, we readily have that $\lim_{N\to\infty}\e F_N= \inf_{\mathcal{N}}Q(\sigma)$ and this establishes the first equality in \eqref{proof:add:eq4}.

As for the second inequality in \eqref{proof:add:eq4}, note that $\inf_{\sigma\in \mathcal{N}}Q(\sigma)\geq \inf_{\sigma\in \mathcal{N}_{\mathrm{inv}}}Q(\sigma).$
It suffices to show the reverse inequality, for which we follow the argument of Lemma 7 in \cite{panchenko2013}.  Recall that from Lemma 6 in \cite{panchenko2013} and the identity \eqref{identity}, for any function $\sigma\in \mathcal{S}$, we have
\begin{align*}
	\lim_{N\to\infty}\frac{1}{N}\e F_{N,\eps}^{\VB}& \leq \ln 2+\frac{\alpha}{2}\int_\eps^\infty \frac{\ln\cosh(\beta x)}{x^{\alpha+1}}dx+\frac{1}{n}\e \ln \mathcal{A}_{n,\eps}(\sigma)- \frac{1}{n}\e \ln \mathcal{B}_{n,\eps,1}(\sigma),
\end{align*}
and from Theorem \ref{whytruncation} and \eqref{proof:add:eq3}, passing to the limit gives 
\begin{equation*}
	\lim_{N\to\infty}\e F_{N}\leq \ln2 + \frac{\alpha}{2}\int_0^\infty \frac{\ln\cosh(\beta x)}{x^{\alpha+1}}dx+\frac{1}{n}\e \ln\mathcal{A}_{n}(\sigma)- \frac{1}{n}\e \ln \mathcal{B}_{n,1}(\sigma).
\end{equation*}
We claim that for any $\sigma\in \mathcal{N}_{\mathrm{inv}},$ 
\begin{align*} 
	\begin{split}
		\e\ln\frac{\mathcal A_{n+1}(\sigma)}{\mathcal A_n(\sigma)}&=\e\ln \mathcal A_1(\sigma)\,\,\mbox{and}\,\,
		\e\ln\frac{\mathcal B_{n+1}(\sigma)}{\mathcal B_n(\sigma)}=\e\ln \mathcal B_1(\sigma).
	\end{split}
\end{align*}
If these are valid, then we readily have $$ \inf_{\mathcal{N}_{\mathrm{inv}}}Q(\sigma)\leq\inf_{\mathcal{N}}Q(\sigma)=\lim_{N\to\infty}\e F_N\leq \inf_{\mathcal{N}_{\mathrm{inv}}}Q(\sigma)$$ and this establishes the second equality in \eqref{proof:add:eq4}. To prove our claim, we will only establish the proof for the first equation since the other one can be treated similarly. First of all, note that from Theorem \ref{parisistep1},  we have
\begin{equation*}
	\e A_1(\sigma)=	\lim_{\eps\downarrow 0}\e\ln \e_{u,x,\delta} \prod_{k:|\xi_{n+1,k}|\geq \eps} \bigl(1+\hth(\beta \xi_{n+1,k}) s_{n+1,k}\delta_{n+1}\bigr)
\end{equation*}
and a small modification of Theorem \ref{parisistep1} yields that
\begin{align*}
	\e A_{n+1}(\sigma)&=\lim_{\eps\downarrow 0}\e\ln \e_{u,x,\delta} \prod_{k:|\xi_{n+1,k}|\geq \eps} \bigl(1+\hth(\beta \xi_{n+1,k}) s_{n+1,k}\delta_{n+1}\bigr) \\
	&\qquad\qquad\qquad\qquad\prod_{i\le n}\prod_{k\ge 1} \bigl(1+\hth(\beta\xi_{i,k}) s_{i,k}\delta_{i}\bigr).
\end{align*}
From these, it suffices to show that for any $\eps>0,$
\begin{align*}
	&\e\ln \Bigl(\e_{u,x,\delta} \Bigl[\prod_{k:|\xi_{n+1,k}|\geq \eps} \bigl(1+\hth(\beta \xi_{n+1,k}) s_{n+1,k}\delta_{n+1}\bigr) \frac{\prod_{i\le n}\prod_{k\ge 1} (1+\hth(\beta\xi_{i,k}) s_{i,k}\delta_{i})}{\e_{u,x,\delta} \prod_{i\le n}\prod_{k\ge 1} (1+\hth(\beta\xi_{i,k}) s_{i,k}\delta_{i})}\Bigr]\Bigr)\\
	&=\e\ln \Bigl(\e_{u,x,\delta} \Bigl[\prod_{k:|\xi_{n+1,k}|\geq \eps} \bigl(1+\hth(\beta \xi_{n+1,k}) s_{n+1,k}\delta_{n+1}\bigr)\Bigr]\Bigr).
\end{align*}
Denote by $\e'$ the expectation conditionally on $(\xi_{n+1,k})_{k\geq 1}.$ Observe that conditionally on $(\xi_{n+1,k})_{k\geq 1},$ 
$$\prod_{k:|\xi_{n+1,k}|\geq \eps} \bigl(1+\hth(\beta \xi_{n+1,k}) s_{n+1,k}\delta_{n+1}\bigr)$$ 
is a finite product and it is uniformly bounded from above and is uniformly positive. Since $x\in [c,C]\to \ln x$ for some $c,C>0$ can be uniformly approximated by polynomials, to establish the equality above, we only need to show that for all integer $r\geq 1$,
\begin{align*}
	&\e'\Bigl[\Bigl(\e_{u,x,\delta}\Bigl[ \prod_{k:|\xi_{n+1,k}|\geq \eps} (1+\hth(\beta \xi_{n+1,k}) s_{n+1,k}\delta_{n+1}) \frac{\prod_{i\le n}\prod_{k\ge 1} (1+\hth(\beta\xi_{i,k}) s_{i,k}\delta_{i})}{\e_{u,x,\delta} \prod_{i\le n}\prod_{k\ge 1} (1+\hth(\beta\xi_{i,k}) s_{i,k}\delta_{i})}\Bigr]\Bigr)^r\Bigr]\\
	&=\e'\Bigl[\Bigl(\e_{u,x,\delta}\Bigl[ \prod_{k:|\xi_{n+1,k}|\geq \eps} (1+\hth(\beta \xi_{n+1,k}) s_{n+1,k}\delta_{n+1})\Bigr]\Bigr)^r\Bigr].
\end{align*}
Now, if we expand the first product on the left-hand side, the product on the right-hand side, and $(\cdot)^r$ on the two sides, then by comparison, we see that the equality above holds as long as we have that for any $I_1,\ldots,I_r\subseteq \{k:|\xi_{n+1,k}|\geq \eps\},$
\begin{equation*}
	\e'\frac{\prod_{l\leq r}U_{l,0}(\sigma)}{V_0(\sigma)^r}=\e'\prod_{l\leq r}\prod_{i\in I_l} s_i^l,
\end{equation*}
which follows from the assumption that $\sigma$ satisfies the invariance principle \eqref{main:thm1:eq1},
where
\begin{align*}
	U_{l,0}(\sigma)&=\e_{u,x,\delta}\Bigl(\prod_{i\le n}\prod_{k\ge 1} (1+\hth(\beta\xi_{i,k}) s_{i,k}\delta_{i})\Bigr)\Bigl(\prod_{k\in I_l}s_i\Bigr),\\
	V_0(\sigma)&=\e_{u,x,\delta}\Bigl(\prod_{i\le n}\prod_{k\ge 1} (1+\hth(\beta\xi_{i,k}) s_{i,k}\delta_{i})\Bigr).
\end{align*}
This ends the  proof of our claim.  

\subsection{Proof of Proposition \ref{hightemp}}

Let $0<\beta<\beta_\alpha$ be fixed. Recall the free energy $F_{N,\eps}^{\VB}(\beta)$ of the PVB model associated to the choice of parameter $(y,\gamma)$ in \eqref{input} in Section \ref{Sec:levyDilutedModel}. From \cite{Guerra2004}, it was already known that for any small enough $\eps,$ as long as $\beta$ satisfies 
$$\alpha\int_{\eps}^\infty \frac{\tanh^2(\beta x)}{x^{\alpha+1}}dx=2\gamma \e \tanh^2(\beta y)<1,$$
we have that 
\begin{equation}\label{PVB:overlap}
\lim_{N\to\infty}\e\la R_{1,2}^2\ra_\eps^{\VB}=0.
\end{equation}
In particular, recalling the definition of $\beta_\alpha$ from \eqref{beta_critical}, the above inequality is satisfied as long as $\eps$ is small enough. 

We claim that for small enough $\eps,$ any $\sigma\in \mathcal{M}_\eps$ satisfies $\e_x\sigma(w,u,v,x)=0.$ a.s.. If this is valid, then we readily have that $\e_x\sigma(w,u,v,x)=0$ a.s. for all $\sigma\in \mathcal{N}$ and this completes our proof. To establish our claim, let $\sigma\in \mathcal{M}_\eps$ and let $s_{i}^l=\sigma(w,u_l,v_i,x_{i,l})$ for all $i,l\geq 1.$ From \eqref{PVB:overlap} and using symmetry, we readily have
 $\e s_1^1s_1^2s_2^1s_2^2=0.$ Since we can write
\begin{align*}
	\e s_1^1s_1^2s_2^1s_2^2&=\e \bar s(w,u_1,v_1)\bar s(w,u_2,v_1)\bar s(w,u_1,v_2)\bar s(w,u_2,v_2)\\
	&=\e\bigl[ \bigl(\e_v\bar s(w,u_1,v)\bar s(w,u_2,v)\bigr)^2\bigr]
\end{align*} 
for $\bar s(w,u,v):=\int_0^1s(w,u,v,x)dx,$
it follows that a.s.,
$
	\e_v\bar s(w,u_1,v)\bar s(w,u_2,v)=0.
$
From the Lebesgue differentiation theorem, we have $\e_v\bar s(w,u,v)\bar s(w,u,v)=0$ holds a.s.  Thus, $\e \bar s(w,u,v)^2=0$ and $\bar s(w,u,v)=0$ a.s. This finishes the proof of our claim.

\section{Establishing the Limiting Free Energy and Edge Alignments}
\subsection{Proof of Theorem \ref{thm: generalized: alpha=1: free energy}}

From the asymptotic relations $\lim_{x\to\infty} x^{-1}\ln\ch(x)=1$ and $\lim_{x\to 0} 2 x^{-2} \ln \ch(x) =1$, one can check that the limit of the centering term in \eqref{eq: centering}  with scaling $N\ln N$ in case of $\alpha=1$ is \[\lim_{N\to\infty}\frac{1}{N\ln N} \frac{N}{2} \int_{N^{-1}}^{\frac{N-1}{2}}\frac{\ln \ch(\beta x)}{x^2}dx =\frac{\beta}{2}.\]  
Then, from Theorem \ref{fluctuationfreeenergy}, in case of $\alpha=1$, we have that for $ \beta< \beta_\alpha,$  \begin{equation}\label{eq:alpha=1}
        \frac{1}{N \ln N} \ln Z_N \stackrel{p}{\to} \frac{\beta}{2} .
    \end{equation}  On the other hand, 
    the free energy approximates the ground state energy from the inequalities \begin{equation}\label{eq: alpha=1: GSE and free energy}
        \frac{\beta}{N\ln N}  \max_{\sigma} \sum_{i<j} J_{ij} \sigma_i \sigma_j \le \frac{1}{N \ln N} \ln Z_N\le \frac{\ln 2}{\ln N} + \frac{\beta}{N\ln N} \max_{\sigma} \sum_{i<j} J_{ij} \sigma_i \sigma_j
    \end{equation}for any $\beta>0$. Combining \eqref{eq:alpha=1} and \eqref{eq: alpha=1: GSE and free energy}, we see that the ground state energy converges to a constant \begin{equation}\label{eq: alpha=1: GSE}
        \frac{1}{N\ln N}\max_{\sigma} \sum_{i<j} J_{ij} \sigma_i \sigma_j \stackrel{p}{\to} \frac{1}{2}
    \end{equation} as $N\to\infty$ for $\beta<\beta_\alpha$.
    Observe that the quantities in \eqref{eq: alpha=1: GSE} do not depend on $\beta$, so we conclude that \eqref{eq: alpha=1: GSE} remains to be true for all $\beta>0$.
    Thus, in view of  \eqref{eq: alpha=1: GSE and free energy},   \eqref{eq:alpha=1} is also true for all $\beta>0$.

\subsection{Proof of Theorem \ref{thm:free_energy_below_one}}

For any integer $k\geq 1,$ denote $[k]=\{1,2,\ldots,k\}.$ Let $X$ have density $$\frac{\alpha}{2} \frac{\1_{ \{ |x| \ge 1\}}}{| x|^{\alpha+1}} $$ so that $X/N^{1/\alpha}$ has the same distribution as $J$.  Set $Y=|X|$.  For $n \ge 1$, let $X_1, X_2, \ldots, X_n$ be i.i.d.\ copies of $X$ and denote $Y_i = |X_i|$.   Let $(I_1, I_2, \ldots, I_n)$ be the random permutation of $[n]$ such that 
\[  |X_{I_1}| > |X_{I_2}|  > \cdots > |X_{I_n}|. \] 
From \eqref{PPP2}, recall the PPP, $(\gamma_j)_{j\geq 1}$, with intensity $\nu(x) =  \alpha x^{ -(1+\alpha)}$ on $(0, \infty)$ ordered descendingly. Let $\mathcal{A}_{\ge}$ be the following subset of $\mathbb{R}^\infty$,
\[ \mathcal{A}_{\ge}  = \{ a=(a_1, a_2, \ldots) : a_1 \ge a_2 \ge \cdots \ge 0 \text{ such that } \lim_i a_i  = 0\}, \]
equipped with the metric $d(a,a')=\sup_{i\geq 1}|a_i-a_i'|$, which ensures that $\mathcal{A}_\ge$ is a complete, separable metric space.  The convergence
in distribution for $\mathcal{A}_{\ge}$-valued random variables is equivalent to finite dimensional weak convergence (see, for example, \cite{bertoin2006random}).

In the following lemma, we collect some well-known results for the order statistics of the heavy-tailed distribution, tailored to our special choice of distribution (see \cite{lepage1981convergence} and \cite[Lemma 2.4]{bordenave2011spectrum}).  We include a proof for completeness.

\begin{lemma} \label{lem:heavy_tail_conv_representation}
	The following statements hold:
	\begin{enumerate}
		\item[(a)]  For $0< \alpha < 2$,
		\begin{equation} \label{eq:Y_repre_poisson}
			n^{-1/\alpha} \bigl(Y_{I_1}, Y_{I_2}, \ldots, Y_{I_n}\bigr) \stackrel{d}{=} \bigl( \gamma_{n+1}/n)^{1/\alpha} ( \gamma_1^{-1/\alpha}, \gamma_2^{-1/\alpha}, \ldots, \gamma_n^{-1/\alpha} \bigr).
		\end{equation}
		Also, 
		\begin{equation} \label{eq:X_repre_poisson}
			n^{-1/\alpha}\bigl (X_1, X_2, \ldots, X_n\bigr) \stackrel{d}{=} ( \gamma_{n+1}/n)^{1/\alpha}\bigl ( \eps_1 \gamma_{r(1)}^{-1/\alpha}, \eps_2 \gamma_{r(2)}^{-1/\alpha}, \ldots, \eps_n\gamma_{r(n)}^{-1/\alpha} \bigr),   
		\end{equation}
		where $\eps_1, \eps_2, \ldots, \eps_n$ are i.i.d.\ Rademacher variables and $r:[n] \to [n]$ is an independent uniform random permutation, both independent of $(\gamma_j)_{j \ge 1}$. 
		\item[(b)]   For $0< \alpha < 2$, as $n \to \infty$, 
		\[ n^{-1/\alpha} \bigl(Y_{I_1}, Y_{I_2}, \ldots, Y_{I_n}, 0, 0, \ldots\bigr) \stackrel{d}{\to} \bigl( \gamma_1^{-1/\alpha}, \gamma_2^{-1/\alpha}, \ldots \bigr) \text{ on } \mathcal{A}_{\ge}.\]
		\item[(c)]  Assume that  $0 < \alpha < 1$. As $n \to \infty$, 
		\[ n^{-1/\alpha} \sum_{j=1}^n Y_{j} \stackrel{d}{\to} \sum_{j=1}^\infty \gamma_j^{-1/\alpha},  \]
		where the infinite sum is finite a.s.
		Also, for any fixed integer $R \ge 1$, as $n\to\infty,$
		\[ n^{-1/\alpha}\Bigl ( \sum_{j=1}^R Y_{I_j}, \sum_{j=R+1}^n Y_{I_j} \Bigr)  \stackrel{d}{\to} \Bigl( \sum_{j=1}^R \gamma_j^{-1/\alpha}, \sum_{j=R+1}^\infty \gamma_j^{-1/\alpha}\Bigr) .  \]
	\end{enumerate}
\end{lemma}

\begin{proof}
	(a) It is easy to check that $Y \stackrel{d}{=} U^{-1/\alpha}$, where $U$ has uniform distribution on $[0, 1]$. Therefore, by independence, $$(Y_{I_1}, Y_{I_2}, \ldots, Y_{I_n}) \stackrel{d}{=} (U_{(1)}^{-1/\alpha}, U_{(2)}^{-1/\alpha}, \ldots, U_{(n)}^{-1/\alpha}),$$ where $U_{(1)} < U_{(2)}< \cdots < U_{(n)}$ are order statistics of $n$ i.i.d.\ uniform variables on $[0,1]$. \eqref{eq:Y_repre_poisson} is now a consequence of the following  identity (see, for example, \cite[Theorem 3.7.11]{durrett_2019}),
	\[ \bigl( U_{(1)}, U_{(2)}, \ldots, U_{(n)}\bigr)   \stackrel{d}{=}  \bigl( \gamma_1/\gamma_{n+1},  \gamma_2/\gamma_{n+1}, \ldots,  \gamma_n/\gamma_{n+1}\bigr). \]
	Next, since $X$ has symmetric distribution, $\mathrm{sgn}(X)$ is a Rademacher variable, independent of $Y = |X|$.
	Therefore, 
	$$\bigl(X_{I_1}, X_{I_2}, \ldots, X_{I_n}\bigr) \stackrel{d}{=} \bigl(\eps_1 Y_{I_1}, \eps_2 Y_{I_2}, \ldots, \eps_n Y_{I_n}\bigr),$$ 
	where $\eps_i$'s are independent of $Y_i$'s. Also,  the ranking $(I_1, I_2, \ldots, I_n)$ is independent of the order statistics $(Y_{I_1}, Y_{I_2}, \ldots, Y_{I_n})$, which immediately yields 
	\eqref{eq:X_repre_poisson}  from \eqref{eq:Y_repre_poisson}.
	
	(b) By strong law of large numbers, $ \gamma_{n+1}/n \to 1$ almost surely, which yields the finite dimensional convergence of $( \gamma_{n+1}/n)^{1/\alpha} ( \gamma_1^{-1/\alpha}, \gamma_2^{-1/\alpha}, \ldots, \gamma_n^{-1/\alpha}, 0, 0, \ldots )$ to $( \gamma_1^{-1/\alpha}, \gamma_2^{-1/\alpha}, \ldots )$.

	(c) By  \eqref{eq:Y_repre_poisson},  $$n^{-1/\alpha} \sum_{j=1}^n Y_{j} \stackrel{d}{=}  ( \gamma_{n+1}/n)^{1/\alpha} \sum_{j=1}^n  \gamma_j^{-1/\alpha}. $$ By strong law of large numbers, 
	$\gamma_j/j \to 1$ almost surely as $j \to \infty$. Thus, since $0 < \alpha < 1$, the series $ \sum_{j=1}^\infty  \gamma_j^{-1/\alpha} $ is convergent almost surely. It follows that
	$$( \gamma_{n+1}/n)^{1/\alpha} \sum_{j=1}^n  \gamma_j^{-1/\alpha} \to \sum_{j=1}^\infty  \gamma_j^{-1/\alpha}$$ almost surely. The second part of (c) has a similar proof. 
\end{proof}

Before proving the theorem, we need a simple lemma on weak convergence, whose proof is omitted. 
\begin{lemma} \label{lem:sanwich_weak_conv}
	Let $(F_N)_{N \ge 1}, (W_{N, R})_{N ,R \ge 1}, (W_N)_{N \ge 1} $ be sequences of random variables defined on the same probability space such that 
	\[ W_{N, R} \le F_N \le W_N \ \ \text{ for any } N, R \ge 1.\]
	Suppose that for any fixed $R \ge 1$, $W_{N, R}  \stackrel{d}{\to} W_{\infty, R}$ and $W_N \stackrel{d}{\to} W_\infty$ as $N \to \infty$. Also, assume that $W_{\infty, R} \stackrel{d}{\to} W_\infty$ as $R \to \infty$. Then $F_N \stackrel{d}{\to} W_\infty$.
	
\end{lemma}

\begin{proof}[\bf Proof of Theorem~\ref{thm:free_energy_below_one}]

 Note that $Z_N \le 2^N \exp( \beta \sum_{i < j} |J_{ij}|)$. Therefore, 	by Lemma~\ref{lem:heavy_tail_conv_representation}(c),
	\[ \frac{1}{N^{1/\alpha}} \ln Z_N \le\frac{1 }{N^{1/\alpha}} \Big( N \ln 2  + \beta \sum_{i < j} |J_{ij}| \Big)  \stackrel{d}{\to} \frac{\beta}{2^{1/\alpha}}  \sum_{k=1}^\infty \gamma_k^{-1/\alpha}. \]
For the lower bound, recall the sequence $(U(k,1),U(k,2))_{1\leq k\leq n}.$ 
	For any fixed $R \ge 1$,
	\begin{align*} 
		Z_N &\ge \max_\sigma \exp\Bigl( \beta \sum_{i < j} J_{ij} \sigma_i \sigma_j\Bigr)  \\
		&\ge  \exp \Big( \beta\max_\sigma  \sum_{k=1}^R J_{U(k, 1), U(k, 2)} \sigma_{U(k, 1)} \sigma_{U(k, 2)} - \beta \sum_{k=R+1}^n  |J_{U(k, 1), U(k, 2)}|  \Big).
	\end{align*}
	Let $\mathcal{C}_{N, R}$ be the event that the edges $\{ ( U(k, 1), U(k, 2)): 1 \le k \le R \}$ on the vertex set $[N]$ are all incident on distinct vertices. Clearly, $\mathbb{P}(\mathcal{C}_{N, R}) \to 1$ as $N \to \infty$. Given the disorders $(J_{ij})$, on  the event $\mathcal{C}_{N, R}$, we can choose $\sigma$ such that $ \mathrm{sgn}(\sigma_{I(k, 1)} \sigma_{I(k, 2)})  =  \mathrm{sgn}(J_{I(k, 1), I(k, 2)})$ for each $1 \le k \le R$. Consequently, again by Lemma~\ref{lem:heavy_tail_conv_representation}(c), we can write
	\begin{align*} 
		\frac{1}{N^{1/\alpha}} \ln Z_N &\ge \frac{\1_{\mathcal{C}^c_{N, R}}}{N^{1/\alpha}} \ln Z_N  + \frac{\beta}{N^{1/\alpha}} \Big(  \sum_{k=1}^R |J_{U(k, 1), U(k, 2)}|   -  \sum_{k=R+1}^n  |J_{U(k, 1), U(k, 2)}| \Big) \1_{\mathcal{C}_{N, R}}   \\
		&\stackrel{d}{\to}  \frac{\beta}{2^{1/\alpha}} \Big(  \sum_{k=1}^R  \gamma_k^{-1/\alpha} -  \sum_{k=R+1}^\infty  \gamma_k^{-1/\alpha} \Big).
	\end{align*}
	The proof  now follows from Lemma~\ref{lem:sanwich_weak_conv}.
\end{proof}

\subsection{Proof of Theorem \ref{thm:Gibbs_structure}}

Recall the random edges $(U(k,1),U(k,2))$ for $1\leq k\leq n=N(N-1)/2$ defined through \eqref{heavytailed:add:eq1}.
We say the edge $(i, j)$ has rank $k$, and write $r(i, j) = k$, if $|J_{ij}|$ is the $k$-th largest  among $(|J_{ij}|)_{1\leq i< j\leq N}$. In other words, for $1\leq i<j\leq N$, $r(i, j) = k$ if and only if 
$( i,j) = (U(k, 1), U(k, 2)).$ The rank function $r: \{(i,j):1\leq i<j\leq N\} \to  \{1,\ldots,n\}$ is a uniform bijection.  We adopt the convention that $ r(i, j)  = r(j, i)$ if $i> j$ and $r(i, i)= \infty$ for all~$i$. 
The proof of Theorem \ref{thm:Gibbs_structure} relies crucially on the following properties of the disorder matrix.

\begin{lemma} \label{lem:Gibbs_structure}
	Let $A$ be a $N \times N$ symmetric matrix with $A_{ii} = 0$ and $A_{ij}   = |J_{ij}|$ for $1\leq i < j\leq N$.  Let $\delta>0$ be an arbitrary small constant and let $\kappa=\min(1/2,1-\alpha)-\delta.$ Take $R=N^{\kappa}$  and  $T = N^{1/2+\delta/2}$. Then 
	\begin{itemize}
		\item[(a)]  With probability going to one, no row of $A$ has two distinct entries with ranks at most $R$ and $T$ respectively. As a result, with probability going to one, the number of entries with rank at most $R$ in each row  of $A$ is either zero or one. 
		\item[(b)]  With probability going to one, the rows of $A$ satisfy that for all $1\leq i\leq N,$ whenever there exists some $j$ such that   $r(i, j) \le R$, the following inequality holds
		$$
		\sum_{j: r(i,j) \le R} |J_{ij}|  -  \sum_{j: r(i,j) > R} |J_{ij}| \ge 10^{-1/\alpha}(N/R)^{1/\alpha}.
		$$
		
	\end{itemize}
\end{lemma}
\begin{proof}
	(a) Let $D_1, D_2, \ldots, D_{N-1}$ be i.i.d.\ uniform variables on $[n]$. The probability that they are all distinct is given by
	\[ \prod_{i=0}^{N-2} \Bigl(1 - \frac{i}{n}\Bigr) \ge \exp\Bigl(- \frac{2}{n} \sum_{i=0}^{N-2}i\Bigr)\geq e^{-2},\]
	where we used the inequality that $1 - x \ge e^{-2x}$ for $0<x \le 1/2$. Observe that for any row $i$, the vector $(r(i, j))_{j \ne i}$ has the same distribution as $(D_1, \ldots, D_{N-1})$ conditioned on them being all distinct. Therefore,  for each $i$, 
	\begin{equation}\label{heavytail:add:eq2}
		\mathbb{P}\bigl( (r(i, j))_{j \ne i} \in \cdot\bigr) \le \frac{\mathbb{P}((D_1, \ldots, D_{N-1}) \in \cdot )}{ \mathbb{P} (\text{$D_1, \ldots, D_{N-1}$ are all distinct} )}. 
	\end{equation} 
	Now since the probability of the event  that there exist $j_1 \ne j_2$  such that $D_{j_1} \le R$ and $D_{j_2} \le T$ is bounded above by $ (N-1)^2 (R/n) (T/n) = 4 RT N^{-2},$ from the display above, the probability that there exist $j_1 \ne j_2$ such that 
	$r(i, j_1) \le R$ and $r(i, j_2) \le T$ is bounded above by $ 4 e^2 RT N^{-2}.$ By a union bound, the probability of the event that there is at least one row of $A$  having two distinct entries with ranks at most $R$ and $T$ is bounded above by $4 e^2 RT N^{-1},$ which goes to $0$. This proves part (a).
	
	(b)  First, we claim that with probability going to one, no row of $A$ has at least $\lceil\ln N\rceil$ entries with ranks at most $N$. To see this, from \eqref{heavytail:add:eq2}, for a particular row, such probability is bounded above by 
	\[ e^{2} {N - 1 \choose \lceil\ln N\rceil} \Big(\frac{N}{n}\Big)^{\lceil\ln N\rceil} \le  e^{2} \Big(\frac{e (N-1)}{\lceil\ln N\rceil}\Big)^{\lceil\ln N\rceil}  \Big(\frac{N}{n}\Big)^{\lceil\ln N\rceil} =  e ^{2} \Big(\frac{2e}{\lceil\ln N\rceil}\Big)^{\lceil\ln N\rceil}, \]
	where we have used the bound $${m \choose k} \le \Bigl(\frac{em}{k}\Bigr)^k$$ for $0 < k \le m.$ The claim now follows from taking a union bound over $N$ rows.  
	
	Next, by Lemma~\ref{lem:heavy_tail_conv_representation}(a),  we can write $$|J_{ij}| = (n/N)^{1/\alpha} (\gamma_{n+1}/n)^{1/\alpha} \gamma_{r(i, j)}^{-1/\alpha},\,\,1\leq i\neq j\leq N,$$ where $(\gamma_j)_j$ is independent of $r$. By the strong law of large numbers,   for any sequence $j_0(N) \to \infty$, with probability tending to one,  
	$$
	0.9 j^{-1/\alpha}  \le \gamma_j^{-1/\alpha}  \le  1.1 j^{-1/\alpha},\,\,\forall j\geq j_0(N).  
	$$
	Now suppose that $i$-th row of $A$ has an entry with rank at most $R$. By part(a) of the lemma and our claim above, with probability tending to one, the following three statements hold: (i) the number of the entries in that row with rank at most $R$ is exactly one, (ii) the rest of the entries in the row has rank at least $T+1$ and (iii) except at most $\lceil\ln N\rceil$ entries, all other entries in the row has rank at least $N+1$. On the intersection of these events, 
	\begin{align} \label{eq:row_lb1}
		\nonumber &\sum_{j: r(i,j) \le R} |J_{ij}|  -  \sum_{j: r(i,j) > R} |J_{ij}|  \\
		&\ge (\gamma_{n+1}/N)^{1/\alpha} \big (  0.9   R^{-1/\alpha}  - \ln N  \cdot 1.1 T^{-1/\alpha} - N \cdot 1.1  N^{-1/\alpha}   \big) \nonumber\\
		&\ge (0.9) 2^{-1/\alpha} (N - 1)^{1/\alpha} \big (  0.9   R^{-1/\alpha}  - \ln N  \cdot 1.1 T^{-1/\alpha} - N \cdot 1.1  N^{-1/\alpha}   \big).
	\end{align}
	From the choice of $\kappa$, we have\[ \frac{ T^{-1/\alpha} \ln N }{R^{-1/\alpha} }  \le \frac{ (N^{1/2 + \delta/2})^{-1/\alpha} \ln N }{ (N^{1/2 - \delta})^{-1/\alpha} } \to 0\,\,\mbox{and}\,\,
	\frac{ N^{1-1/\alpha} }{R^{-1/\alpha} }  \le \frac{ N^{ - (1-\alpha)/\alpha} }{N^{ - ((1- \alpha) - \delta)/\alpha } } \to 0. \]
	Thus,  for large $N$, \eqref{eq:row_lb1} is bounded below by  $10^{-1/\alpha}(N/R)^{1/\alpha}$, completing the proof. 
\end{proof}

\begin{proof}[\bf Proof of Theorem~\ref{thm:Gibbs_structure}]  Let $\mathcal{C}_N$ be the event  that the disorders $(J_{ij})$ satisfy the statements (a) and (b) of Lemma~\ref{lem:Gibbs_structure} and $\mathbb{P}(\mathcal{C}_N) \to 1$. We assume that the event $\mathcal{C}_N$ holds for the rest of the proof. First of all, define the subset $\mathcal{V}$ of $[N]$ as 
	\[ \mathcal{V} = \{  U(k, p):  p \in \{1, 2\}, 1 \le k \le R \} .   \]
	By the statement (a) of $\mathcal{C}_N$,  the edges $(U(k, 1), U(k, 2))$ for $1 \le k \le R$ cannot share any common vertex. Indeed, if $(U(k,1),U(k,2))$ and $(U(k',1),U(k',2))$ for some $k,k'\leq R$ share a vertex, say $a,$ then the $a$-th row of $A$ will contain two entries that are ranked $k,k'\leq R$, a contradiction. Hence, the elements in $\mathcal{V}$ are distinct and the size of $\mathcal V$ is $2R$.

		Write $\sigma \in \{-1,1\}^N$ as $\sigma = (\sigma_{\mathcal{V}}, \sigma_{\mathcal{V}^c})\in \{-1,1\}^{\mathcal{V}}\times \{-1,1\}^{[N]\setminus\mathcal{V}^c}$. For any  $B \subseteq [R]$,  define
	\[ \Sigma_{\mathcal{V}}^B = \Bigl\{  \tau\in \{-1, 1\}^{\mathcal{V}} : \chi_k(\tau) = 0   \text{ for all } k \in B \text{ and } \chi_k(\tau) = 1 \text{ for all } k \in B^c := [R] \setminus B \Bigr\}\]
	and set 
	\[ \Sigma_{\mathcal{V}} = \Sigma_{\mathcal{V}}^\emptyset =  \Bigl\{  \tau\in \{-1, 1\}^{\mathcal{V}} : \chi_k(\tau) = 1  \text{ for all } 1 \le k \le R    \Bigr\} ,\]
	where, with a slightly abuse of notation, for $\tau\in \{-1,1\}^{\mathcal{V}}$ and $k\in [R],$
	$$\chi_k(\tau):=\1_{\{\mbox{sgn}(J_{U(k,1),U(k,2)}\tau_{U(k,1)}\tau_{U(k,2)})=1\}}.$$ 
	Observe that for $\sigma\in \{-1,1\}^N,$
	\begin{equation}\label{eq:equiv_xi}
		\chi_k(\sigma)  = 1, \  \forall \  1 \le k \le R  \ \ \Leftrightarrow  \ \  \sigma_{\mathcal{V}}\in  \Sigma_{\mathcal{V}}. 
	\end{equation}
	Note that $| \Sigma_{\mathcal{V}}^B|   =2^{2R}$  for any $B \subseteq [R]$ . Let  $f_B$  be a bijection from $ \Sigma_{\mathcal{V}}^B$ to   $\Sigma_{\mathcal{V}} $ defined as follows. For each $1\leq k\leq R,$ $f_B(\tau)_{U(k,2)}=\tau_{U(k,2)}$ and $f_B(\tau)_{U(k,1)}=-\tau_{U(k,1)}$ or $\tau_{U(k,1)}$ depending on whether $k\in B$ or $k\in B^c.$ 
	
Now, for any $\sigma$ such that $\sigma_{\mathcal{V}} \in  \Sigma_{\mathcal{V}}$ and  for any  $ \emptyset \ne B \subseteq [R]$, using the convention that $J_{ij} = J_{ji}$ and $J_{ii} = 0$, observe that the spin configuration,
$
\bigl(f_B^{-1}(\sigma_{\mathcal{V}}), \sigma_{{\mathcal{V}}^c}\bigr),
$
 is obtained by keeping the spins in $\sigma$ fixed on the coordinates $(U(k,1):k\in B^c)$, $(U(k,2):k\in [R]),$ and $k\in \mathcal{V}^c$ and flipping the sign of $\sigma$ on the coordinates $(U(k,1):k\in B).$ With this, we readily check that
	\begin{align*}
		-\bigl(H_N(f_B^{-1}(\sigma_{\mathcal{V}}), \sigma_{{\mathcal{V}}^c})-H_N(\sigma)\bigr)
		&=2\beta\sum_{k\in B}\sum_{j\in [N]}J_{U(k,1),j}(-\sigma_{U(k,1)})\sigma_j\\
		&=-2\beta\sum_{k\in B}\Bigl(|J_{U(k,1),U(k,2)}|+\sigma_{U(k,1)}\sum_{j:j\neq U(k,2)}J_{U(k,1),j}\sigma_j\Bigr)\\
		&\leq -2\beta\sum_{k\in B}\Bigl(|J_{U(k,1),U(k,2)}|-\sum_{j:j\neq U(k,2)}|J_{U(k,1),j}|\Bigr)\\		
		&\leq - 2 \beta |B|  10^{-1/\alpha}(N/R)^{1/\alpha},
	\end{align*}
	where the last inequality follows from condition (a) and (b) of $\mathcal{C}_N$. 
	%
	%
	Therefore, for $N$ sufficiently large,  on the event $\mathcal{C}_N$, 
	\begin{align}
		\sum_{ \emptyset \ne B \subseteq [R]}  \max_{ \sigma:  \sigma_{\mathcal{V}} \in \Sigma_{\mathcal{V} } }  \frac{   e^{-  H_N (f_B^{-1}(\sigma_{\mathcal{V}}), \sigma_{{\mathcal{V}}^c}) }  }{ e^{-  H_N(\sigma)} } &\le \bigl(1+ e^{-  2 \beta 10^{-1/\alpha}(N/R)^{1/\alpha}      } \bigr)^R  -1 \le e^{ R  e^{-  2  \beta 10^{-1/\alpha}(N/R)^{1/\alpha} }  }  -1 \nonumber\\
		&\le 2  e^{-  2 \beta 10^{-1/\alpha}(N/R)^{1/\alpha} +\ln R}\le 2  e^{-  2  \beta 10^{-1/\alpha} N^{1/(2\alpha)} },\label{eq:bound using R}
	\end{align}
	where we have used the elementary inequalities $ 1 + x \le e^x, x \in \mathbb{R}$ and $ e^x \le 1+ 2x, 0 \le x <1$. Now from \eqref{eq:equiv_xi}, we obtain 
	\begin{align*}
		1 - \langle \1_{\{ \chi_k  = 1,  \ 1 \le k \le R \} } \rangle &= \frac{\sum_{\sigma: \sigma_{\mathcal{V}}  \not \in \Sigma_{\mathcal{V}} } e^{-  H_N(\sigma)}  }{\sum_{\sigma: \sigma_{\mathcal{V}} \in \Sigma_{\mathcal{V}} } e^{-  H_N(\sigma)}  + 
			\sum_{\sigma: \sigma_{\mathcal{V}} \not \in  \Sigma_{\mathcal{V}} } e^{-  H_N(\sigma)}  } \\
		&\le \frac{  \sum_{ \emptyset \ne B \subseteq [R]} \sum_{\sigma: \sigma_{\mathcal{V}} \in \Sigma^B_{\mathcal{V}} }      e^{-  H_N (\sigma_{\mathcal{V}}, \sigma_{{\mathcal{V}}^c}) }  }{\sum_{\sigma: \sigma_{\mathcal{V}} \in \Sigma_{\mathcal{V}} } e^{-  H_N(\sigma)}    } \\
		&\le \sum_{ \emptyset \ne B \subseteq [R]}  \max_{ \sigma:  \sigma_{\mathcal{V}} \in \Sigma_{\mathcal{V} } }   \frac{   e^{-  H_N (f_B^{-1}(\sigma_{\mathcal{V}}), \sigma_{{\mathcal{V}}^c}) }  }{ e^{-  H_N(\sigma)} } \le 2  e^{-  2  \beta 10^{-1/\alpha} N^{1/(2\alpha)} }.
	\end{align*}
	The theorem now follows from dominated convergence theorem.
\end{proof}

\bibliographystyle{acm}

{\footnotesize\bibliography{references}}

\begin{thebibliography}{10}

\bibitem{Achlioptas2005}
{\sc Achlioptas, D., and Naor, A.}
\newblock The two possible values of the chromatic number of a random graph.
\newblock {\em Ann. of Math. (2) 162}, 3 (2005), 1335--1351.

\bibitem{aggarwal2022mobility}
{\sc Aggarwal, A., Bordenave, C., and Lopatto, P.}
\newblock Mobility edge of {L}{\'e}vy matrices.
\newblock {\em preprint arXiv:2210.09458\/} (2022).

\bibitem{aggarwal2018goe}
{\sc Aggarwal, A., Lopatto, P., and Yau, H.-T.}
\newblock {GOE} statistics for {L}{\'e}vy matrices.
\newblock {\em J. Eur. Math. Soc. 23}, 11 (2021), 3707–3800.

\bibitem{aizenman}
{\sc Aizenman, M., Lebowitz, J.~L., and Ruelle, D.}
\newblock Some rigorous results on the {S}herrington-{K}irkpatrick spin glass
  model.
\newblock {\em Comm. Math. Phys. 112}, 1 (1987), 3--20.

\bibitem{Alberici}
{\sc Alberici, D., and Contucci, P.}
\newblock Solution of the monomer-dimer model on locally tree-like graphs.
  {R}igorous results.
\newblock {\em Comm. Math. Phys. 331}, 3 (2014), 975--1003.

\bibitem{Aldous1985}
{\sc Aldous, D.~J.}
\newblock Exchangeability and related topics.
\newblock In {\em \'{E}cole d'\'{e}t\'{e} de probabilit\'{e}s de
  {S}aint-{F}lour, {XIII}, 1983}, vol.~1117 of {\em Lecture Notes in Math.}
  Springer, Berlin, 1985, pp.~1--198.

\bibitem{Auffinger2009}
{\sc Auffinger, A., Ben~Arous, G., and P\'{e}ch\'{e}, S.}
\newblock Poisson convergence for the largest eigenvalues of heavy tailed
  random matrices.
\newblock {\em Ann. Inst. Henri Poincar\'{e} Probab. Stat. 45}, 3 (2009),
  589--610.

\bibitem{BGT}
{\sc Bayati, M., Gamarnik, D., and Tetali, P.}
\newblock Combinatorial approach to the interpolation method and scaling limits
  in sparse random graphs.
\newblock {\em Ann. Probab. 41}, 6 (2013), 4080--4115.

\bibitem{BenArous2008}
{\sc Ben~Arous, G., and Guionnet, A.}
\newblock The spectrum of heavy tailed random matrices.
\newblock {\em Comm. Math. Phys. 278}, 3 (2008), 715--751.

\bibitem{bertoin2006random}
{\sc Bertoin, J.}
\newblock {\em Random fragmentation and coagulation processes}, vol.~102.
\newblock Cambridge University Press, 2006.

\bibitem{BCS}
{\sc Biswas, R., Chen, W.-K., and Sen, A.}
\newblock Free energy of a diluted spin glass model with quadratic
  {H}amiltonian.
\newblock {\em Ann. Probab. 51}, 1 (2023), 359--395.

\bibitem{bordenave2011spectrum}
{\sc Bordenave, C., Caputo, P., and Chafai, D.}
\newblock Spectrum of large random reversible {M}arkov chains: heavy-tailed
  weights on the complete graph.
\newblock {\em Ann. Probab. 39}, 4 (2011), 1544–1590.

\bibitem{bordenave2013localization}
{\sc Bordenave, C., and Guionnet, A.}
\newblock Localization and delocalization of eigenvectors for heavy-tailed
  random matrices.
\newblock {\em Probab. Theory Relat. Fields 157}, 3-4 (2013), 885--953.

\bibitem{bordenave2017delocalization}
{\sc Bordenave, C., and Guionnet, A.}
\newblock Delocalization at small energy for heavy-tailed random matrices.
\newblock {\em Comm. Math. Phys. 354\/} (2017), 115--159.

\bibitem{Chatterjee}
{\sc Chatterjee, S.}
\newblock Spin glass phase at zero temperature in the {E}dwards-{A}nderson
  model.
\newblock {\em preprint arXiv:2301.04112\/} (2023).

\bibitem{Cizeau1993}
{\sc Cizeau, P., and Bouchaud, J.~P.}
\newblock Mean field theory of dilute spin-glasses with power-law interactions.
\newblock {\em Journal of Physics A: Mathematical and General 26}, 5 (mar
  1993), L187.

\bibitem{Cizeau1994}
{\sc Cizeau, P., and Bouchaud, J.~P.}
\newblock Theory of {L}\'evy matrices.
\newblock {\em Phys. Rev. E 50\/} (Sep 1994), 1810--1822.

\bibitem{Coja-Oghlan2018}
{\sc Coja-Oghlan, A., Krzakala, F., Perkins, W., and Zdeborov\'{a}, L.}
\newblock Information-theoretic thresholds from the cavity method.
\newblock {\em Adv. Math. 333\/} (2018), 694--795.

\bibitem{Coja-Oghlan2019}
{\sc Coja-Oghlan, A., and Perkins, W.}
\newblock Spin systems on {B}ethe lattices.
\newblock {\em Comm. Math. Phys. 372}, 2 (2019), 441--523.

\bibitem{de-Bruijn1952}
{\sc de~Bruijn, N.~G., and Erd\"{o}s, P.}
\newblock Some linear and some quadratic recursion formulas. {II}.
\newblock {\em Proceedings of the Koninklijke Nederlandse Akademie van
  Wetenschappen: Series A: Mathematical Sciences 14\/} (1952), 152--163.

\bibitem{Ding2016}
{\sc Ding, J., Sly, A., and Sun, N.}
\newblock Maximum independent sets on random regular graphs.
\newblock {\em Acta Math. 217}, 2 (2016), 263--340.

\bibitem{Sly2022}
{\sc Ding, J., Sly, A., and Sun, N.}
\newblock Proof of the satisfiability conjecture for large {$k$}.
\newblock {\em Ann. of Math. (2) 196}, 1 (2022), 1--388.

\bibitem{durrett_2019}
{\sc Durrett, R.}
\newblock {\em Probability: Theory and Examples}, 5~ed.
\newblock Cambridge University Press, 2019.

\bibitem{Edwards1975}
{\sc Edwards, S.~F., and Anderson, P.~W.}
\newblock Theory of spin glasses.
\newblock {\em Journal of Physics F: Metal Physics 5}, 5 (may 1975), 965.

\bibitem{Franz}
{\sc Franz, S., and Leone, M.}
\newblock Replica bounds for optimization problems and diluted spin systems.
\newblock {\em J. Statist. Phys. 111}, 3-4 (2003), 535--564.

\bibitem{Guerra}
{\sc Guerra, F.}
\newblock Broken replica symmetry bounds in the mean field spin glass model.
\newblock {\em Comm. Math. Phys. 233}, 1 (2003), 1--12.

\bibitem{Guerra2004}
{\sc Guerra, F., and Toninelli, F.~L.}
\newblock The high temperature region of the {V}iana-{B}ray diluted spin glass
  model.
\newblock {\em J. Statist. Phys. 115}, 1-2 (2004), 531--555.

\bibitem{Hoover1982}
{\sc Hoover, D.~N.}
\newblock Row-column exchangeability and a generalized model for probability.
\newblock In {\em Exchangeability in probability and statistics ({R}ome,
  1981)}. North-Holland, Amsterdam-New York, 1982, pp.~281--291.

\bibitem{Jagannath}
{\sc Jagannath, A., and Lopatto, P.}
\newblock Existence of the free energy for heavy-tailed spin glasses, 2022.

\bibitem{Mezard}
{\sc Janzen, K., Engel, A., and M\'ezard, M.}
\newblock Thermodynamics of the {L}\'evy spin glass.
\newblock {\em Phys. Rev. E 82\/} (Aug 2010), 021127.

\bibitem{Hartmann}
{\sc Janzen, K., Hartmann, A.~K., and Engel, A.}
\newblock Replica theory for {L}\'evy spin glasses.
\newblock {\em J. Stat. Mech. Theory Exp.}, 4 (2008), P04006, 17.

\bibitem{Kanter1987}
{\sc Kanter, I., and Sompolinsky, H.}
\newblock Mean-field theory of spin-glasses with finite coordination number.
\newblock {\em Phys. Rev. Lett. 58\/} (Jan 1987), 164--167.

\bibitem{Klein}
{\sc Klein, M.~W.}
\newblock Temperature-dependent internal field distribution and magnetic
  susceptibility of a dilute ising spin system.
\newblock {\em Phys. Rev. 173\/} (Sep 1968), 552--561.

\bibitem{koster}
{\sc K\"{o}sters, H.}
\newblock Fluctuations of the free energy in the diluted {SK}-model.
\newblock {\em Stochastic Process. Appl. 116}, 9 (2006), 1254--1268.

\bibitem{lepage1981convergence}
{\sc LePage, R., Woodroofe, M., and Zinn, J.}
\newblock Convergence to a stable distribution via order statistics.
\newblock {\em Ann. Probab.\/} (1981), 624--632.

\bibitem{bethelattice}
{\sc M\'{e}zard, M., and Parisi, G.}
\newblock The {B}ethe lattice spin glass revisited.
\newblock {\em Eur. Phys. J. B Condens. Matter Phys. 20}, 2 (2001), 217--233.

\bibitem{MPV}
{\sc M\'{e}zard, M., Parisi, G., and Virasoro, M.}
\newblock {\em Spin glass theory and beyond}, vol.~9 of {\em World Scientific
  Lecture Notes in Physics}.
\newblock World Scientific Publishing Co., Inc., Teaneck, NJ, 1987.

\bibitem{Metz}
{\sc Neri, I., Metz, F.~L., and Boll\'{e}, D.}
\newblock The phase diagram of {L}\'{e}vy spin glasses.
\newblock {\em J. Stat. Mech. Theory Exp.}, 1 (2010), P01010, 21.

\bibitem{nolan2020univariate}
{\sc Nolan, J.~P.}
\newblock Univariate stable distributions.
\newblock {\em Springer Series in Operations Research and Financial Engineering
  10\/} (2020), 978--3.

\bibitem{panchenkobook}
{\sc Panchenko, D.}
\newblock {\em The Sherrington-Kirkpatrick Model}.
\newblock Springer Monographs in Mathematics. Springer New York, 2013.

\bibitem{panchenko2013}
{\sc Panchenko, D.}
\newblock Spin glass models from the point of view of spin distributions.
\newblock {\em Ann. Probab. 41}, 3A (2013), 1315--1361.

\bibitem{Panchenko2016}
{\sc Panchenko, D.}
\newblock Structure of finite-{RSB} asymptotic {G}ibbs measures in the diluted
  spin glass models.
\newblock {\em J. Stat. Phys. 162}, 1 (2016), 1--42.

\bibitem{talagrand2004}
{\sc Panchenko, D., and Talagrand, M.}
\newblock Bounds for diluted mean-fields spin glass models.
\newblock {\em Probab. Theory Relat. Fields 130}, 3 (2004), 319--336.

\bibitem{resnick}
{\sc Resnick, S.~I.}
\newblock {\em Heavy-Tail Phenomena: Probabilistic and Statistical Modeling}.
\newblock Springer Series in Operations Research and Financial Engineering.
  Springer New York, 2010.

\bibitem{rockafellar-1970a}
{\sc Rockafellar, R.~T.}
\newblock {\em Convex analysis}.
\newblock Princeton Mathematical Series. Princeton University Press, Princeton,
  N. J., 1970.

\bibitem{Sherrington1975}
{\sc Sherrington, D., and Kirkpatrick, S.}
\newblock Solvable model of a spin-glass.
\newblock {\em Phys. Rev. Lett. 35\/} (Dec 1975), 1792--1796.

\bibitem{Starr}
{\sc Starr, S., and Vermesi, B.}
\newblock Some observations for mean-field spin glass models.
\newblock {\em Lett. Math. Phys. 83}, 3 (2008), 281--303.

\bibitem{Talagrand2013vol1}
{\sc Talagrand, M.}
\newblock {\em Mean Field Models for Spin Glasses: Volume I: Basic Examples}.
\newblock A Series of Modern Surveys in Mathematics. Springer Berlin
  Heidelberg, 2013.

\bibitem{Talagrand2013vol2}
{\sc Talagrand, M.}
\newblock {\em Mean Field Models for Spin Glasses: Volume II: Advanced
  Replica-Symmetry and Low Temperature}.
\newblock A Series of Modern Surveys in Mathematics. Springer Berlin
  Heidelberg, 2013.

\bibitem{tarquini2016level}
{\sc Tarquini, E., Biroli, G., and Tarzia, M.}
\newblock Level statistics and localization transitions of {L}{\'e}vy matrices.
\newblock {\em Physical review letters 116}, 1 (2016), 010601.

\bibitem{Viana1985}
{\sc Viana, L., and Bray, A.~J.}
\newblock Phase diagrams for dilute spin glasses.
\newblock {\em Journal of Physics C: Solid State Physics 18}, 15 (may 1985),
  3037.

\end{thebibliography}

\newpage

\appendix

\section{Establishing Theorem \ref{invariancestep1}}\label{ProofStep1}

\subsection{Uniform Convergence of Auxiliary Functionals}

In this subsection, we will state the uniform convergence of two auxiliary functionals as a preparation for the proof of Theorem \ref{invariancestep1}.
 Let $n\geq 1,$ $\varepsilon>0$, and $\lambda>0.$ From Subsection \ref{sec:inv+var}, recall $s_I,\hat s_I$ and the PPPs $(\xi_{i,k})_{k\geq 1},$ $(\xi_{i,k}')_{k\geq 1},$ $(\xi_k')_{k\geq 1}.$ For $r\geq 1$ and $I\in \mathbb{N}^r$, denote $$\bar s_I = \bar \sigma(w,u,v_I) =  \int_0^1 \sigma(w,u,v_I,x) d x.$$ 
For any $\sigma\in \mathcal{S},$ define
$A_{n,\eps} (\sigma) =\ln \mathcal{A}_{n,\eps}(\sigma)$ and $B_{n,\eps,\lambda} (\sigma)=\ln \mathcal{B}_{n,\eps,\lambda}(\sigma), $
where
\begin{align}\begin{split}\label{AB:add:eq1}
		\mathcal{A}_{n,\eps}(\sigma)&=\e_{u, \delta} \prod_{i\le n}\prod_{k:|\xi_{i,k}|\geq \eps} ( 1+ \hth(\beta \xi_{i,k}) \bar s_{i, k} \delta_i ),\\
		\mathcal{B}_{n,\eps,\lambda}(\sigma)&=\e_{u}\prod_{i\le n} \prod_{k:|\xi_{i,k}'|\geq \eps} ( 1+ \hth(\beta \xi_{i,k}')\bar s_{1,i,k} \bar s_{2,i,k} ).
	\end{split}
\end{align}
To describe the limit of these objectives, we define
${A}_{n} (\sigma) =\ln \calA_{n }(\sigma)$ and ${B}_{n,\lambda} (\sigma)=\ln \calB_{n,\lambda}(\sigma),$
where
\begin{align*}
	\calA_n(\sigma)&=\e_{u, \delta} \prod_{i\le n}\prod_{k \geq 1} ( 1+ \hth(\beta \xi_{i,k}) \bar s_{i, k} \delta_i ),\\
	\calB_{n,\lambda}(\sigma)&=\e_{u}\prod_{i\le n} \prod_{k \geq 1} ( 1+ \hth(\beta \xi_{i,k}')\bar s_{1,i,k} \bar s_{2,i,k} ).
\end{align*}
The following proposition justifies the well-definiteness and integrability of the four functionals above.

\begin{proposition}\label{epsto0loglimitA}
	Let $n\geq 1$, $\eps>0,$ and $\lambda>0.$ 	For any $\sigma$, ${A}_{n,\eps} (\sigma)$, ${B}_{n,\eps,\lambda} (\sigma),$ ${A}_n(\sigma)$, and $B_{n,\lambda} (\sigma)$ are a.s. finite and integrable.
\end{proposition}

\begin{theorem}\label{parisistep1}
	Let $n\geq 1$ and $\lambda>0.$ We have that
	\begin{align}
		\label{add:eq2}
		\lim_{\eps\downarrow 0}\sup_\sigma \e |A_{n,\eps}(\sigma)- A_n(\sigma)|&=0,\\
			\lim_{\eps\downarrow 0}\sup_\sigma \e|B_{n,\eps,\lambda}(\sigma)-B_{n,\lambda}(\sigma)|&=0 \nonumber.   
	\end{align}
\end{theorem}

The proofs of the two results are presented in the next two subsections.

\subsection{Proof of Proposition \ref{epsto0loglimitA}} \label{prop1}

We only establish Proposition \ref{epsto0loglimitA} for $A_{n,\eps}$ and $A_n$ since the case for $B_{n,\eps,\lambda}$ and $B_{n,\lambda}$ are nearly identical. To simplify our notation, we shall denote $\mathcal{A}_{n,\eps}$ and $\mathcal{A}_n$ by $\calA_{\eps}$ and $\mathcal{A},$ respectively. For any measurable set $S\subseteq (0,\infty),$ define $$\mu_S(dx)=\frac{\alpha I(|x|\in S)}{2|x|^{\alpha+1}}dx$$
and the conditional measure of $\mu$ on the set $S$ as
$$
\mu_{\cdot|S}(dx)=\frac{\mu_S(dx)}{\mu_S(\{x\in \mathbb{R}:|x|\in S\})}.
$$
Note that
\begin{equation*}
	\mu_{[\eps,\infty)}=\mu_{[1,\infty)}(|x|\geq 1)\cdot \mu_{\cdot|[1,\infty)}+\mu_{[\eps,1)}(\eps\leq |x|<1)\cdot \mu_{\cdot|[\eps,1)}=\mu_{\cdot|[1,\infty)}+(\gamma-1)\mu_{\cdot|[\eps,1)}.
\end{equation*}
Let $(\hat g_{i,k})$ be i.i.d. copies from $\mu_{\cdot|[1,\infty)}$ and  $(\hat g_{i,k,\eps})$ be i.i.d. copies from $\mu_{\cdot|(\eps,1)}$ As usual, they are independent of each other and also of other randomness. From the above identity,
we can decompose $(g_{i,k,\eps})_{1\leq k\leq \pi_i(\gamma)}$ in distribution as 
$
(\hat g_{i,k})_{1\leq k\leq \pi_i(1)}\cup
(\hat g_{i,k,\eps})_{1\leq k\leq \pi_i(\gamma-1)}.
$
Let $\sigma\in \mathcal{S}.$ From this decomposition, we can write
\begin{align*}
	\begin{split} \calA_{\eps}(\sigma)
		&\stackrel{d}{=}   \e_{u, \delta} \Bigl(\prod_{i\le n} \prod_{ k  \le \pi_i(1) } ( 1+ \hth(\beta \hat g_{i,k} )  \bar \sigma(w, u, v_{i,k}) \delta_i )\Bigr)\\
		&\qquad\quad\Bigl( \prod_{i\le n}\prod_{ k \le \pi_i( \gamma - 1) } ( 1+ \hth(\beta \hat g_{i,k, \eps}) \bar \sigma(w, u, \hat v_{i,k}) \delta_i )\Bigr).
	\end{split}
\end{align*}
From this,
\begin{align*}
	\begin{split}
		\calA_{\eps} (\sigma)&\stackrel{d}{\le}  \Bigl(\prod_{i\le n}\prod_{ k  \le \pi_i(1) } ( 1+ |\hth(\beta \hat g_{i,k} )| )\Bigr)
		W_{1,\eps}(\sigma)^2,\\
		\calA_{\eps}^{-1}(\sigma)  &\stackrel{d}{\le}  \Bigl(\prod_{i\le n}\prod_{ k  \le \pi_i(1) } ( 1 -  |\hth(\beta \hat g_{i,k} )| )^{-1}\Bigr) W_{2,\eps}(\sigma)^2,
	\end{split}
\end{align*}
where
\begin{align*}
	W_{1,\eps}(\sigma)&=\Bigl(\e_{u, \delta} \prod_{i\leq n}\prod_{ k \le \pi_i( \gamma  - 1) } \big ( 1+ \hth(\beta \hat g_{i,k, \eps}) \bar \sigma(w, u, \hat v_{i,k}) \delta_i \big) \Bigr)\Bigr)^{1/2},\\
	W_{2,\eps}(\sigma)&=\Bigl(\e_{u, \delta} \prod_{i\le n} \prod_{ k \le \pi_i( \gamma-1) }  \big ( 1+ \hth(\beta \hat g_{i,k, \eps}) \bar \sigma(w, u, \hat v_{i,k}) \delta_i \big )\Bigr)^{-1/2}.
\end{align*}
Denote by $\ln_+ x  = \max( \ln x, 0)$ for $x> 0$. Note that $| \ln x | \le \ln_+ x + \ln_+ x^{-1},$  $\ln_+ x \le  \sqrt{x}$ for $x > 0$, and 
$\ln_+ (ab) \le  \ln_+ (a) + \ln_+ (b)$ for $a, b > 0$. It follows that if $x,a,b,a',b'\geq 0$ satisfy $x\leq ab$ and $x^{-1}\leq a'b'$, then $|\ln x|\leq \ln_+a+\ln_+a'+\sqrt{b}+\sqrt{b'}.$ Consequently,
\begin{equation}\label{add:eq-5}
	|A_{\eps}(\sigma) | \stackrel{d}{\le}  Z_{1} + Z_{2} + W_{1, \eps}(\sigma) + W_{2, \eps}(\sigma),
\end{equation}
where 
\begin{equation*}
	Z_{1} :=   \sum_{i\le n}\sum _{ k  \le \pi_i(1) } \ln ( 1+ |\hth(\beta \hat g_{i,k} )| ) \ge 0\,\,\mbox{and}\,\,Z_{2} :=   \sum_{i\le n}\sum _{ k  \le \pi_i(1) }  - \ln ( 1 -  |\hth(\beta \hat g_{i,k} )| ) \ge 0.
\end{equation*}
To show that ${A}_{\eps}(\sigma)\in L^1,$ it suffices to establish that $Z_1, Z_2 \in L^1$ and both $W_{1, \eps}$ and $W_{2, \eps}$ are uniformly $L^2$ bounded. To this end, we bound
\[ \ln ( 1+ |\hth(u)| ) \le |u|, \quad  -  \ln ( 1 -  |\hth(u)| ) \le |u|.  \]
Since $1<\alpha<2$, these inequalities imply that  $\e \ln ( 1+ |\hth(\beta \hat g_{1,1} )| ) < \infty $ and $ - \e   \ln ( 1 -  |\hth(\beta\hat g_{1,1} )| ) < \infty $,  
and therefore, $\e Z_{1} < \infty$ and $\e Z_{2} < \infty.$
Next, for any $\eps > 0$, 
\begin{equation}
	\label{add:eq-7}
	\e W_{1, \eps}(\sigma)^2 = \e  \prod_{i\le n} \prod_{ k \le \pi_i (\gamma-1) } \big ( 1+ \hth(\beta \hat g_{i,k, \eps}) \bar \sigma(w, u, \hat v_{i,k}) \delta_i\big) =1.
\end{equation}
On the other hand, by Jensen's inequality, we have 
\begin{align}
	\nonumber	&\e W_{2, \eps}(\sigma)^2 \\
	\nonumber	&\le \e   \prod_{i\le n}  \prod_{ k \le   \pi_i(\gamma-1) } \big ( 1+ \hth(\beta \hat g_{i,k, \eps}) \bar \sigma(w, u, \hat v_{i,k}) \delta_i \big)^{-1} \\
	\nonumber	&= \e_{w, u, \delta} \exp \Big ( n (\gamma  - 1) \Big (  \e_{g, v} \big ( 1+ \hth(\beta \hat g_{\eps}) \bar \sigma(w, u, v) \delta \big)^{-1}  -1 \Big) \Big)\\
	\nonumber	&= \e_{w, u, \delta} \exp \Big ( n (\gamma  - 1) \Big (  \e_{g, v} \Big [  \big ( 1+ \hth(\beta \hat g_{\eps}) \bar \sigma(w, u, v) \delta \big)^{-1}  -1 + \hth(\beta \hat g_{\eps}) \bar \sigma(w, u, v) \delta \Big] \Big) \Big)\\
	\nonumber	&= \e_{w,u,\delta} \exp \Big (  n (\gamma-1)\e_{g,v}   \frac{\hth^2(\beta \hat g_{\eps})\bar\sigma^2(w,u,v)\delta^2}{1+\hth(\beta \hat g_{\eps})\bar\sigma(w,u,v)\delta}  \Big)\\
	\label{add:eq-6}	&\le \exp \Big (   n  \int_{-1}^1 \frac{\hth^2(\beta x)}{1-\hth(\beta )} \mu(dx) \Big) <\infty,
\end{align}
where $\hat g_{\eps}$ is a copy of $\hat g_{i,k,\eps}$ and the last inequality holds since
$$
\frac{\hth^2(\beta \hat g_{\eps})\bar\sigma^2(w,u,v)\delta^2}{1+\hth(\beta \hat g_{\eps})\bar\sigma(w,u,v)\delta}\leq \frac{\hth^2(\beta \hat g_{\eps})}{1-\hth(\beta)}.
$$
These together imply that $(A_{\eps}(\sigma))_{\eps>0}$ is uniformly integrable.

Next, we proceed to show that $A (\sigma)$ is a.s. finite  and integrable. We need the following lemma.

\begin{lemma}\label{normofYbounded}
	For $m\ge 1$, we have $\sup_{\sigma, \eps} \e \calA_{\eps}(\sigma)^m <\infty.$
\end{lemma}

\begin{proof}
	From the H\"older inequality, it suffices to establish lemma only for $n=1.$  Our assertion follows directly from 
	\begin{align*}  
		\e \calA_{\eps}(\sigma)^m	&=\e    \prod_{k \le \pi(\eps^{-\alpha})}  \big ( 1+ \hth(\beta g_{1,k, \eps}) \bar s_{1,k} \delta_1 \big)^m\\
		&= \e_{w, u, \delta} \exp \Big (   \eps^{-\alpha} \Big (  \e_{g, v} \big ( 1+ \hth(\beta g_{1,1,\eps}) \bar s_{1,k} \delta_1 \big)^m  -1 \Big) \Big)\\
		&=  \e_{w, u} \exp \Big (   \eps^{-\alpha} \Big (  \e_{g, v} \Big [ \sum_{1\le \ell\le m/2} {m\choose 2\ell} \hth^{2\ell}(\beta g_{1,1,\eps}) \bar s_{1,k}^{2\ell} \Big]  \Big) \Big)\\
		&\le \e_{w, u} \exp \Big (   \e_{v}  \bar s_{1,k}^2 2^m \int_{ |x| > \eps}  \hth^2(\beta x)   \mu(dx) \Big)\\
		&\le  \exp \Big (  2^m \int_{\mathbb{R}}  \hth^2(\beta x)   \mu(dx) \Big)<\infty,
	\end{align*}
	where  the third equality holds due the symmetry of $g_{1,k,\eps}.$
\end{proof}

Observe that 
\begin{equation*}
	Y_{m}(\sigma):=\prod_{i\leq n}\prod_{1\leq k\leq m } ( 1+ \hth(\beta \xi_{i,k}) \bar s_{i, k} \delta_i ),\,\,m\geq 1
\end{equation*}
is a positive martingale with respect to the filtration 
$(\sigma(\xi_{i,k},\bar s_{i,k}:1\leq k\leq m,1\leq i\leq n)\bigr)_{m\geq 1}.$ It follows that $Y_{m}(\sigma)$ converges a.s. as $m\to\infty$. Denote this limit by $Y(\sigma).$ Lemma \ref{normofYbounded} further implies that this convergence holds in $L^m$ for any $m\geq 1.$
This concludes that $\calA (\sigma)=\e_{u,\delta}Y(\sigma)$ is a.s. finite and nonnegative and furthermore,
\begin{equation}
	\label{add:eq001}	\calA(\sigma)=\e_{u,\delta}Y(\sigma)=\lim_{m\to\infty}\e_{u,\delta}Y_{m}(\sigma)=\lim_{\eps\downarrow 0}\calA_{\eps}(\sigma).
\end{equation}
Hence,
by using Fatou's lemma and the fact that $({A}_{\eps}(\sigma))_{\eps>0}$ is uniformly integrable, we arrive at 
\begin{equation}\label{add:eq-8}
	\e |A(\sigma) | \leq\liminf_{\eps\downarrow 0}\e |A_{\eps}(\sigma) |<\infty.
\end{equation}
This bound ensures that $A(\sigma)$ is finite a.s. and is integrable, completing our proof. 


\subsection{Proof of Theorem \ref{parisistep1}}

We will again only establish the uniform convergence only for ${A}_{n,\eps}$ since the case for ${B}_{n,\eps,\lambda}(\sigma)$ can be treated similarly. Recall our notation $\cal A,$ $A,$ $\calA_{\eps},$ $A_\eps,$ $W_{1,\eps},W_{2,\eps},Z_1,Z_2$ from the previous subsection. We need a few lemmas. First of all, we show that $\calA$ and $\calA_\eps$ are  uniformly positive in $\sigma$ and $\eps$ and that $\calA_\eps$ converges to $\calA$ uniformly in $\sigma.$

\begin{lemma}\label{uniform X}
	We have
	$$\lim_{\eta\downarrow0}\inf_{\sigma,\eps>0}\Prob(\calA_{\eps}(\sigma) >\eta)=1\,\,\mbox{and}\,\,\lim_{\eta\downarrow0}\inf_{\sigma}\Prob(\calA(\sigma) >\eta)=1.$$
	In particular, for any $\eps>0$, $\min\{\calA(\sigma),\calA_{\eps}(\sigma)\}>0$ a.s..
\end{lemma}
\begin{proof}
	Using \eqref{add:eq-8}, we have $\sup_{\sigma} \e|A(\sigma)|\le \sup_{\sigma,\eps} \e|A_\eps(\sigma)| <\infty$.
	For $0<\eta<1$, we have \begin{equation*}
		\sup_\sigma \Prob(\calA(\sigma) \le\eta)=\sup_\sigma \mathbb{P}(1/\calA(\sigma)\ge 1/\eta)\le \frac{\sup_\sigma \e|\ln(1/\calA(\sigma))|}{\ln (1/\eta)}
	\end{equation*}
	and taking $\eta\downarrow 0$ yields the result.
	The other claim can be treated similarly.
\end{proof}

\begin{lemma}\label{uniform Y}
	For any $m\ge 1$, we have $\lim_{\eps\downarrow 0}\sup_{\sigma} \e  |\calA_\eps(\sigma) -\calA(\sigma)|^m =0.$
\end{lemma}
\begin{proof}
	We only consider the case that $n=1$ as the general case can be treated similarly. Let $m\geq 1$ be fixed.
	Set
	\begin{equation*}
		Y(\sigma) =\prod_{k\ge 1} ( 1+ \hth(\beta \xi_{1,k} ) \bar s_{1, k} \delta_1 ),\,\,
		Y_\eps(\sigma)=\prod_{k: |\xi_{1,k}| > \eps} ( 1+ \hth(\beta \xi_{1,k} ) \bar s_{1, k} \delta_1 ).
	\end{equation*}
	Recall from the end of the proof of Proposition \ref{epsto0loglimitA} that the infinite product in $Y(\sigma)$ converges a.s.. Thus, $Y_{\eps}(\sigma)\to Y(\sigma)$ a.s. as $\eps\downarrow 0.$
	From Fatou's lemma, $$ \e |Y(\sigma)-Y_\eps(\sigma)|^m\le \liminf_{\eps'\downarrow0} \e|Y_{\eps'}(\sigma) -Y_\eps(\sigma)|^m.$$
	For $\eps>0,$, we bound
	\begin{align*}
		|Y_{\eps'}(\sigma) -Y_\eps(\sigma)|^m&\leq 2^m\bigl(|Y_{\eps'}(\sigma)|^m+|Y_\eps(\sigma)|^m\bigr)|Y_{\eps'}(\sigma)-Y_\eps(\sigma)|\\
		&=2^m\bigl(|Y_{\eps'}(\sigma)|^m+|Y_\eps(\sigma)|^m\bigr)|Y_\eps(\sigma)|\Bigl(1-\prod_{k: \eps'<|\xi_{1,k}|\le\eps}(1+\hth(\beta \xi_{1,k})\bar s_{1,k} \delta_1)\Bigr).
	\end{align*}
	From the Cauchy-Schwarz inequality and Lemma \ref{normofYbounded}, it suffices to show that \[\sup_\sigma \sup_{0< \eps'<\eps}\e \Bigl(1-\prod_{k: \eps'<|\xi_k|\le\eps}(1+\hth(\beta \xi_{1,k})\bar s_{1,k} \delta_1)\Bigr)^2 \xrightarrow[\eps\downarrow0]{} 0.\]
	To see this, we compute \begin{align*}
		&\e \Bigl(1-\prod_{k: \eps'<|\xi_{1,k}|<\eps}(1+\hth(\beta \xi_{1,k})\bar s_{1,k} \delta_1)\Bigr)^2\\
		&=1-2\e\prod_{k: \eps'<|\xi_{1,k}|\le\eps}(1+\hth(\beta \xi_{1,k})\bar s_{1,k} \delta_1)+\e\prod_{k: \eps'<|\xi_{1,k}|\le\eps}(1+\hth(\beta \xi_{1,k})\bar s_{1,k} \delta_1)^2.
	\end{align*}
	Note that for $g_{\eps,\eps',k}$ i.i.d. sampled from $\mu_{\cdot|(\eps',\eps]}$, the random set $\{\xi_{1,k}:\eps'<|\xi_{1,k}|\leq \eps\}$ is in distribution the same as $\{g_{\eps,\eps',k}:1\leq k\leq \pi((\eps')^{-\alpha}-\eps^{-\alpha})\}.$ Therefore, using the symmetry of $g_{\eps,\eps',k}$,  the first two terms together equals $-1$. As for the third term, write
	\begin{align*}
		&\e\prod_{k: \eps'<|\xi_{1,k}|\le\eps}(1+\hth(\beta \xi_{1,k})\bar s_{1,k} \delta_1)^2\\
		&=\e\prod_{k\le \pi(\eps^{-\alpha}  - \eps'^{-\alpha})}(1+\hth(\beta g_{\eps,\eps',k})\bar s_{1,k} \delta_1)^2\\
		&=\e_{w,u,\delta}\exp\Bigl((\eps^{-\alpha}  - \eps'^{-\alpha})\bigl(\e_{g,v}(1+\hth(\beta g_{\eps,\eps',k})\bar s_{1,k} \delta_1)^2-1\bigr)\Bigr)\\
		&=\e_{w,u}\exp\Bigl((\eps^{-\alpha}  - \eps'^{-\alpha})\e_{g,v}\hth^2(\beta g_{\eps,\eps',k})\bar s_{1,k}^2)\Bigr)\\
		&=\e_{w,u}\exp\Bigl(\e_{v}\bar s_{1,k}^2 \int_{\eps'<|x|\le \eps}\hth^2(\beta x)\mu(dx)\Bigr)\\
		&\le \exp\Bigl( \int_{|x|\le \eps}\hth^2(\beta x)\mu(dx)\Bigr) \xrightarrow[\eps\downarrow0]{}1,
	\end{align*}
	which completes our proof.
\end{proof}

By using the two lemmas above, we continue to show that $A_\eps$ converges to $A$ uniformly in probability.

\begin{lemma}\label{add:lem1}
	For any $\eta>0$,
	$\lim_{\eps\downarrow 0}\sup_\sigma \Prob(|A_\eps(\sigma) -A(\sigma)| > \eta)=0.$
\end{lemma}
\begin{proof}
	Fix $\eta, \eta'>0$.
	From Lemma \ref{uniform X}, we can choose $\eta''>0$ such that $\inf_{\sigma}\Prob(\calA(\sigma)>\eta'')>1-\eta'$. 
	From Lemma \ref{uniform Y}, we can choose $\eps_0>0$ such that for any $0<\eps<\eps_0$, $\sup_\sigma \e|\calA(\sigma)-\calA_\eps(\sigma)|<\eta\eta'\eta''.$
	For any $0<\eps<\eps_0$, write \begin{align*}
		\Prob(|A_\eps(\sigma) -A(\sigma)| > \ln(1+\eta)) &= \Prob(\calA_\eps(\sigma)>\calA(\sigma)(1+\eta))+\Prob(\calA(\sigma)>\calA_\eps(\sigma) (1+\eta))\\
		&=\Prob\biggl(\frac{\calA_{\eps}(\sigma)-\calA(\sigma)}{\calA(\sigma)}>\eta\biggr)+\Prob\biggl(\frac{\calA(\sigma)-\calA_{\eps}(\sigma)}{\calA_{\eps}(\sigma)}>\eta\biggr).
	\end{align*}
	We bound the first term by \begin{align*}
		\Prob\Bigl(\frac{\calA_{\eps}(\sigma)-\calA(\sigma)}{\calA(\sigma)}>\eta\Bigr)&\le \Prob\biggl(\frac{\calA_{\eps}(\sigma)-\calA(\sigma)}{\calA(\sigma)}>\eta, A(\sigma)>\eta''\biggr)+\eta' \\
		&\le \Prob(\calA_{\eps}(\sigma)-\calA(\sigma)>\eta\eta'')+\eta'  \\
		&\le \frac{\e|\calA(\sigma)-\calA_{\eps}(\sigma)|}{\eta\eta''}+\eta'\le 2\eta'.
	\end{align*}
	It follows that \[\sup_\sigma \Prob\biggl(\frac{\calA_{\eps}(\sigma)-\calA(\sigma)}{\calA(\sigma)}>\eta\biggr)\le 2\eta'. \]
	We can similarly bound the second term and get the desired result.
\end{proof}

The last lemma we need establishes a uniform control on the large deviation between $A_\eps$ and $A$.

\begin{lemma}\label{lemma}
	For $M>0$ and $1<\alpha<2$, $$\sup_{\sigma,\eps}\e \bigl[|A_{\eps}(\sigma) -A(\sigma)|;|A_{\eps}(\sigma) -A(\sigma)|>M\bigr]$$ is bounded above by \begin{align*}
		&16\bigl(\e Z_1\1_{\{Z_1>M/8\}}+Z_2\1_{\{Z_2>M/8\}}\bigr)+2\e|A(\sigma)|\1_{\{A(\sigma)|>M/2\}}\\
		&\qquad+\frac{128}{M}\Bigl(1+\exp\int_{-1}^{1}\frac{\hth^2(\beta x)}{1-\hth(\beta)}\mu (dx)\Bigr) .
	\end{align*}
\end{lemma}
\begin{proof}
	For any random variables $X,Y$ and $M>0,$ it is easy to see that
	\begin{align*}
		&\e|X+Y|\1_{\{|X+Y|>M\}}\le 2 \bigl( \e |X|\1_{\{2|X|>M\}}+\e |Y|\1_{\{2|Y|>M\}} \bigr).
	\end{align*}
	Applying the above to $A_\eps$ and $A$ yields \begin{align*}
		\e |A_{\eps}(\sigma) -A(\sigma)|\1_{\{|A_\eps(\sigma) -A(\sigma)|>M\}} \le 2 \e |A_\eps(\sigma)| \1_{\{|A_\eps(\sigma)|>M/2\}}+2\e |A(\sigma)| \1_{\{|A(\sigma)|>M/2\}}.
	\end{align*}
	Recall from \eqref{add:eq-5} that $|A_\eps(\sigma)| \stackrel{d}{\le}Z_1+Z_2+W_{1,\eps}(\sigma)+W_{2,\eps}(\sigma)$.
	Then, by iteration, we have \begin{align*}
		&\e |A_\eps(\sigma)| \1_{\{|A_\eps(\sigma)|>M/2\}}\\
		&\le \e(Z_1+Z_2+W_{1,\eps}(\sigma)+W_{2,\eps}(\sigma))\1_{\{Z_1+Z_2+W_{1,\eps}(\sigma)+W_{2,\eps}(\sigma)>M\}}
		\\&\le 8(\e Z_1\1_{\{Z_1>M/8\}}+\e Z_2\1_{\{Z_2>M/8\}}+\e W_{1,\eps}(\sigma)\1_{\{W_{1,\eps}(\sigma)>M/8\}}+\e W_{2,\eps}(\sigma)\1_{\{W_{2,\eps}(\sigma)>M/8\}}).
	\end{align*}
	The proof is now complete by noticing that from \eqref{add:eq-7} and \eqref{add:eq-6},
	\[\e W_{1,\eps}(\sigma)\1_{\{W_{1,\eps}(\sigma)>M/8\}} \le \frac{8}{M}\e W_{1,\eps}(\sigma)^2=\frac{8}{M}\] 
	and \[\e W_{2,\eps}(\sigma)\1_{\{W_{2,\eps}(\sigma)>M/8\}}\le \frac{8}{M}\e W_{2,\eps}(\sigma)^2\leq\frac{8}{M}\exp \Big (   n  \int_{-1}^1 \frac{\hth^2(\beta x)}{1-\hth(\beta )} \mu(dx) \Big).\]
\end{proof}

\begin{proof}[\bf Proof of \eqref{add:eq2} in Theorem \ref{parisistep1}] 
	Denote by $ {L}_\eps(\sigma)=A_\eps(\sigma)-A(\sigma).$ 
	Fix $0<\eta<1$.
	From the statement and the proof of Proposition \ref{epsto0loglimitA}, we see that $A(\sigma),Z_1,Z_2$ are integrable.
	As a result, by Lemma \ref{lemma}, we can choose $M>1$ sufficiently large so that $\sup_{\sigma,\eps}\e [| {L}_\eps(\sigma)|;| {L}_\eps(\sigma)|>M]\le \eta.$
	From Lemma \ref{add:lem1}, we can choose $\eps_0>0$ small such that for $0<\eps<\eps_0,$ $\sup_\sigma \Prob(| {L}_\eps(\sigma)| > \eta)<\eta/M.$
	Putting these together yields that for any $0<\eps<\eps_0$, 
	\begin{align*}
		\sup_{\sigma}	\bigl|\e {L}_\eps(\sigma)\bigr|&	\leq \sup_{\sigma}\e| {L}_\eps(\sigma)|\\
		&	\leq \sup_{\sigma}\e\bigl[	| {L}_\eps(\sigma)|;| {L}_\eps(\sigma)|\geq M\bigr]+\sup_{\sigma}\e\bigl[	| {L}_\eps(\sigma)|;\eta\leq | {L}_\eps(\sigma)|<M\bigr]+\eta\\
		&	\leq \sup_{\sigma}\e\bigl[	| {L}_\eps(\sigma)|;| {L}_\eps(\sigma)|\geq M\bigr]+M\sup_{\sigma}\mathbb{P}(\eta\leq | {L}_\eps(\sigma)|<M)+\eta\\
		&\leq 3\eta,
	\end{align*}
	completing our proof.
\end{proof}

\subsection{Proof of Theorem \ref{invariancestep1}}

Note that $U_{l,\lambda}(\sigma)$ and $V_\lambda(\sigma)$ are almost surely finite and are integrable.

\medskip

{\noindent \bf Proof of \eqref{invariancestep1:eq1}:}
First of all, from the same argument for the a.s. finiteness and integrability for $\calA_{n,\epsilon}$ and $\calA_n$ in Proposition \ref{epsto0loglimitA} (see Lemma \ref{normofYbounded} and \eqref{add:eq001}), we see that $U_{l,\lambda}$ and $V_{\lambda}$ are a.s. finite and have any finite $m$-th moments. In addition, we can have that
\begin{align}\label{add:eq-17}
	\sup_{\sigma}\e|\ln V_\lambda|\leq \sup_{\sigma,\eps}\e|\ln V_\lambda|<\infty.
\end{align}
Second, from Lemma \ref{uniform Y}, we can further see that for any $m\geq 1,$
\begin{align}
	\begin{split}\label{add:eq-15}
		\lim_{\eps\downarrow 0}\sup_{\sigma}\e|U_{l,\eps,\lambda}-U_{l,\lambda}|^m=0,\,\,\lim_{\eps\downarrow 0}\sup_{\sigma}\e|V_{\eps,\lambda}-V_{\lambda}|^m=0.
	\end{split}
\end{align}
Now similar to the proof of Lemma \ref{uniform X}, \eqref{add:eq-17} readily implies that  
\begin{align}
	\label{add:eq-18}
	\lim_{\eta\downarrow 0}	\bigl(\sup_{\sigma, \eps} \Prob(V_{\eps,\lambda}\le \eta )+\Prob(V_\lambda \le \eta ) \bigr)=0.
\end{align}
From this and noting that $|U_{l,\eps,\lambda}|\le V_{\eps,\lambda}$ implies
$$
\frac{\prod_{l\le q} |U_{l,\eps,\lambda}|}{V_{\eps,\lambda} ^q}\leq 1\,\,\mbox{and}\,\,\frac{\prod_{l\le q} |U_{l,\lambda}|}{V_\lambda^q}\leq 1,
$$
we have
\begin{equation*}
	\lim_{\eta\downarrow 0}\limsup_{\eps\downarrow 0}\sup_\sigma \e \biggl| \frac{\prod_{l\le q} U_{l,\eps,\lambda}}{V_{\eps,\lambda} ^q}\1_{\{V_{\eps,\lambda} \le\eta\}}- \frac{\prod_{l\le q} U_{l,\lambda}}{V_\lambda^q}\1_{\{V_\lambda\le\eta\}}\biggr|=0.
\end{equation*}
On the other hand, 
\begin{align*}
	\e \biggl| \frac{  \prod_{l\le q} U_{l,\eps,\lambda}}{V_{\eps,\lambda} ^q}\1_{\{V_{\eps,\lambda} >\eta\}}- \frac{\prod_{l\le q} U_{l,\lambda} }{V_\lambda^q} \1_{\{V_\lambda>\eta\}}\biggl| 
	&\le \e \biggl| \prod_{l\le q} U_{l,\eps,\lambda}-\prod_{l\le q} U_{l,\lambda}\biggr|\frac{\1_{\{V_\lambda>\eta\}}}{V_\lambda^q}\\
	&+ \e  \biggl(\prod_{l\le q} |U_{l,\eps,\lambda}|\biggr) \biggl|\frac{1}{V_{\eps,\lambda} ^q}-\frac{1}{V_\lambda^q}\biggr|\1_{\{V_\lambda>\eta\}}\\
	&+\e  \frac{  \prod_{l\le q} |U_{l,\eps,\lambda}|}{V_{\eps,\lambda} ^q} |\1_{\{V_{\eps,\lambda}>\eta\}} -\1_{\{V_\lambda>\eta\} }|.
\end{align*}
Here, the first term is obviously bounded by 
$
	\eta^{-q}\e | \prod_{l\le q} U_{l,\eps,\lambda}-\prod_{l\le q} U_{l,\lambda}|
$
and the second term is bound by
\begin{equation*}
	\e  V_{\eps,\lambda}^q \biggl|\frac{1}{V_{\eps,\lambda} ^q}-\frac{1}{V_\lambda^q}\biggr|\1_{\{V_\lambda>\eta\}} \leq 	\eta^{-q}\e  |{V_{\eps,\lambda} ^q}-{V_\lambda^q}|.
\end{equation*}
As for the third term, it is bounded above by $\mathbb{P}(V_{\eps,\lambda}\geq \eta)+\mathbb{P}(V_{\lambda}\geq \eta).$ Our assertion then follows from \eqref{add:eq-17}, \eqref{add:eq-15}, and \eqref{add:eq-18}.

\medskip

{\noindent \bf Proof of \eqref{invariancestep1:eq2}:} Note that from Proposition \ref{epsto0loglimitA}, $P(\sigma)=\e A_1(\sigma)-\e B_{1,1}(\sigma)$ is finite. From the definition of $P_\eps$ in \eqref{add:eq-12}, since $(g_{i,k,\eps})_{1\leq k\leq \pi_i(\gamma)}$ and $(g_{i,k,\eps})_{1\leq k\leq \pi_i(\gamma/2)}$ are PPPs with mean measures $\mu_{[\eps,\infty)}$ and $\mu_{[\eps,\infty)}/2$, these imply that, in terms of random sets,
\begin{align*}
	\{g_{i,k,\eps}:1\leq k\leq \pi_i(\gamma)\}\stackrel{d}{=}\{\xi_{i,k}:|\xi_{i,k}|\geq \eps\,\,\mbox{for $k\geq 1$}\},\\
	\{g_{i,k,\eps}:1\leq k\leq \pi_i(\gamma/2)\}\stackrel{d}{=}\{\xi_{i,k}':|\xi_{i,k}'|\geq \eps\,\,\mbox{for $k\geq 1$}\}.
\end{align*}
Consequently, $P_\eps(\sigma)=\e A_{1,\eps}(\sigma)-\e B_{1,\eps,1}(\sigma)$ and \eqref{invariancestep1:eq2} follows from Theorem \ref{parisistep1} and \eqref{Qeps}.

\section{Establishing Theorem \ref{invariancestep2}}\label{ProofStep2}

Recall $\mathcal{A}_{n,\eps}$ and $B_{n,\eps,\lambda}$ from \eqref{AB:add:eq1}. We need a key lemma.

\begin{lemma}\label{UVconvergence}
	Suppose $\sigma_m \to \sigma_0$ f.d.d. as $m\to \infty$. For any $\eps>0$ and $\lambda>0,$ we have that as $m\to\infty,$
	\[\bigl(U_{1,\eps,\lambda}(\sigma_m),\ldots,U_{q,\eps,\lambda}(\sigma_m), V_{\eps,\lambda} (\sigma_m)\bigr)\stackrel{d}{\to}\bigl(U_{1,\eps,\lambda}(\sigma_0),\ldots,U_{q,\eps,\lambda}(\sigma_0), V_{\eps,\lambda} (\sigma_0)\bigr)\]
\end{lemma}
and
\begin{equation*}
	\mathcal{A}_{n,\eps}(\sigma_m)\stackrel{d}{\to} \mathcal{A}_{n,\eps}(\sigma_0),\,\,\mathcal{B}_{n,\eps,\lambda}(\sigma_m)\stackrel{d}{\to} \mathcal{B}_{n,\eps,\lambda}(\sigma_0).
\end{equation*}
\begin{proof}
	We will only establish our first assertion since the arguments for the other two are identical.
	Denote by $(\delta_{l,\ell,i})_{\ell,i\geq 1}$ for $l,\ell,i\geq 1$ i.i.d. Rademacher random variables and $u_{l,\ell},x_{l,\ell,i},x_{l,\ell,i,k},\hat x_{l,\ell,i,k}$ for $l,\ell,i,k\geq 1$ i.i.d. uniform random variables on $[0,1]$. These are independent of each other and everything else. Denote
	\begin{align*}
		s_{m,l,\ell,i}&=\sigma_m(w,u_{l,\ell},v_i,x_{l,\ell,i}),\\
		s_{m,l,\ell,i,k}&=\sigma_m(u,u_{l,\ell},v_{i,k},x_{l,\ell,i,k}),\\
		\hat s_{m,l,\ell,i,k}&=\sigma_m(w,u_{l,\ell},\hat v_{i,k},\hat x_{l,\ell,i,k}).
	\end{align*}
	Note that for any $a_1,\ldots,a_q,b\geq 1,$ we can write by introducing independent copies,
	\begin{align*}
		\Bigl(	\prod_{l\leq q}U_{l,\eps,\lambda}(\sigma_m)^{a_l}\Bigr) V_{\eps,\lambda}(\sigma_m)^b&=\e_{u,x,\delta} L_m,
	\end{align*}
	where 
	\begin{align*}
		L_m&:=\Bigl(\prod_{l\leq q}\prod_{\ell\leq a_l}\prod_{i\in C_{l} ^1}\delta_{l,\ell,i}\Bigr) \Bigl(\prod_{l\leq q}\prod_{\ell\leq a_l+b}\prod_{i\le n}\prod_{k\le \pi_i(2\gamma)}\big(1+\hth (\beta g_{i,k,\eps }) \delta_{l,\ell,i}  s_{m,l,\ell,i,k}\big)\Bigr)\\
		&\qquad\Bigl(\prod_{l\leq q}\prod_{\ell\leq a_l}\prod_{i\in C_{\ell}^2}  s_{m,l,\ell,i} \Bigr)\Bigl(\prod_{l\leq q}\prod_{\ell\leq a_l+b}\prod_{k\le \pi(\gamma)}\bigl(1+\hth(\beta \hat g_{k,\eps }) { \hat{s}}_{m,l,\ell,1,k} { \hat{s}}_{m,l,\ell,2,k}\bigr)\Bigr).
	\end{align*}
	From the convergence of $\sigma_m\to \sigma_0$ in f.d.d., we see that
$
		\e_{u,v,x}L_m\to  \e_{u,v,x}L_0$ as $m\to \infty.$
	On the other hand,  note that
	\begin{align*}
		|\e_{u,v,x}L_m|&\leq 2^{\sum_{i=1}^n\pi_i(2\gamma)\cdot \sum_{l=1}^q(a_i+b)+\pi(\gamma)\sum_{l=1}^q(a_l+b)},\,\,\forall m\geq 0.
	\end{align*}
	For any $M\geq 1,$ let $I_M:=\1_{\{\pi_1(2\gamma)\leq M,\ldots,\pi_n(2\gamma)\leq M,\pi(\gamma)\leq M\}}$ and $I_M^c:=1-I_M.$
	It follows from the bounded convergence theorem,
	\begin{align*}
		\e \Bigl(	\prod_{l\leq q}U_{l,\eps,\lambda}(\sigma_m)^{a_l}\Bigr) V_{\eps,\lambda}(\sigma_m)^bI_M&=\e\bigl[\e_{u,\delta,x}[ L_m] I_M\bigr]=\e\bigl[\e_{u,v,x}[ L_m] I_M\bigr]\\
		&\to \e\bigl[\e_{u,v,x}[ L_0] I_M\bigr]=\e\bigl[\e_{u,\delta,x}[ L_0] I_M\bigr]\\
		&=\e \Bigl(	\prod_{l\leq q}U_{l,\eps,\lambda}(\sigma_0)^{a_l}\Bigr) V_{\eps,\lambda}(\sigma_0)^bI_M\,\,\mbox{as $m\to \infty$}.
	\end{align*}
	From this and noting that $U_{l,\eps,\lambda}(\sigma_m)I_M$ and $V_{\eps,\lambda}(\sigma_m)I_M$ are $m$-uniformly bounded by a deterministic constant, the Cram\'er-Wold theorem then implies that
	$$
	\bigl(U_{1,\eps,\lambda}(\sigma_m)I_M,\ldots,U_{q,\eps,\lambda}(\sigma_m)I_M,V_{\eps,\lambda}(\sigma_m)I_M\bigr)\stackrel{d}{\to} \bigl(U_{1,\eps,\lambda}(\sigma_0)I_M,\ldots,U_{q,\eps,\lambda}(\sigma_0)I_M,V_{\eps,\lambda}(\sigma_0)I_M\bigr).
	$$
	Finally, since $\eps$ is fixed, it can be checked directly that
	$$
	\sup_{m\geq 0}\e\bigl|U_{l,\eps,\lambda}(\sigma_m)\bigr|^2<\infty\,\,\mbox{and}\,\,\sup_{m\geq 0}\e\bigl|V_{\eps,\lambda}(\sigma_m)\bigr|^2<\infty,
	$$
	which implies that as $M\to\infty,$
	\begin{align*}
		\sup_{m\geq 0}\e\bigl|U_{l,\eps,\lambda}(\sigma_m)I_M^c\bigr|\leq \bigl(\sup_{m\geq 0}\e\bigl|U_{l,\eps,\lambda}(\sigma_m)\bigr|^2\bigr)^{1/2}\bigl(\e I_M^c\bigr)^{1/2}\to 0,\\
		\sup_{m\geq 0}\e V_{\eps,\lambda}(\sigma_m)I_M^c \leq\bigl( \sup_{m\geq 0}\e V_{\eps,\lambda}(\sigma_m)^2\bigr)^{1/2}\bigl(\e I_M^c\bigr)^{1/2}\to 0.
	\end{align*}
	From Slutsky’s theorem, our first assertion follows.
\end{proof}

\begin{proof}[\bf Proof of Theorem \ref{invariancestep2}]
	It is easy to see that $V_{\eps,\lambda}(\sigma_0)>0$ a.s.. The continuous mapping theorem and Lemma \ref{UVconvergence} imply that \[\frac{\prod_{l\le q}U_{l,\eps,\lambda}(\sigma_\eta) }{V_{\eps,\lambda}^q (\sigma_\eta)} \xrightarrow[\eps\downarrow 0]{d}\frac{\prod_{l\le q}U_{l,\eps,\lambda}(\sigma_0) }{V_{\eps,\lambda}^q (\sigma_0)}.\]
	Next, since these are bounded random variables, we can apply the bounded convergence theorem to conclude the proof for our first statement. As for the second one, note that $\mathcal{A}_{n,\eps}(\sigma_0)>0$ and $\mathcal{B}_{n,\eps,\lambda}(\sigma_0)>0$ a.s.. It follows from the continuous mapping theorem and Lemma \ref{UVconvergence}, $A_{n,\eps}(\sigma_\eta)$ converges to $A_{n,\eps}(\sigma_0)$ and $B_{n,\eps,\lambda}(\sigma_0)$ converges to $B_{n,\eps,\lambda}(\sigma_0)$ weakly as $\eta\to 0.$ To show that their expectations also converge, note that the argument before Lemma \ref{normofYbounded} in Subsection \ref{prop1} readily implies that $(A_{n,\eps}(\sigma_\eta))_{\eta>0}$ is uniformly integrable and the same argument also yields the uniform integrability of $(B_{n,\eps,\lambda}(\sigma_\eta))_{\eta>0}$. Consequently,
	$$
	\lim_{\eta \downarrow 0}\e {A}_{n,\eps}(\sigma_\eta)=\e {A}_{n,\eps}(\sigma_0)\,\,\mbox{and}\,\,	\lim_{\eta \downarrow 0}\e {B}_{n,\eps,\lambda}(\sigma_\eta)=\e {B}_{n,\eps,\lambda}(\sigma_0).
	$$ 
	In particular, letting $n=1$ and $\lambda=1$ implies our second assertion.
\end{proof}

\section{Proof of Lemma \ref{expectationlemma}}\label{sec:proof of expectationlemma}
Recall the L\'evy disorder $J$ defined in  \eqref{def:J}. In this section, we will establish a lemma which is crucial for proving Theorems  \ref{freeenergy}, \ref{fluctuationfreeenergy}, \ref{hightempoverlap}, and \ref{main:thm2} for the general L\'evy model.

\begin{lemma} \label{expectationlemma}
	Suppose $0<\alpha<2$. Let $f:\mathbb{R}\to \mathbb{R}$ be continuously differentiable and let $\bar f(x):=(f(x)+f(-x))/2$ for $x>0$.
	\begin{itemize}
		\item[$(i)$] Suppose $f(0)=0$, $ |f'(x)|=O(|x|^\delta)$ as $x\to0$, and $| f'(x)|=O(|x|^\eps)$ as $x\to\pm\infty$ for some $-\infty<\eps<\alpha-1<\delta$. Then \[\lim_{N\to\infty }N\e f(  J) = \alpha \int_0^\infty \frac{\bar f(x)}{x^{1+\alpha}}dx.\]
		
		\item[$(ii)$] Let $\eta>0.$ Suppose $|f'(x)|=O(|x|^\eps)$ as $x\to\pm\infty$ for some $-\infty<\eps<\alpha-1$. Then \[\lim_{N\to\infty} N\e f(  J) \1_{|  J|\geq \eta}=\alpha\int_{\eta}^{\infty} \frac{\bar f(x)}{x^{1+\alpha } }dx.\]
		\item[$(iii)$] Let $\eta>0.$ Suppose $f(0)=0$ and $|f'(x)|=O(|x|^\delta)$ as $x\to0$ for some $\alpha-1<\delta<\infty$. Then \[\lim_{N\to\infty} N\e f(  J)\1_{|  J|\le \eta}=\alpha \int_0^\eta \frac{\bar f(x)}{x^{1+\alpha}}dx.\]
	\end{itemize}
\end{lemma}

Before we proceed with the proof, we first collect some useful properties.  We refer the readers to \cite{resnick} as a detailed reference. 
First of all, by the definition of $a_N$, it is easy to see that \begin{align}
	1\ge  N\p(|  X|>a_N)\to 1 \label{a_N}
\end{align} as $N\to\infty$. A well-known fact is that \begin{align}
	a_N=N^{1/\alpha} \ell(N) \label{a_N asymptotics}
\end{align} for some slowly varying function $\ell$ at $\infty$ which depends on $L$. Moreover, the following uniform convergence holds for any $0<a<\infty$:  \begin{align}
	\lim_{y\to\infty}\sup_{t\in [a,\infty)}\Bigl|\frac{\p (|  X|>ty)}{\p(|  X|>y)}-\frac{1}{t^\alpha}\Bigr|=0. \label{unifconv}
\end{align}  
Potter's bound says that if $L$ is a slowly varying function, for any $\eta>0$, there exists $y_0>0$ such that for all $y\ge y_0$ and $t\ge 1$, \begin{align}
	\frac{1-\eta}{t^{\eta}}<\frac{L(ty)}{L(y)} < \frac{1+\eta}{t^{-\eta}}.\label{potter}
\end{align}
An immediate consequence of \eqref{a_N asymptotics} and \eqref{potter} is that for any $\eps>0$, \begin{equation} 
 \label{order of a_N}
	N^{-\eps+1/\alpha}\ll a_N \ll N^{\eps+1/\alpha}.
\end{equation}

We only prove Lemma \ref{expectationlemma}$(i)$. 
Conditions on $f$ and Karamata's theorem imply that $\e|f(  J)| <\infty$, so we have $\e f(  J)=\e f(-  J)=\e \bar f (  J)= \e \bar f (|  J|)$ by the symmetry of $  J$. Then, $f(0)=0$ and Fubini's theorem  gives \begin{align*}
	\e f(  J)= \int_0^\infty (\bar f)'(x) \p (|  J|>x)dx. 
\end{align*}  Note that \[N\p(|  J|>x)=N\p(|  X|>a_N)\frac{\p(|  X|>xa_N)}{\p(|  X|>a_N)},\] so \eqref{a_N} and \eqref{unifconv} imply that for $x>0$, \begin{align}
	\lim_{N\to\infty}N\p(|  J|>x)= \frac{1}{x^{\alpha}}. \label{eq:probconv}
\end{align} Inequality \eqref{potter} implies that for any $0<\eta<\min\{\delta-\alpha+1, \alpha-\eps-1 \}$, there exists $x_0>0$ such that for all $x>0$ and $N$ large enough satisfying $a_N\wedge xa_N>x_0$, we have \begin{align}
	\1_{a_N\wedge x a_N >x_0}\frac{\p(|  X|>xa_N)}{\p(|  X|>a_N)}\le \frac{1+\eta}{x^{\alpha-\eta}}\1_{x\ge 1}+\frac{1}{(1-\eta)x^{\alpha+\eta}}\1_{x_0/a_N<x<1}.\label{majorant}
\end{align}
On the interval $[1,\infty)$, applying the dominated convergence theorem with \eqref{eq:probconv} and \eqref{majorant}, we have \begin{align}
	\lim_{N\to\infty}\int_1^\infty (\bar f)'(x) N\p(|  J|>x) dx= \int_1^\infty \frac{(\bar f)'(x)}{x^\alpha}dx.\label{dct1}
\end{align}
There exists $\delta'>0$ small enough such that $|f'(x)|\le C|x|^{\delta}$ for all $|x|<\delta'$, where $C$ is a large enough constant.
If  $N$ is large enough so that $x_0/a_N<\delta'$, \begin{align*}
	\int_{0}^{x_0/a_N}|(\bar f)'(x)|N\p (|  J|>x)dx \le   CN\int_{0}^{x_0/a_N}x^{\delta} dx =\frac{Cx_0^{1+\delta}}{1+\delta} N a_N^{-1-\delta}.
\end{align*}  \eqref{a_N asymptotics} and \eqref{potter} imply that $\lim_{N\to\infty}Na_N^{-1-\delta}=0$, so \begin{align}
	\lim_{N\to\infty}\int_{0}^{x_0/a_N}(\bar f)'(x)N\p (|  J|>x)dx=0. \label{dct2}
\end{align}
Now, if $N$ is large enough so that $x_0/a_N<1$, then from \eqref{a_N} and \eqref{majorant}, \begin{align*}
	\1_{x_0/a_N<x<1}|(\bar f)'(x)|N\p(|  X|>xa_N)&\le \frac{|(\bar f)'(x)|}{(1-\eta)x^{\alpha+\eta}}\1_{x_0/a_N<x<1}\\
	&\le \frac{Cx^\delta \1_{0<x<\delta'}+|(\bar f)'(x)|\1_{\delta'\le x\le1}}{(1-\eta)x^{\alpha+\eta}} \\
	&\le \frac{1}{1-\eta} \bigl(Cx^{\delta-\alpha-\eta}+|(\bar f)'(x)|(\delta')^{-\alpha-\eta} \bigr),
\end{align*}
where the last line is integrable on $(0,1)$. Thus, the dominated convergence theorem gives \begin{align}
	\lim_{N\to\infty}\int_{x_0/a_N}^{1}(\bar f)'(x)N\p (|  J|>x)dx= \int_0^1 \frac{(\bar f)'(x)}{x^\alpha}dx.\label{dct3}
\end{align} 

Finally, combining \eqref{dct1},\eqref{dct2}, and \eqref{dct3} yields the desired result  \begin{align*}
	\lim_{N\to\infty}N\e f(  J)=\lim_{N\to\infty}\int_0^\infty (\bar f)'(x) N\p(|  J|>x) dx= \int_0^\infty \frac{(\bar f)'(x)}{x^\alpha}dx=\alpha\int_{0}^{\infty} \frac{\bar f(x)}{x^{\alpha +1} }dx,
\end{align*}  where the assumptions on $f$ assure that the integration by parts holds in the last equality.

\section{Proofs for General Power-Law Distributions}\label{sec:generalization}


Recall the heavy-tailed distribution $X$, the L\'evy disorder $J$, and the free energy $F_N$ for the general L\'evy model at the beginning of Section \ref{Sec:results}.

\subsection{Theorem \ref{con}}
In the general heavy-tailed setting, \eqref{eq:D} and \eqref{eq:D_bd} become 
\begin{align*}
    ND_k&= \e \bigl[\ln \la \exp(\beta a_N^{-1} X_{e_k} \sigma_{e_k})\ra_k\big|\mathcal F_{k}\bigr]- \e \bigl[\ln \la \exp(\beta a_N^{-1} X_{e_k} \sigma_{e_k})\ra_k\big|\mathcal F_{k-1}\bigr],
    \\
    |D_k |&\le \beta N^{-1}a_N^{-1}\bigl(|X_{e_k}|+\e|X_{e_k}|\bigr),
    \end{align*}
 respectively.
Then, \eqref{eq:D final bound} becomes
\begin{equation*}
    \e \Bigl| \frac{1}{N}\ln Z_N  - \frac{1}{N} \e  \ln Z_N \Bigr|^p\le \frac{C \e|X|^p}{    N^{p-2} a_N^{p}},
\end{equation*} 
and \eqref{order of a_N} guarantees the desired bound for $N^{-1}\ln Z_N$. We can proceed in a similar way for $\ln \bar Z_N$.

\subsection{Theorems \ref{freeenergy} and \ref{main:thm2}}\label{sec:extension of main:thm2}

Let $X^P$, $J^P$, and $F_N^P$ be respectively $X,$ $J$, and $F_N$ associated to the Pareto distribution. 
To establish the universality for Theorems \ref{freeenergy} and \ref{main:thm2}, it suffices to show

\begin{theorem}[Universality of free energy]\label{thm:universality}
 For $1<\alpha<2$ and $\beta>0$, we have \[\limsup_{N\to\infty}|\e F_N-\e F_N^P|=0.\]
\end{theorem}
\begin{remark}
\cite[Proposition 1.2]{Starr} established  a quantitative universality for general SK-type free energies in terms of the statistics of disorder distributions. The following proof adapts a similar approach.
\end{remark}

\begin{proof}
	Consider an arbitrary enumeration $\{(i_t,j_t):1\leq t\leq N(N-1)/2\}$ of $\{(i,j):1\leq i<j\leq N\}$. For any $1\leq t\leq N(N-1)/2+1$, define 
	\begin{equation*}
		   F_{N,t}=\frac{1}{N}\ln \sum_{\sigma}\exp\Bigl(\beta\sum_{r<t}J_{i_rj_r}^P\sigma_{i_r}\sigma_{j_r}+\beta\sum_{r\geq t}   J_{i_rj_r}\sigma_{i_r}\sigma_{j_r}\Bigr).
	\end{equation*}
	Write
	\begin{align*}
		F_{N}^P-   F_{N}  &=\sum_{t=1}^{N(N-1)/2}\bigl(   F_{N,t+1}-   F_{N,t}\bigr)\\
		&=\frac{1}{N}\sum_{t=1}^{N(N-1)/2}\bigl(\ln \bigl\la e^{\beta  J_{i_tj_t}^P\sigma_{i_t}\sigma_{j_t}}\bigr\ra_{t}-\ln\bigl \la e^{\beta    J_{i_tj_t}\sigma_{i_t}\sigma_{j_t}}\bigr\ra_{t}\bigr),
	\end{align*}
	where $\la \cdot\ra_t$ is the Gibbs expectation associated to the Hamiltonian
	$$
	   H_{N,t}(\sigma)=\beta\sum_{r<t}J_{i_rj_r}^P\sigma_{i_r}\sigma_{j_r}+\beta\sum_{r>t}   J_{i_rj_r}\sigma_{i_r}\sigma_{j_r}.
	$$
	Now, similar to \eqref{add:add:eq1}, we can write
	\begin{equation*}
		\e \ln \bigl\la e^{\beta    J_{i_tj_t}\sigma_{i_t}\sigma_{j_t}}\bigr\ra_{t}= \e \ln \cosh \beta    J_{i_tj_t}- \e \sum_{r=1}^\infty \frac{1}{2r} \tanh^{2r}(\beta   J_{i_tj_t})\cdot  \la \sigma_{i_t}\sigma_{j_t}\ra_t^{2r}.
	\end{equation*}
	Let $K>0$ and $k\geq 1.$ From this and using the independence between $   J_{i_tj_t}$ and $\la \cdot\ra_t,$
	\begin{align*}
		\e\ln \bigl\la e^{\beta     J_{i_tj_t}\sigma_{i_t}\sigma_{j_t}}\bigr\ra_{t}&=\e\bigl[\ln \bigl\la e^{\beta     J \sigma_{i_t}\sigma_{j_t}}\bigr\ra_{t};|   J |> K\bigr]+\e\bigl[\ln \cosh \beta    J ;|   J |\leq K\bigr]\\
		&-\sum_{r=1}^\infty \frac{1}{2r}\e \bigl[\tanh^{2r}(\beta    J );|   J |\leq K\bigr]\cdot \e \la \sigma_{i_t}\sigma_{j_t}\ra_t^{2r},
	\end{align*}
	where $   J$ is independent of $\la \cdot\ra_t.$
	Note that $ |\ln  \la e^{\beta     J \sigma_{i_t}\sigma_{j_t}} \ra_{t} |\leq \beta|   J|$
	and
	\begin{align*}
		\sum_{r=k+1}^\infty \frac{1}{2r}\e \bigl[\tanh^{2r}(\beta    J );|   J |\leq K\bigr]\cdot \e \la \sigma_{i_t}\sigma_{j_t}\ra_t^{2r}
		&\leq \e\bigl[\tanh^2(\beta    J );|   J |\leq K\bigr]\sum_{r=k}^\infty \tanh^{2r}(\beta K)\\
		&=\e\bigl[\tanh^2(\beta    J);|   J|\leq K\bigr]\frac{\tanh^{2k}(\beta K)}{1-\tanh^2(\beta K)}.
	\end{align*}
	Consequently,
	\begin{align*}
		&\Bigl|\e\ln \bigl\la e^{\beta     J \sigma_{i_t}\sigma_{j_t}}\bigr\ra_{t}-\e\bigl[\ln \cosh \beta    J ;|   J |\leq K\bigr]
		+\sum_{r=1}^k \frac{1}{2r}\e \bigl[\tanh^{2r}(\beta    J );|   J |\leq K\bigr]\cdot \e \la \sigma_{i_t}\sigma_{j_t}\ra_t^{2r}\Bigr|\\
		&\leq \beta \e\bigl[|    J |;|   J |> K\bigr]+\e\bigl[\tanh^2(\beta    J );|   J |\leq K\bigr]\frac{\tanh^{2k}(\beta K)}{1-\tanh^2(\beta K)}.
	\end{align*}
	Using Lemma \ref{expectationlemma}, we have
	\begin{align*}
		&\limsup_{k\to\infty}\limsup_{N\to\infty}\Bigl|\frac{1}{N}\sum_{t=1}^{N(N-1)/2}\e\ln \bigl\la e^{\beta     J_{i_tj_t}\sigma_{i_t}\sigma_{j_t}}\bigr\ra_{t}-\frac{\alpha}{2}\int_0^K\frac{\ln \cosh(\beta x)}{x^{\alpha+1}}dx\\
		&\qquad\qquad+\sum_{r=1}^k \frac{\alpha}{2r}\int_{0}^K\frac{\tanh^{2r}(\beta x)}{x^{\alpha+1}}dx\cdot \frac{1}{N^2}\sum_{t=1}^{N(N-1)/2}\e \la \sigma_{i_t}\sigma_{j_t}\ra_t^{2r}\Bigr|\\
		&\leq\limsup_{k\to\infty}\Bigl( \frac{\alpha\beta}{2}\int_{K}^\infty \frac{1}{x^\alpha}dx+\frac{\alpha}{2}\int_{0}^K\frac{\tanh^2(\beta x)}{x^{\alpha+1}}dx\cdot\frac{\tanh^{2k}(\beta K)}{1-\tanh^2(\beta K)}\Bigr)\\
		&=\frac{\alpha\beta}{2}\int_{K}^\infty \frac{1}{x^\alpha}dx=\frac{\alpha \beta}{2(\alpha-1)}K^{1-\alpha}.
	\end{align*}
	Note that by substituting $   J$ with $J^P$, the same computation still carries through and yields the same inequality. As a result,
	\begin{align*}
		\limsup_{N\to\infty}\Bigl|\frac{1}{N}\sum_{t=1}^{N(N-1)/2}\bigl(\e\ln \bigl\la e^{\beta     J_{i_tj_t}\sigma_{i_t}\sigma_{j_t}}\bigr\ra_{t}-\e\ln \bigl\la e^{\beta  J_{i_tj_t}^P\sigma_{i_t}\sigma_{j_t}}\bigr\ra_{t}\bigr)\Bigr|\leq \frac{\alpha \beta}{\alpha-1}K^{1-\alpha}.
	\end{align*}
	Since this holds for any $K>0,$ sending $K\to\infty$ completes our proof.
\end{proof}

\subsection{Theorem \ref{fluctuationfreeenergy}}
\subsubsection*{Proof of Theorem \ref{thm: generalized: fluctuation of Zbar}}
We adapt the proof of Theorem 3.8.2 in \cite{durrett_2019}. Recall $n=N(N-1)/2.$
Let $\eps>0$ and   \begin{align*}
	&S_1^\eps= \sum_{i<j}a_N\ln \ch (\beta   J_{ij})\1_{\{a_N\ln \ch (\beta   J_{ij}) \le \eps a_n \}}, \\ &S_2^\eps= \sum_{i<j}a_N\ln \ch (\beta   J_{ij})\1_{\{a_N\ln \ch (\beta   J_{ij}) > \eps a_n \}},
\end{align*} so that \[a_N\ln \bar Z_N =S_1^\eps+S_2^\eps.\]
Furthermore, let auxiliary quantities \begin{align*}
	&d_N= n\e \bigl(a_N\ln \ch (\beta   J) \1_{\{a_N\ln\ch(\beta   J)\le \beta a_n\}}\bigr),\\
	&d^\eps_N= n\e \bigl(a_N\ln \ch (\beta   J) \1_{ \{\eps a_n< a_N\ln\ch(\beta   J)\le \beta a_n\}} \bigr).
\end{align*}
Define the centered random variable \begin{align*}
	\bar S_1 ^\eps = S_1^\eps-(d_N-d^\eps_N)&= \sum_{i<j}\bigl(a_N\ln \ch (\beta   J) \1_{\{a_N\ln\ch(\beta   J)\le \eps a_n\}}\\&\qquad -\e a_N\ln \ch (\beta   J) \1_{\{a_N\ln\ch(\beta   J)\le \eps 
		a_n\}}\bigr).
\end{align*}
From these, we can write \begin{align}
	\frac{\ln \bar Z_N-a_N^{-1}d_N}{a_N^{-1}a_n}= \frac{\bar S_1^\eps}{a_n}+\frac{S_2^\eps-d^\eps_N}{a_n}. \label{eq:general fluctuation: decomp}
\end{align}
We claim that the first term in \eqref{eq:general fluctuation: decomp} can be controlled uniformly in $N$:  \begin{align}
	\limsup_{N\to\infty}\e (a_n^{-1}\bar S_1^\eps)^2 \le \frac{2\beta^\alpha \eps^{2-\alpha}}{2-\alpha}. \label{eq:general fluctuation: pt 1}
\end{align}
Indeed, \begin{align*}
	\e (a_n^{-1}\bar S_1^\eps)^2&\le n \e \Bigl(\frac{a_N}{a_n}\ln \ch \Bigl(\beta \frac{  X}{a_N}\Bigr) \1_{\{a_N\ln\ch(\beta   Xa_N^{-1})\le \eps a_n\}}\Bigr)^2 \\&= \int_0^\eps 2x n\p \bigl(|  X|> \beta^{-1}a_N\mathrm{arccosh}(\exp(xa_N^{-1}a_n))\bigr)dx\\
	&=n \int_0^\eps 2x \p \bigl(|  X|> \beta^{-1}a_Na_n^{-1}\mathrm{arccosh}(\exp(xa_N^{-1}a_n)) a_n\bigr) dx \\
	&\le n \int_0^\eps 2x \p \bigl(|  X|> \beta^{-1}x a_n\bigr) dx,
\end{align*} where we used the fact that $\mathrm{arccosh}(e^x)\ge x$ for $x\ge0$ in the last inequality. From \eqref{a_N}, \eqref{unifconv}, and \eqref{majorant}, the dominated convergence theorem implies that the last expression converges to the right-hand side of \eqref{eq:general fluctuation: pt 1} as required.

Next, we consider the second term in \eqref{eq:general fluctuation: decomp}.
We claim that \begin{align}
	\lim_{N\to\infty }\e \exp(itS_2^\eps/a_n)=\exp\Bigl(\beta^\alpha\int_\eps^\infty  (e^{itx}-1)\alpha x^{-1-\alpha}  dx\Bigr).\label{eq:general fluctuation: pt 2}
\end{align}
Let $I_N (\eps)\colonequals|\{(i,j): a_N\ln \ch (\beta   X_{ij}/a_N) > \eps a_n\}|$. Given $I_N(\eps) =k$, the quantity $S_2^\eps/a_n$ is a sum of $k$ independent random variables supported on the interval $[\eps,\infty)$ with the tail distribution \begin{equation}
	\frac{\p(a_N\ln \ch (\beta   X/a_N) > x a_n)}{\p(a_N\ln \ch (\beta   X/a_N) > \eps a_n)}, \,\,\forall x\ge \eps. \label{eq:conditional distribution}
\end{equation} Let $\psi^\eps_N$ be the characteristic function corresponding to the conditional distribution \eqref{eq:conditional distribution}. Note that from \eqref{unifconv} and  $\lim_{x\to\infty} \mathrm{arccosh}(e^{x})/x=1$,  the conditional distribution converges as $N\to\infty,$ \begin{align*}
	\frac{\p(a_N\ln \ch (\beta   X/a_N) > x a_n)}{\p\bigl(a_N\ln \ch (\beta   X/a_N) > \eps a_n\bigr)} &= \frac{\p(a_N\ln \ch (\beta   X/a_N) > x a_n)}{\p (\beta   X> xa_n)}\frac{\p (\beta   X> \eps a_n)}{\p(a_N\ln \ch (\beta   X/a_N) > \eps a_n)} \\ &\quad \cdot \frac{\p (\beta   X> xa_n)}{\p (\beta   X> \eps a_n)} \to \eps^\alpha x^{-\alpha},
\end{align*} which implies that, as $N\to\infty$, $\psi_N^\eps$ converges to the limit \[\psi^\eps(t)= \int_\eps^\infty  e^{itx}\alpha \eps^\alpha x^{-1-\alpha}  dx, \ t\in\mathbb R. \] Moreover, \eqref{unifconv}, \eqref{eq:probconv}, and $\lim_{x\to\infty} \mathrm{arccosh}(e^{x})/x=1$ imply  \begin{equation}
	\lim_{N\to\infty}n\p(a_N\ln \ch (\beta   X/a_N) > \eps a_n)= \beta^\alpha \eps^{-\alpha}.\label{eq:general fluctuation prob conv}
\end{equation}  It is easy to see that for any fixed $k\ge 1$,  as $n\to\infty$, \[{n\choose k}\frac{1}{n^k}\uparrow \frac{1}{k!}. \] Moreover, $\lim_{N\to\infty}(1-f_n/n)^n= e^{-f}$ and $(1-f_n/n)^n\le e^{-0.5 f}$ for large enough $n$ whenever a sequence $(f_n)_n$ satisfies $\lim_{N\to\infty}f_n=f$. Thus, by the dominated convergence theorem, \begin{align*}
	\e \exp\Bigl(\frac{itS_2^\eps}{a_n}\Bigr)&= \sum_{k=0}^n{n \choose k}\p\bigl(a_N\ln \ch (\beta   X/a_N) > \eps a_n\bigr)^k\\
 &\qquad\cdot\bigl(1-\p(a_N\ln \ch (\beta   X/a_N) > \eps a_n)\bigr)^{n-k}(\psi_N^\eps(t))^k\\ &\to \sum_{k\ge 0}\frac{1}{k!}(\beta^\alpha\eps^{-\alpha})^k \exp(-\beta^\alpha\eps^{-\alpha})(\psi^\eps(t))^k\\
 &= \exp\bigl(-\beta^\alpha\eps^{-\alpha}(1-\psi^\eps(t))\bigr)=\exp\Bigl(\beta^\alpha\int_\eps^\infty  (e^{itx}-1)\alpha x^{-1-\alpha}  dx\Bigr),
\end{align*} as $N\to\infty$ for any $t\in\mathbb R$, thereby proving \eqref{eq:general fluctuation: pt 2}.

Lastly, from \eqref{eq:general fluctuation prob conv},  \begin{align}
	\frac{d^\eps_N}{a_n}&=  n\e \Bigl(\frac{a_N}{a_n}\ln \ch \Bigl(\beta \frac{  X}{a_N}\Bigr) \1_{ \{\eps a_n< a_N\ln\ch(\beta   X/a_N)\le \beta a_n\}} \Bigr)\nonumber\\
	&= \int_0^\infty  n\p \Bigl(\eps \vee x<\frac{a_N}{a_n}\ln \ch \Bigl(\beta \frac{  X}{a_N}\Bigr)\le \beta\Bigr)dx \nonumber \\
	&= \eps n\p \Bigl(\eps <\frac{a_N}{a_n}\ln \ch \Bigl(\beta \frac{  X}{a_N}\Bigr)\le \beta \Bigr)+\int_\eps ^\beta n\p \Bigl( x<\frac{a_N}{a_n}\ln \ch \Bigl(\beta \frac{  X}{a_N}\Bigr)\le \beta \Bigr) dx \nonumber\\
	&\to \eps \beta^\alpha(\ \eps^{-\alpha}-\beta^\alpha)+\int_\eps^\beta \beta^\alpha(x^{-\alpha}-\beta^\alpha)dx= \beta^\alpha\int_\eps^\beta x\alpha x^{-1-\alpha}dx. \label{eq:general fluctuation: pt 3}
\end{align}
From \eqref{eq:general fluctuation: pt 2} and \eqref{eq:general fluctuation: pt 3}, we have \begin{align}
\nonumber	\lim_{N\to\infty}\e \exp\Bigl(it\frac{S_2^\eps-d^\eps_N}{a_n}\Bigr)& = \exp\Bigl(\beta^\alpha\int_\beta^\infty  (e^{itx}-1)\alpha x^{-1-\alpha}  dx \\
 &\,\,+\beta^\alpha\int_\eps ^\beta   (e^{itx}-1-itx)\alpha x^{-1-\alpha}  dx\Bigr). \label{eq:general fluctuation: pt 4}
\end{align}
Note that the second term on the right-hand side converges as $\eps\to0$ since $e^{itx}-1-itx =O(x^2)$ as $x\to 0$.
Combining \eqref{eq:general fluctuation: decomp}, \eqref{eq:general fluctuation: pt 1}, and \eqref{eq:general fluctuation: pt 4}, it is straightforward to see that Lemma \ref{cut} yields the desired result.

\subsubsection*{Proof of Theorem \ref{fluctuation hat Z}} 
For $\eps>0$, from \eqref{a_N}, \[  J\1_{\{|  J|>\eps\}}\stackrel{d}{=}Bg_{\eps,N}, \] where $B\sim \mathrm {Ber}((1+o(1))\eps^{-\alpha}/N)$ and $g_{\eps,N}$ has a symmetric distribution with the tail probability
$\p(|g_{\eps,N}| > x) = (\eps/x)^\alpha L(a_N x)/L(a_N \eps), x > \eps$ and they are independent. It is easy to see that  $g_{\eps,N}\stackrel{d}{\to}g_\eps$ as $N\to\infty$, where $g_\eps$ is as given in \eqref{eq:g_eps density}.

Choose $\eta>0$ small such that $(1+\eta)\beta/\beta_\alpha <1$. The inequality in the statement of Lemma \ref{Bigraphdie} should read as \[\sup_{N \ge 1} \e \Bigl(\sum_{\partial \Gamma = \emptyset, |\Gamma|\geq m} w(\Gamma)  \Bigr)^2 \leq  Ce^{\sqrt{2m}}((1+\eta)\beta/\beta_{\alpha})^{\alpha m}, \,\,m \ge 1.\] Similarly, the claim relevant to \eqref{eq:cycle_claim} is now that for sufficiently large $N$, we have \[\sum_{\gamma: |\gamma|=m} \e w(\gamma)^2 \le \frac{1}{2m}((1+\eta)\beta/\beta_{\alpha})^{\alpha m},\] which in the proof we use that 
\[\limsup_{N\to\infty}N\e \hth^2(\beta J)\le ((1+\eta)\beta/\beta_\alpha)^\alpha\] as apparent from Lemma \ref{expectationlemma}$(i)$.
The proof of Lemma \ref{lem:fixed_m_approx} remains essentially unchanged. 
As for the proof of Theorem \ref{fluctuation hat Z}, the convergence in distribution given by \eqref{eq:charac1} holds since Lemmas \ref{cyclepoisson} and \ref{vertexdisjoint} continue to hold if we change the parameter $p$ with $p_N = p(1+o(1))$, and the rest of the argument remains the same.

\subsection{Theorem \ref{hightempoverlap}}
We recall two useful facts that 
$a_N=N^{1/\alpha}\ell (N)$ for some slowly varying function $\ell$ at $\infty$ and that slowly varying functions decay or grow at most sub-polynomially fast at $\infty$, see \cite{resnick}. Moreover, $X$ is $L^p$-integrable for any $0<p<\alpha$. 

\subsubsection*{Proof of \eqref{siteoverlapzero}}
First of all, we claim that Proposition \ref{prop1'} also holds in the limit for the general L\'evy model,
\begin{equation}\label{eq: generalized: prop1'}
	\lim_{N\to\infty}\Bigl|\e\frac{1}{N}\sum_{i<j}  J_{ij}\la \sigma_i\sigma_j\ra-\frac{c_1}{2}\bigl(1- \e\bigl\la R_{  L}^2\bigr\ra\bigr)\Bigr|=0.
\end{equation}
To see this, recall the proof of Proposition \ref{prop1'}. Note that the bounds in \eqref{cavityineq} and \eqref{overlap:eq3} should be replaced by $C_0a_N^{-1}\e|  X|$ and that \[  \e\Bigl[\frac{|  J|}{1-|\tanh(\beta   J)|} ;|  J|\leq K\Bigr] \le e^{2\beta K} \frac{\e|  X|}{a_N}.
\]
It follows that for any $K>0$,
\begin{align*}
	&	\e\Bigl|\frac{1}{N}\sum_{i<j}  J_{ij}\la \sigma_i\sigma_j\ra-\frac{1}{N}\sum_{i<j}  J_{ij}f\bigl(\la \sigma_i\sigma_j\ra_{ij}',\beta   J_{ij}\bigr)\Bigr|\\
	&\leq \frac{(\e|  X|)^2 e^{2\beta K}}{N^{2/\alpha-1} \ell(N)^2}+N\e[|  J|;|  J|\geq K].
\end{align*}
For fixed $K,$ the first vanishes as $N\to\infty$, and from Lemma \ref{expectationlemma}$(ii)$, the second term has the limit $\alpha(\alpha-1)^{-1}K^{-\alpha+1}.$ Therefore, we can send $K\to\infty$ to obtain
\begin{align*}
	&	\limsup_{N\to\infty}\e\Bigl|\frac{1}{N}\sum_{i<j}  J_{ij}\la \sigma_i\sigma_j\ra-\frac{1}{N}\sum_{i<j}  J_{ij}f\bigl(\la \sigma_i\sigma_j\ra_{ij}',\beta   J_{ij}\bigr)\Bigr|=0.
\end{align*}
The proof of our claim is finished by \eqref{add:equation1} and  \eqref{eq:summationbyparts}.
Now, the proof of \eqref{siteoverlapzero} is identical to the one presented in Section \ref{subsec:add2}.

\subsubsection*{Proof of \eqref{eq: overlap and temperature}}
Note that $c_k\to \gamma_k$ by Lemma \ref{expectationlemma}.  From \eqref{eq: generalized: prop1'}, \eqref{eq: limit of R^2} still holds for general heavy-tails.
The proof is finished by noticing that Theorem \ref{superadditivity} and Theorem \ref{thm:universality} imply \eqref{eq: expectation of Hamiltonian bounded in beta} for general heavy-tails.

\subsubsection*{Proof of \eqref{bondoverlapnonzero}}

We need to match the first two moments as in Proposition \ref{bondoverlap:firstmoment}. With the help of Lemma \ref{expectationlemma}$(ii)$ and $(iii)$, for any $K>0,$  we have that 
$$
\lim_{N\to\infty}N\e\tanh^2(\beta   J)1_{\{|  J|\geq K\}}=2C_K
$$
and that as long as  $N$ is large,
\begin{align*}
	\e\tanh^2(\beta   J)1_{\{|  J|\leq K\}}&\leq \frac{2\alpha \beta^2}{N(2-\alpha)}K^{2-\alpha},\,\,\p(|  J|\geq K)\le \frac{2}{NK^\alpha}.
\end{align*}
These readily ensure that Proposition \ref{bondoverlap:firstmoment} also holds for the general L\'evy model by an identical argument. Finally, note that $  M:=\sum_{i<j}\mathbbm{1}_{\{|  J|\geq K\}}\sim \mbox{Binomial}(N(N-1)/2,\p(|  J|\geq K))$ satisfies $N  M^{-1}\mathbbm{1}_{\{  M\geq 1\}}\to 2K^\alpha$ in probability. Following the proof in Section \ref{subsection:add:add3} yields \eqref{bondoverlapnonzero} for the general L\'evy model.

\subsection{Theorem \ref{thm: generalized: alpha=1: free energy}}
	Let us write $c_N\sim d_N$ if $\lim_{N\to\infty} c_N/ d_N =1$ for any positive real sequences $(c_N)_{N\ge 1}$ and $(d_N)_{N\ge 1}$. 
 Let us write $g=g(x)=\mathrm{arccosh}(e^x)$ for the inverse function of the strictly increasing function $x\mapsto \ln \ch(x)$ on $[0,\infty)$ throughout this section.
	Recall the notations $a_N$, $b_N= a_N^{-1} a_{{N\choose 2}}$, and \[h_N =\int_0^{b_N}   \hth( y)N \p (|  X|> ya_N)dy, \quad N\ge 1.\]
	
	\begin{lemma}\label{lem: generalized: centering}
		For $0<\alpha<2$ and $\beta >0$, it holds that \begin{equation*}
			N\e \Bigl(\ln \ch (\beta   J) \1_{\{a_N\ln\ch(\beta   J)\le a_{N\choose 2}\}}\Bigr) \sim \beta^{\alpha} h_N. 
		\end{equation*}
		Furthermore, if $0< \alpha\le 1$, then $1\ll h_N\ll N^{1-\alpha+\eps}$ for any $\eps>0$. 
		If $1\le \alpha <2$, then $Nh_N/b_N \to \infty$.
	\end{lemma}

	\begin{proof}
		Let us first show the last two assertions.
		Suppose $0<\alpha\le 1$. 
		By Fatou's lemma and \eqref{unifconv},  \begin{align*}
			\liminf_{N\to\infty}h_N &\ge \int_0^\infty \frac{ \hth( y)}{y^{\alpha}}dy=\infty .
		\end{align*} 
		On the other hand, using Potter's bound \eqref{potter}, we have that for any $\eps>0$, \[h_N \ll \int_0^{b_N}\hth(y)(1+o(1))y^{-\alpha+\eps}dy \ll b_N^{1-\alpha+\eps}\ll N^{1-\alpha +2\eps}.\]
		Note that the integral for $h_N$ is finite on any finite interval $[0,K)$ for $K>0$ by \eqref{dct2} and \eqref{dct3}.
		
		Suppose $1\le \alpha <2$.
		By a change of variable and Fatou's lemma,
		\begin{align*}
			\liminf_{N\to\infty}\frac{Nh_N}{b_N}&=  \liminf_{N\to\infty}\int_0^1 \hth(b_N y) N^2 \p \bigl(|  X|>y a_{{N\choose 2}}\bigr) dy
			\\ & \quad \ge\int_0^1 \frac{2}{y^{\alpha}}dy =\infty.
		\end{align*}
		
		Now, let us assume $0<\alpha<2.$
		Since the derivative of $ \ln \ch(x)$ is $\hth(x)$, Fubini's theorem shows \begin{equation*}
			N\e \Bigl(\ln \ch (\beta   J) \1_{\{a_N\ln\ch(\beta   J)\le a_{N\choose 2}\}}\Bigr)= \int_0^{\beta^{-1}g(b_N)} \beta  \hth(\beta y)N \p (|  X|> ya_N)dy.
		\end{equation*} 
		By a change of variable, the previous display is equal to \begin{equation} \label{eq: generalized: centering when alpha=1: 1}
			\int_0^{g(b_N)}\hth(y)N\p (|  X|>ya_N/\beta) dy.
		\end{equation} 
		
		We first analyze \eqref{eq: generalized: centering when alpha=1: 1} and show that \begin{equation}\label{eq: generalized: centering when alpha=1: step 1}
			\int_0^{g(b_N)}\hth(y)N\p (|  X|>ya_N/\beta) dy \sim \beta^{\alpha} \int_0^{g(b_N)}\hth(y)N\p (|  X|>ya_N) dy.
		\end{equation}
		The proof of Equations \eqref{dct2} and \eqref{dct3} applies to our case since the function $x\mapsto \ln\ch(x)$ is well-behaved near zero, so the integral on the interval $[0,\beta)$ of \eqref{eq: generalized: centering when alpha=1: 1} converges to a finite value: \begin{align} \label{eq: generalized: centering when alpha=1: 2}
			\lim_{N\to\infty}\int_0^{\beta}\hth(y)N\p (|  X|>ya_N/\beta) dy&=  \lim_{N\to\infty}\int_0^{1} \beta  \hth(\beta y)N \p (|  X|> ya_N)dy \nonumber
			\\&=  \int_0^{1}  \frac{\beta  \hth(\beta y)}{y^{\alpha}} dy <\infty.
		\end{align}
		Similarly, taking the integral only on the interval $[0,\beta)$ in the RHS of \eqref{eq: generalized: centering when alpha=1: step 1}, we have \begin{equation}\label{eq: generalized: centering when alpha=1: 3}
			\lim_{N\to\infty}\beta^{\alpha} \int_0^{\beta}   \hth( y)N \p (|  X|> ya_N)dy= \beta^{\alpha} \int_0^\beta \frac{\hth(y)}{y^{\alpha}}dy=\int_0^{1}  \frac{\beta  \hth(\beta y)}{y^{\alpha}} dy,
		\end{equation} which matches with \eqref{eq: generalized: centering when alpha=1: 2}.
		Notice that the integral on the interval $[\beta, g(b_N))$ of the \eqref{eq: generalized: centering when alpha=1: 1} can be rewritten \begin{align} \label{eq: generalized: centering when alpha=1: 4}
			&\int_\beta ^{g(b_N)} \hth( y)N \p (|  X|> ya_N/\beta)dy \nonumber
			\\&= \int_\beta^{g(b_N)} \hth( y) N \p( |  X|> y a_N) \frac{\p (|  X|> ya_N/\beta)}{\p( |  X|> ya_N)}dy.
		\end{align}
		Combining \eqref{eq: generalized: centering when alpha=1: 2}, \eqref{eq: generalized: centering when alpha=1: 3}, and \eqref{eq: generalized: centering when alpha=1: 4},  in order to demonstrate \eqref{eq: generalized: centering when alpha=1: step 1}, it suffices to show \begin{align}
			\nonumber   &\int_\beta ^{g(b_N)} \hth( y) N \p( |  X|> y a_N) \Bigl( \frac{\p (|  X|> ya_N/\beta)}{\p( |  X|> ya_N)} - \beta^{\alpha}\Bigr)dy\\
			\label{eq: generalized: centering when alpha=1: 6}   &= o\Bigl(  \int_0^{g(b_N)}\hth(y)N\p (|  X|>ya_N) dy \Bigr).
		\end{align}
		From \eqref{unifconv}, \[ \sup_{y\in [\beta,\infty)}\Bigl| \frac{\p (|  X|> ya_N/\beta)}{\p( |  X|> ya_N)} - \beta^{\alpha}\Bigr| = o(1),\] which readily affirms \eqref{eq: generalized: centering when alpha=1: 6}.
		
		Next, we show that \begin{equation} \label{eq: generalized: centering when alpha=1: step 2}
			\Bigl|\int_0^{g(b_N)}\hth(y)N\p (|  X|>ya_N) dy - \int_0^{b_N}\hth(y)N\p (|  X|>ya_N) dy\Bigr| \to 0.
		\end{equation}
		It is elementary to see that $x\mapsto g(x)- x$ is an increasing non-negative function converging to $\ln 2$.
		Using \eqref{a_N} and \eqref{unifconv},  the modulus of the difference in \eqref{eq: generalized: centering when alpha=1: step 2} is bounded by \begin{align*}
			\int_{b_N}^{g(b_N)}\hth(y)N\p (|  X|>ya_N) dy  &\le \sup_{y\in[b_N,\infty)} \Bigl| \frac{\p(|  X|>ya_N)}{\p(|  X|>a_N)}-\frac{1}{y^{\alpha}}\Bigr| \int_{b_N}^{g(b_N)}\hth(y) dy
			\\& \qquad + \int_{b_N}^{g(b_N)}\frac{\hth(y)}{y^{\alpha}}dy
			\\&\le o(1)+ \frac{\ln 2}{b_N^{\alpha}},
		\end{align*}  which vanishes as $N\to\infty$.
		
		Finally, we can combine \eqref{eq: generalized: centering when alpha=1: step 1} and \eqref{eq: generalized: centering when alpha=1: step 2} to conclude the desired asymptotic behavior.
	\end{proof}

	\begin{lemma}\label{lem: generalized: auxiliary}
		For $0<\alpha<2$ and $\beta >0$, it holds that 
            \begin{align*}
	&		 {N \choose 2} \e \Bigl(\ln \ch (\beta   J) \1_{\{\ln\ch(\beta   J)\le \beta b_N\}}\Bigr) - {N \choose 2} \e \Bigl(\ln \ch (\beta   J) \1_{\{\ln\ch(\beta   J)\le  b_N\}}\Bigr) 
    \\& = b_N \Bigl( \int_{\beta^{-1}}^1 \frac{\beta}{y^{\alpha}} dy +o(1) \Bigr).
		\end{align*}
	\end{lemma} 
\begin{proof}
    By a change of variable, the LHS of the above is 
    \begin{align*}
        &\int_{\beta^{-1}g(b_N)}^{\beta^{-1}g(\beta b_N)} \beta  \hth(\beta y) {N\choose 2} \p (|X|> ya_N)dy 
        \\& =   b_N\int_{(\beta b_N)^{-1}g(b_N)}^{(\beta b_N)^{-1}g(\beta b_N)}\beta \hth(\beta b_N y) {N\choose 2} \p \bigl(|X|> y a_{N\choose 2}\bigr) dy
        \\& = b_N  \Bigl( \int_{\beta^{-1}}^1 \frac{\beta}{y^{\alpha}} dy +o(1) \Bigr),
    \end{align*} where we have used \eqref{a_N}, \eqref{unifconv}, and the fact that $\lim_{x\to\infty}g(x)/x =1$ in the last equality of the previous display.
\end{proof}

	\begin{proof}[Proof of Theorem \ref{thm: generalized: alpha=1: free energy}]
		Lemma \ref{lem: generalized: centering} furnishes the desired properties of the sequence $h_N.$
		Let us derive the limit of the free energy.
By Lemma \ref{lem: generalized: centering} and \ref{lem: generalized: auxiliary}, we may rewrite \ref{thm: generalized: fluctuation of Zbar} as \[\frac{1}{b_N}\Bigl(\ln \bar Z_N - (1+o(1))\frac{N \beta h_N}{2} + e_N  \Bigr)\stackrel{d}{\to} \beta Y_\alpha\] for some sequence $|e_N|= O(b_N) $.
  Dividing both sides by $Nh_N/b_N$ which diverges to $\infty$ by Lemma \ref{lem: generalized: centering}, we have that for all $\beta>0$, \begin{equation*}
			\frac{1}{N h_N}\ln \bar Z_N  \stackrel{p}{\to} \frac{\beta}{2}
		\end{equation*} as $N\to\infty$ by Lemma \ref{lem: generalized: centering}.
		Again by Lemma \ref{lem: generalized: centering}, taking a logarithm and dividing by $Nh_N$ in both sides of Theorem \ref{fluctuation hat Z}, we have that for $0<\beta<\beta_\alpha$, \[\frac{1}{N h_N}\ln \widehat Z_N\stackrel{p}{\to}0.\] From the previous displays,  \begin{equation}\label{eq: generalized: alpha=1: Z}
			\frac{1}{N h_N}\ln Z_N\stackrel{p}{\to} \frac{\beta}{2}
		\end{equation}  as $N\to\infty$  for all $0<\beta<\beta_\alpha$.
		Similar to \eqref{eq: alpha=1: GSE and free energy}, we have \begin{equation} \label{eq: generalized: alpha=1: GSE and free energy}
			\frac{\beta}{N h_N}\max_{\sigma} \sum_{i<j} J_{ij} \sigma_i \sigma_j \le \frac{1}{N h_N}\ln Z_N \le \frac{\ln 2}{h_N} + \frac{\beta}{N h_N}\max_{\sigma} \sum_{i<j} J_{ij} \sigma_i \sigma_j.
		\end{equation} 
		Noticing that $h_N\to\infty$ by Lemma \ref{lem: generalized: centering}, Equations \eqref{eq: generalized: alpha=1: Z} and \eqref{eq: generalized: alpha=1: GSE and free energy} imply \begin{equation}\label{eq: generalized: alpha=1: GSE limit}
			\frac{1}{N h_N}\max_{\sigma} \sum_{i<j} J_{ij} \sigma_i \sigma_j \stackrel{p}{\to} \frac{1}{2}
		\end{equation} as $N\to\infty$ whenever $0<\beta<\beta_\alpha.$ 
		Observe that the quantities in \eqref{eq: generalized: alpha=1: GSE limit} are independent of $\beta$, so it holds for all $\beta>0$. 
		In view of \eqref{eq: generalized: alpha=1: GSE and free energy}, it follows that \eqref{eq: generalized: alpha=1: Z} also holds for any $\beta>0$, which finishes the proof.
	\end{proof}

\subsection{Theorem \ref{thm:free_energy_below_one}}


Our proof relies on the following generalization of Lemma \ref{lem:heavy_tail_conv_representation}. Let $G(t) = \mathbb{P}(|  X| > t)$ for $t > 0$ and define its inverse 
\[ G^{-1}(u) = \inf \{ y> 0: G(y) \le u \},  \ u \in (0, 1). \] We can morally think of $G^{-1}$ as $ x^{-1/\alpha}$ for $x>0$. Let $(  X_i)_{i\geq 1}$ be i.i.d. copies of $  X$ and let $  Y_i=| {X}_i|$ for $i\geq 1.$

\begin{lemma} \label{lem:generalized heavy tail conv representation}
	The following statements hold:
	\begin{enumerate}
		\item[(a)]  Let $0< \alpha < 2$. For $n\geq 1,$
		\begin{equation*} 
			a_n^{-1} \bigl(  Y_{I_1},   Y_{I_2}, \ldots,   Y_{I_n}\bigr) \stackrel{d}{=} a_n^{-1} \bigl(  G^{-1} (\gamma_1/\gamma_{n+1}), G^{-1} (\gamma_2/\gamma_{n+1}), \ldots, G^{-1} (\gamma_n/\gamma_{n+1}) \bigr).
		\end{equation*}
		Also, the random vector $a_n^{-1}(  X_1,   X_2, \ldots,   X_n) $ has the same distribution as that of 
		\begin{equation*} 
			a_n^{-1} \bigl(  \eps_1 G^{-1} (\gamma_{r(1)}/\gamma_{n+1}), \eps_2 G^{-1} (\gamma_{r(2)}/\gamma_{n+1}), \ldots, \eps_n G^{-1} (\gamma_{r(n)}/\gamma_{n+1}) \bigr),  
		\end{equation*}
		where $\eps_1, \eps_2, \ldots, \eps_n$ are i.i.d.\ Rademacher variables and $r:[n] \to [n]$ is an independent uniform random permutation, both independent of $(\gamma_j)_{j \ge 1}$. 
		\item[(b)] Let $0<\alpha<2$. For any $\eta> 0$, with probability tending to one, for all $1 \le j \le n/\log n,$ we have 
		\[ (1-\eta) \gamma_j^{-(1/\alpha + \eta) } \le a_n^{-1} G^{-1}(\gamma_j/\gamma_{n+1}) \le (1+ \eta) \gamma_j^{-(1/\alpha - \eta) },  \]
		and for all $n/ \log n \le j \le n$, 
		\[  a_n^{-1} G^{-1}(\gamma_j/\gamma_{n+1}) \le (1+ \eta) n^{-(1/\alpha - \eta) }.  \]
		
		\item[(c)]   Let $0< \alpha < 2$. As $n \to \infty$, 
		\[ a_n^{-1} \bigl(  Y_{I_1},   Y_{I_2}, \ldots,   Y_{I_n}, 0, 0, \ldots\bigr) \stackrel{d}{\to} \bigl( \gamma_1^{-1/\alpha}, \gamma_2^{-1/\alpha}, \ldots \bigr) \text{ on } \mathcal{A}_{\ge}.\]
		
		\item[(d)]  Assume that  $0 < \alpha < 1$. As $n \to \infty$, 
		\[ a_n^{-1} \sum_{j=1}^n   Y_{j} \stackrel{d}{\to} \sum_{j=1}^\infty \gamma_j^{-1/\alpha},  \]
		where the infinite sum is finite a.s..
		Also, for any fixed integer $R \ge 1$, as $n\to\infty,$
		\[ a_n^{-1}\Bigl ( \sum_{j=1}^R   Y_{I_j}, \sum_{j=R+1}^n   Y_{I_j} \Bigr)  \stackrel{d}{\to} \Bigl( \sum_{j=1}^R \gamma_j^{-1/\alpha}, \sum_{j=R+1}^\infty \gamma_j^{-1/\alpha}\Bigr) .  \]
	\end{enumerate}
\end{lemma}

\begin{proof}
	The proof of Lemma \ref{lem:generalized heavy tail conv representation}(a) can be traced back to \cite{bordenave2011spectrum}. Here we only give a sketch of the proof of part (b), and parts (c) and (d) will follow.  Define $V=1/G$ which is non-decreasing, and let $V^{-1}(t)= \inf\{x>0: 1/ G(x)\ge t \}=G^{-1}(1/t)$ for $t>0$. Since $G$ is regularly varying with index $-\alpha$, we have that $V$ is regularly varying with index $\alpha$. From Proposition 2.6 of \cite{resnick}, $V^{-1}$ is regularly varying with index $1/\alpha$. Note that $a_n=V^{-1}(n)$ for any $n\in \mathbb N$. Then, we have \begin{equation}
		a_n^{-1}G^{-1}(\gamma_j/\gamma_{n+1})= \frac{V^{-1}(\gamma_{n+1}/\gamma_j)}{V^{-1}(n)}. \label{eq: V inverse}
	\end{equation} For $1\le j\le n/\log n$, we may use \eqref{potter} for the function $V^{-1}$ and the fact that $\lim_{n\to\infty}\gamma_{n+1}/n = 1$ with probability one to conclude the first half of the result from \eqref{eq: V inverse}. For $n/\log n \le j \le n$, the other half of the result follows from the monotonicity $a_n^{-1}G^{-1}(\gamma_j/\gamma_{n+1})\le a_n^{-1}G^{-1}(\gamma_{\lfloor n/\log n \rfloor }/\gamma_{n+1}),$ and then applying \eqref{potter}. 
	
\end{proof}

\begin{proof}[\bf Proof of Theorem \ref{thm:free_energy_below_one}] Using Lemma \ref{lem:generalized heavy tail conv representation}(d), we have the upper bound \[ \frac{1}{b_N} \ln Z_N \le\frac{ 1}{a_N^{-1}a_{{N\choose 2}}} \Big( N \ln 2  + \beta\sum_{i < j} |  J_{ij}| \Big)  \stackrel{d}{\to} \beta  \sum_{k=1}^\infty \gamma_k^{-1/\alpha}. \]
	The lower bound can similarly be derived.
\end{proof}

\subsection{Theorem \ref{thm:Gibbs_structure}}
Due to the presence of possible atoms, we can have ties in the order statistics of $  Y_i$'s. This necessitates that we redefine the ranks carefully. In this case, we can pick a uniform decreasing ordering $\sigma$ maintaining $  Y_{\sigma(1)}\ge \dots \ge   Y_{\sigma(n)}$ and define the rank of $(i,j)$ as $r(i,j)=\sigma^{-1}(i,j)$. Note that this definition retains the key property that the resulting rank function still forms a uniform bijection.

The exact distribution of the disorder is irrelevant to the proof of Lemma \ref{lem:Gibbs_structure}(a). However, in view of Lemma \ref{lem:generalized heavy tail conv representation}, recalling $n=N(N-1)/2,$ the inequality in Lemma \ref{lem:Gibbs_structure}(b) should be modified to \[
\sum_{j: r(i,j) \le R} |  J_{ij}|  -  \sum_{j: r(i,j) > R} |  J_{ij}| \ge c a_N^{-1}a_n R^{-(1/\alpha +\eta)},
\] where $c$ is a small absolute constant and $0<\eta<\min\{3\delta/(2\alpha(1-\delta/2), \delta/(\alpha(2-\alpha-\delta)), 1/\alpha\}$. In fact, we can write \[|  J_{ij}|= a_N^{-1}G^{-1}(\gamma_{r(i,j)}/\gamma_{n+1}), \ 1\le i\neq j\le N,\] and \eqref{eq:row_lb1} becomes \[a_N^{-1}a_n\bigl(0.9(1-\eta)R^{-1/\alpha-\eta} -\ln N\cdot 1.1(1+\eta)T^{-1/\alpha+\eta}-N\cdot 1.1(1+\eta)N^{-1/\alpha+\eta} \bigr).\]
Then, the final bound in \eqref{eq:bound using R} is \[2 \exp\bigl(-  2c  \beta a_N^{-1}a_n N^{-(1/\alpha+\eta)/2}\bigr),\] which converges to $0$ as $N\to\infty$.

\end{document}